\documentclass[11pt,a4paper,twoside]{article}

\usepackage[english]{babel}
\usepackage[T1]{fontenc}
\usepackage{amsmath}
\usepackage{amsfonts}
\usepackage{amssymb}
\usepackage{enumerate}
\usepackage{mathrsfs}
\usepackage{dsfont}
\usepackage{constants}
\usepackage{theorem}
\usepackage{times}

\newcommand{\ba}{{\bf a}}
\newcommand{\bd}{{\bf d}}
\newcommand{\bm}{{\bf m}}
\newcommand{\bM}{{\bf M}}
\newcommand{\E}{{\bf E}}
\renewcommand{\P}{{\bf P}}
\newcommand{\ddr}{\mathrm{d}}
\newcommand{\R}{\mathbb{R}}
\newcommand{\N}{\mathbb{N}}
\newcommand{\un}{\mathbf{1}}

\newcommand{\Dpsi}{\psi^{\prime}}
\newcommand{\Dpsistar}{\psi^{*\prime}}
\newcommand{\invpsi}{\psi^{-1}}
\newcommand{\invphi}{\varphi^{-1}}
\newcommand{\invDpsi}{\psi^{\prime-1}}

\newcommand{\Nstar}{\mathcal{N}^*}
\newcommand{\Ma}{\mathcal{M}_a^*}

\newcommand{\cD}{\mathcal{D}}
\newcommand{\cI}{\mathcal{I}}
\newcommand{\cJ}{\mathcal{J}}
\newcommand{\cM}{\mathcal{M}}
\newcommand{\cN}{\mathcal{N}}
\newcommand{\cP}{\mathscr{P}}

\newcommand{\cF}{\mathcal{F}}
\newcommand{\cR}{\mathcal{R}}
\newcommand{\cT}{\mathcal{T}}
\newcommand{\cW}{\mathcal{W}}
\def\z{{\cal Z}}
\newcommand{\veps}{\varepsilon}

\def\w{{\rm w}}
\newcommand{\wh}{\widehat}

\newcommand{\bbP}{\mathbb{P}}
\newcommand{\bbE}{\mathbb{E}}
\newcommand{\bbD}{\mathbb{D}}
\newcommand{\bbZ}{\mathbb{Z}}

\def\cq{$\hfill \square$}
\def\cqfd{$\hfill \blacksquare$}

\newcommand{\ur}{u_{r}}

\newcommand{\urtild}{\tilde{u}_{r}}

\newcommand{\Bor}{B(0,r)}

\newcommand{\supp}{\mathtt{Supp}}
\def\ov{\overline}
\def\noi{\noindent}

\newcommand{\bdelta}{\boldsymbol{\delta}}
\newcommand{\bgamma}{\boldsymbol{\gamma}}
\newcommand{\boldeta}{\boldsymbol{\eta}}

\newtheorem{thm}{Theorem}[section]
\newtheorem{lem}[thm]{Lemma}
\newtheorem{prop}[thm]{Proposition}
\newtheorem{cor}[thm]{Corollary}

{\theorembodyfont{\rmfamily}}
{\theorembodyfont{\rmfamily}}
{\theorembodyfont{\rmfamily}\newtheorem{comment}{Comment}[section]}
{\theorembodyfont{\rmfamily}\newtheorem{rem}{Remark}[section]}

\newconstantfamily{c}{symbol=c}
\newconstantfamily{r}{symbol=r}

 \def\blemma{\begin{lem}}\def\elemma{\end{lem}}
 \def\bproposition{\begin{prop}}\def\eproposition{\end{prop}}
 \def\btheorem{\begin{thm}}\def\etheorem{\end{thm}}
 \def\bcorollary{\begin{cor}}\def\ecorollary{\end{cor}}
 \def\bremark{\begin{rem}}\def\eremark{\end{rem}}
 \def\bcondition{\begin{condition}}\def\econdition{\end{condition}}

 \def\benumerate{\begin{enumerate}}\def\eenumerate{\end{enumerate}}
 \def\bitemize{\begin{itemize}}\def\eitemize{\end{itemize}}

 \def\beqlb{\begin{eqnarray}}\def\eeqlb{\end{eqnarray}}
 \def\beqnn{\begin{eqnarray*}}\def\eeqnn{\end{eqnarray*}}

\title{{\bf \textsc{Exact packing measure of the range of\\
$\psi$-Super Brownian motions.}}}

\author{ Xan \textsc{Duhalde}\thanks{\textbf{Institution}: PRES Sorbonne Universit\'es, UPMC Universit\'e Paris 06,  LPMA (UMR 7599). \textbf{Postal address}: LPMA, Bo\^ite courrier 188, 4 place Jussieu, 75252 Paris Cedex 05, FRANCE. \textbf{Email}: xan.duhalde@upmc.fr} \and Thomas {\sc Duquesne}\thanks{\textbf{Institution}: PRES Sorbonne Universit\'es, UPMC Universit\'e Paris 06,  LPMA (UMR 7599). \textbf{Postal address}: Bo\^ite courrier 188, 4 place Jussieu, 75252 Paris Cedex 05, FRANCE. \textbf{Email}: thomas.duquesne@upmc.fr}}
\vspace{2mm}

\begin{document}
\maketitle
\begin{abstract}
We consider super processes whose spatial motion is the $d$-dimensional Brownian motion and whose branching mechanism $\psi$ is critical or subcritical; such processes are called
$\psi$-super Brownian motions.  
If $d\!>\!2\bgamma/(\bgamma\!-\!1)$, where $\bgamma\!\in\!(1,2]$ 
is the lower index of $\psi$ at $\infty$, then the total range of the $\psi$-super Brownian motion 
has an exact packing measure whose gauge function is 
$g(r)\! =\! (\log\log1/r) / \varphi^{-1} ( (1/r\log\log 1/r)^{2})$, 
where $\varphi\! =\! \psi^\prime\! \circ \! \psi^{\!-1}$. More precisely, we show that the occupation measure of the $\psi$-super Brownian motion is the $g$-packing measure restricted to its total range, up to a deterministic multiplicative constant only depending on $d$ and $\psi$. This generalizes the main result of \cite{Duq09} that treats the quadratic branching case. For a wide class of $\psi$, the constant $2\bgamma/(\bgamma\!-\!1)$ is shown to be equal to the packing dimension of the total range.

\smallskip


\smallskip

\end{abstract}

\section{Introduction.}

  The main result of this paper  provides an exact packing gauge function for the total range of 
super processes whose spatial motion is the $d$-dimensional Brownian motion and whose branching mechanism $\psi$ is critical or subcritical. We call such super processes \textit{$\psi$-super Brownian motions} (or \textit{$\psi$-SBM}, for short). 
This generalizes the main result of \cite{Duq09} that concerns the 
Dawson-Watanabe super process corresponding to the quadratic branching mechanism 
$\psi (\lambda)= \lambda^2$.

Before stating precisely our main results, let us briefly recall previous  works related to 
the fine geometric  properties of super processes. Most of these results concern the Dawson-Watanabe super process $(Z_t)_{t\geq 0}$. Dawson and Hochberg \cite{DaHo79} have proved that a.s.~for 
all $t\! >\! 0$, the Hausdorff dimension of the topological support of $Z_t$ is equal to $2\wedge d$.   
In \cite{DP91}, Dawson and Perkins prove that in supercritical dimensions $d\! \geq 3$, 
the Dawson-Watanabe super process $Z_t$ is a.s.~equal to the $h_1$-Hausdorff measure restricted to the topological support of $Z_t$, where $h_1(r)\! =\!  r^2 \log \log 1/r$ (see also Perkins \cite{Per88, Per89} for a close result holding a.s.~for all times $t$). By use of the Brownian snake, Le Gall and Perkins \cite{LGPe95} prove a similar result 
in the critical dimension $d\! =\! 2$ with the gauge function $h_2 (r) \! =\! r^2\log 1/r\log\log\log 1/r$. 
In \cite{LGPT}, Le Gall, Perkins and Taylor have proved that in dimension $d\! \geq \! 3 $, the topological support of $Z_t$ has no exact packing measure.

Dawson and Hochberg in \cite{DaHo79} also proved that the total range of the Dawson Watanabe super process is a.s.~equal to $4 \wedge d$. 
In \cite{DIP}, Dawson Iscoe and Perkins investigate the fine geometric properties of the total occupation measure ${\bf M}\! = \! \int_0^\infty \! Z_t \, \ddr t$ of the Dawson-Watanabe super process: they prove that in supercritical dimensions $d\! \geq \! 5 $, ${\bf M}$ is a.s.~equal to the $h_3$-Hausdorff measure restricted on the total range of the super process, where 
$h_3(r)\! = \! r^4 \log \log 1/r$. In \cite{LG99}, Le Gall considers the critical dimension $d\! = \! 4$ and he proves a similar result with respect to the gauge function $h_4(r)\! = \!  r^4\log 1/r\log\log\log 1/r$. In \cite{Duq09}, the occupation measure ${\bf M}$ is also shown to coincide a.s.~with the $g_1$-packing measure restricted to the total range of the super process, where $g_1(r)\! =\!  r^4(\log\log 1/r)^{-3}$.

For super Brownian motions whose branching mechanism is general, 
less results are available: in \cite{Del99}, Delmas computes the Hausdorff dimension of super Brownian motions whose branching mechanism is stable; this result is eventually extended in \cite{DuLG05} to general branching mechanism $\psi$ 
thanks to geometric considerations on \textit{$\psi$-Lévy trees}. The $\psi$-Lévy trees are the actual genealogical structures of the $\psi$-SBM; they are compact random real trees coded by the height process (introduced by Le Gall and Le Jan \cite{LGLJ1} and further studied in \cite{DuLG}) and they appear as the scaling limits of Galton-Watson trees; their geometric properties are discussed in \cite{DuLG05, DuLG06, Duq10, Duq12}. In particular, it is proved in \cite{Duq12} that L\'evy trees have an exact packing measure, which is closely related to the main result of our article.

\bigskip

Let us introduce precisely the main results of our paper. We first fix a \textit{branching mechanism} $\psi$ that is critical or subcritical: namely, $\psi: \R_+ \rightarrow \R_+$ is the Laplace exponent of a spectrally positive L\'evy process that is of the following L\'evy-Khintchine form:  
\begin{equation}
\label{LevyKhin}
\forall \lambda \in \R_+, \quad  \psi (\lambda )= \alpha \lambda + \beta \lambda^2 + 
\int_{(0, \infty)}\!\!\!\!\! \!\!  (e^{-\lambda r} -1+\lambda r)\,\pi (\ddr r)  \; , \end{equation}
where $\alpha, \beta  \! \in \! \R_+$, and $\pi$ is the \textit{L\'evy measure} that satisfies $\int_{{(0, \infty)}} (r\wedge r^2)\,\pi (\ddr r) < \infty$. The branching mechanism $\psi$ is the main parameter that governs the law of the processes that are considered in this paper. We introduce two 
exponents that compare $\psi$ with power functions at infinity: 
\begin{equation}\label{defexpointro}
\bgamma = \sup \big\{ c \in \R_+ :  \lim_{^{\lambda \rightarrow \infty}} \psi (\lambda) \lambda^{-c} = \infty  \big\} \, , \quad \boldeta = \inf\{ c \in \R_+ : \lim_{^{\lambda \rightarrow \infty}} \psi (\lambda) \lambda^{-c} = 0 \big\} .
\end{equation}
The \textit{lower index} $\bgamma$ and the \textit{upper index} $\boldeta$ have been introduced by Blumenthal and Getoor \cite{BluGe61}: they appear in the fractal dimensions and the regularity of the processes that we consider. The statements of the paper also involve a third exponent:  
\begin{equation}\label{defdeltaintro}
\bdelta= \sup  \big\{ c\! \in \! \R_+  :  \exists \, C \!\in \!(0, \infty)\; \textrm{such that} \; C \psi (\mu) \mu^{-c} \!  \leq \!  \psi (\lambda) \lambda^{-c}  \, , \, 1 \! \leq \! \mu \! \leq \! \lambda  \big\} 
\end{equation} 
that has been introduced in \cite{Duq12}. It is easy to check that $1\! \leq \!  \bdelta \! \leq \!  \bgamma \! \leq \!  \boldeta \! \leq \!  2$. If $\psi$ is regularly varying at $\infty$, all these exponents coincide. In general, they are however distinct and we mention that there exist branching mechanisms $\psi$ of the form (\ref{LevyKhin}) such that $\bdelta\! = \! 1 \! <\!  \bgamma \! = \! \boldeta$ (see Lemma 2.3 and Lemma 2.4 in \cite{Duq12} for more details). In our paper we shall often assume that $ \bdelta \! > \! 1 $ which is a mild regularity assumption on $\psi$ (see Comment \ref{opti} below).

The space $\R^d$ stands for the usual $d$-dimensional Euclidian space. We denote by $M_f(\R^d)$ the space of finite Borel measures equipped with the topology of weak convergence. For all $\mu \! \in \! M_f(\R^d)$ and for all Borel 
measurable functions $f: \R^d \rightarrow \R_+$, we use the following notation: 
$$ \langle f , \mu \rangle = \int_{\R^d} \!\!\! f(x) \,  \mu (\ddr x) \quad \textrm{and} \quad \langle \mu \rangle  = \mu \big(\R^d \big) \; .  $$
Then, $\langle \mu \rangle $ is the \textit{total mass} of $\mu$. We shall also denote by $\supp (\mu)$ the \textit{topological support of $\mu$} that is the smallest closed subset supporting $\mu$.

Unless the contrary is explicitly mentioned, all the random variables that we consider are defined on the same measurable space $(\Omega,\mathcal{F})$. We first introduce a $\R^d$-valued continuous process $\xi=(\xi_t)_{t\geq 0}$; for all $x\! \in \R^d$, we let $\P_x$ 
be a probability measure on $(\Omega, \cF)$ such that $\xi$ under $\P_x$ 
is distributed as a standard $d$-dimensional Brownian motion starting from $x$. 
We also introduce $Z\! = \! (Z_t)_{t\in \R_+}$ that is a $M_f(\R^d)$-valued c\`agl\`ad process defined on $(\Omega, \cF)$, and for all $\mu \! \in \! M_f(\R^d)$, we let $\bbP_\mu$ be a probability measure on $(\Omega, \cF)$ such that $Z$ under $\bbP_\mu$ is distributed as \textit{super Brownian motion with branching mechanism $\psi$}. Namely, under $\bbP_\mu$, $Z$ is a Markov process whose transitions are characterized as follows: for all bounded Borel measurable functions $f: \R^d \rightarrow \R_+$ and for all $s, t\! \in \! \R_+ $, 
\begin{equation}
\label{transition}
\textrm{$\bbP_\mu$-a.s.} \quad \bbE_\mu  \big[ \exp ( - \langle Z_{t+s}, f \rangle ) \, \big| \, Z_s \big]=\exp ( - \langle Z_s , v_t \rangle ),
\end{equation}
where the function $(v_t(x))_{t\in \R_+,x\in \R^d}$ is the unique nonnegative
solution of the integral equation
$$v_t(x)+  \E_x \Big[ \int_0^t \! \! \psi \big( v_{t-s}(\xi_s) \big) \, \ddr s \Big]  \, = \E_x\left[f(\xi_t)\right] \; , \quad x \in \R^d , \; t \in [0, \infty).$$
Dawson-Watanabe super processes correspond, up to scaling in time and space, to the branching mechanism $\psi(\lambda)= \lambda^2$. Super diffusions with general branching mechanisms of the form (\ref{LevyKhin}) have been introduced by Dynkin \cite{Dyn91bis}; for a detailed account on super processes,  we refer to the books of Dynkin \cite{Dynbook1, Dynbook2}, Le Gall \cite{LGZurich}, Perkins \cite{PerSF}, Etheridge \cite{Eth00} and Li \cite{Libook}. 

  We easily check that, under $\bbP_\mu$, 
the process $(\langle Z_t \rangle)_{t\geq 0}$ of the total mass of the $\psi$-SBM is a \textit{continuous states branching process} with branching mechanism $\psi$. Continuous states branching processes have been introduced by Jirina \cite{Ji58} and Lamperti \cite{Lamperti1, Lamperti2}, and further studied by Bingham \cite{Bi2}. 
The assumption $\bdelta \! >\! 1$, implies $\bgamma \! >\! 1$, which easily entails 
$\int^\infty \ddr \lambda / \psi (\lambda) \! <\! \infty $ that is called the \textit{Grey condition}. Under this condition, standard results on continuous states branching processes (see Bingham \cite{Bi2}) 
imply that $\langle Z \rangle$ is absorbed in $0$ in finite time: namely,   
$\bbP_\mu (\exists t \in \R_+ : Z_t = 0 )= 1$. Thus the following definition makes sense: 
\begin{equation}
\label{defboldM}
\bM= \int_0^\infty \!\!\! Z_t \,  \ddr t \; 
\end{equation}
and ${\bf M}$ is therefore a random finite Borel measure on $\R^d$: it is the \textit{occupation} measure of the $\psi$-SBM $Z$. We also define the \textit{total range} of $Z$ by 
\begin{equation}
\label{defboldR}
 {\bf R} = \bigcup_{\varepsilon >0} \overline{ \bigcup_{ t \geq \varepsilon }  \supp (Z_t) }  \; , 
\end{equation}
where for any subset $B$ in $\R^d$, $\overline{B}$ stands for its closure. We recall here a result due to Sheu \cite{She94} that gives a condition on $\psi$ for ${\bf R}$ to be bounded: 
\begin{equation}
\label{Sheucond} 
\textrm{$\bbP_\mu$-a.s.~${\bf R}$ is bounded} \quad \Longleftarrow \!\! =\!\! = \!\! \Longrightarrow\quad 
\int_{1}^\infty \!\! \!\! \frac{\ddr b}{\sqrt{\int_1^b \psi (a) \ddr a }} < \infty  \;\;  \textrm{and $\supp (\mu)$ is compact.}
 \end{equation}
See also Hesse and Kyprianou \cite{HeKy14} for a simple probabilistic proof. Note that if $\bgamma \! >\! 1$, then (\ref{Sheucond}) holds true.

We next denote by $\dim_H$ and $\dim_p$ respectively the \textit{Hausdorff and the packing dimensions} on $\R^d$. We also denote by $\underline{\dim}$ and $\overline{\dim}$ the \textit{lower and the upper box dimensions}. We refer to Falconer  \cite{Falbook} for precise definitions. We next 
recall Theorem 6.3 \cite{DuLG05} that asserts that 
for all $\mu \! \in \! M_f(\R^d)$ distinct from the null measure, the following holds true.   
\begin{equation}\label{dimHRintro}
\textrm{If $\bgamma>1$, then $\bbP_\mu$-a.s.}\qquad 
\dim_H\left({\bf R}\right)=d\wedge\frac{2\boldeta}{\boldeta-1}.
\end{equation}
If furthermore $\supp (\mu)$ is compact, then ${\bf R}$ is bounded (by (\ref{Sheucond})) and Theorem 6.3 \cite{DuLG05} also asserts that $\bbP_\mu$-a.s.~$\underline{\dim}\left({\bf R}\right)=d\wedge\frac{2\boldeta}{\boldeta-1}$. As already mentioned (\ref{dimHRintro}) generalizes the work of Delmas  \cite{Del99} that treats SBMs whose branching mechanism is stable. Note that Assumption $\bgamma \! >\! 1$ is not completely satisfactory for $\dim_H\left({\bf R}\right)$ only depends on $d$ and $\boldeta$ (see Proposition 5.7 \cite{DuLG05} and the discussion in Section 5.3 of this article). 
The first result of our paper computes the packing dimension of ${\bf R}$ under more restrictive assumptions.   
\begin{thm}
\label{thdimension}
Let $\mu \! \in \! M_f(\R^d)$ be distinct from the null measure. 
Let $\psi$ be of the form (\ref{LevyKhin}). Let ${\bf R}$ be the total range of the $\psi$-SBM with initial measure $\mu$, as defined in (\ref{defboldR}). 
Assume that $\bdelta \! >\! 1$ and that $d \! > \! \frac{2\bdelta}{\bdelta-1}$. Then, 
\begin{equation}\label{dimPRintro}
\textrm{$\bbP_\mu$-a.s.} \quad \dim_p(\bf{R})=\frac{2 \bgamma}{\bgamma-1}.
\end{equation}
If furthermore $\supp (\mu)$ is compact, then $\bbP_\mu$-a.s.~$\overline{\dim}(\bf{R}) =\frac{2 \bgamma}{\bgamma-1}$.  
\end{thm}
\begin{comment}
\label{deuxfois} Theorem 5.5 \cite{DuLG05} shows that if $\bgamma \! >\! 1$, 
the Hausdorff and the packing dimensions of the $\psi$-Lévy tree are resp.~$\frac{\boldeta}{\boldeta-1}$ and $\frac{\bgamma}{\bgamma-1}$. Thus, for all sufficiently large $d$,  
(\ref{dimHRintro}) and (\ref{dimPRintro}) imply that the Hausdorff and the packing dimensions of ${\bf R}$ are \textit{twice} that of the $\psi$-L\'evy tree. This can be informally explained by the fact that ${\bf R}$ appears as the range of the $\psi$-L\'evy snake that can be viewed as a Gaussian process indexed by the $\psi$-Lévy tree that is $(\frac{1}{2}\! -\! \veps)$-Hölder regular for any $\veps \! \in \! (0, 1/2)$ (see Lemma 6.4 \cite{DuLG05} and see also (\ref{covar}) in Section \ref{levysnake} of the present paper for more details).  
\cq 
\end{comment}
\begin{comment}
\label{casOK} The above mentioned 
Hölder-regularity of the L\'evy snake studied in \cite{DuLG05} also entails that if $\supp (\mu)$ is compact and if $\bgamma \! >\! 1$, then $\overline{\dim} ({\bf R}) \leq d\wedge\frac{2\bgamma}{\bgamma-1}$. Since $\dim_p(A) \! \leq \! \overline{\dim} (A)$ for any bounded $A\! \subset \! \R^d$ (see e.g.~\cite{Falbook}), an easy argument entails that for all non-null finite measure $\mu$, if $\bgamma \! >\! 1$, then $\bbP_\mu$-a.s.~$\dim_p\left({\bf R}\right)\leq d\wedge\frac{2\bgamma}{\bgamma-1}$. This, combined with (\ref{dimHRintro}) entails that 
\begin{equation}
\label{gametaeqOK}
 \textrm{if $\boldeta= \bgamma >1$, then $\bbP_\mu$-a.s.} \quad \dim_H ({\bf R})=\dim_p ({\bf R})= d\wedge \frac{2\bgamma }{\bgamma-1} \;  \; ,
\end{equation}
with the same equality for the lower and the upper box dimensions if $\supp (\mu)$ is compact. The equality $\bgamma\! =\! \boldeta$ holds true for instance if $\psi$ is regularly varying at $\infty$. 
Therefore, the novelty of Theorem \ref{thdimension} only concerns the cases where $\bgamma \! \neq \! \boldeta$. \cq 
\end{comment}
\begin{comment}
\label{notoptim} In Theorem \ref{thdimension}, note that Assumption $\bdelta \! >\! 1$ is not optimal since the value of $\dim_p\left({\bf R}\right)$ only depends on $\bgamma$.  Our arguments fail to prove (\ref{dimPRintro}) when $d\! \in \! (\frac{2\bgamma}{\bgamma-1}, \frac{2\bdelta}{\bdelta-1})$. We 
conjecture that \textit{if $\bgamma \! >\! 1$, then $\bbP_\mu$-a.s.~$\dim_p ({\bf R})= d\wedge \frac{2\bgamma }{\bgamma-1}$}. \cq  
\end{comment}

Let us set $\varphi= \psi^\prime \circ \psi^{-1}$. The main properties of that increasing function are stated in Section \ref{sectionexpo}. Here, we just notice that the reciprocal function of $\varphi$, that is denoted by $\varphi^{-1} $, is defined from $[\alpha, \infty)$ to $[0, \infty)$. Then, we set
\begin{equation} 
\label{gaugedef}    
 g (r) = \frac{\log \log\frac{1}{r}}{\varphi^{-1} \!\! \left( (\frac{1}{r} \log \log \frac{1}{r})^2 \right)} \; , \; r\in (0, r_0) 
\end{equation}
where $r_0\! = \! \min (\alpha^{-1/2} , e^{-e})$, with the convention $\alpha^{-1/2}\! =\!  \infty$ if 
$\alpha\! = \! 0$. We check (see Section \ref{sectionexpo}) that $g$ is a continuous increasing function 
such that $\lim_{0+} g= 0$. 

 We next denote by $\cP_g$ the $g$-packing measure on $\R^d$, whose definition is recalled in Section 
\ref{sectionpacking}. The following theorem is the main result of the paper.
\begin{thm}
\label{mainth} Let $\mu \! \in \! M_f(\R^d)$ be distinct from the null measure. Let 
$\psi$ be of the form (\ref{LevyKhin}).  Let $Z$ be a $\psi$-SBM starting from $\mu$; let 
${\bf R}$ be its total range, as defined by (\ref{defboldR}), and let ${\bf M}$ be its occupation measure, as defined by (\ref{defboldM}). 
Let $g$ be defined by (\ref{gaugedef}). Assume that 
$$\bdelta >1\quad \textrm{and} \quad d  >  \frac{2\bgamma}{\bgamma-1}\; .$$ 
Then, there exists a positive constant $\kappa_{d,\psi} $ that only depends on $d$ and $\psi$ such that 
\begin{equation}
\label{paacc}
\textrm{$\bbP_\mu$-a.s.} \qquad \bM = \kappa_{d,\psi}\,   \cP_g ( \, \cdot  \cap {\bf R} ) \; .
\end{equation}
\end{thm}
\begin{comment}
\label{opti} Unlike Theorem \ref{thdimension}, we think that 
the assumptions of Theorem \ref{mainth} are optimal in the following sense. 
Indeed, since $d\wedge \frac{_{2\bgamma}}{^{\bgamma -1}}$ is thought to be 
the packing dimension of ${\bf R}$, (\ref{paacc}) probably 
does not hold true when $d\leq \frac{_{2\bgamma}}{^{\bgamma -1}}$. 

Moreover, arguing as in Lemma 2.3  \cite{Duq12}, we easily check that $\bdelta \! >\! 1$ if and only if $g$ satisfies a \textit{doubling condition} (see (\ref{doubling}) Section \ref{sectionpacking}). Although, it is possible to define packing measures with respect to irregular gauge functions (see Edgar \cite{Edg07}), the doubling condition is the minimal assumption 
on a gauge function which implies that the corresponding packing measure has nice properties 
(regularity, comparison lemmas). In this sense, Assumption $\bdelta \! >\! 1$ is required to stay within the framework of standard geometric measures.  \cq 
\end{comment}

\begin{comment}
\label{stableetal}
In the stable cases where $\psi(\lambda)\! =\! \lambda^\gamma$ with $\gamma \! \in\! (1,2]$, then 
$$\forall r \in (0, \infty), \quad g(r)=r^{\frac{2\gamma}{\gamma-1}} \big(\log\log1/r \big)^{-\frac{\gamma+1}{\gamma-1}}\; . $$
If $\gamma\! = \! 2$, then $g(r)= r^4 (\log \log 1/r)^{-3}$ and we recover the result from \cite{Duq09}. Moreover, note that in the stable cases, Theorem \ref{thdimension} is a straightforward consequence of Theorem \ref{mainth}. However, when 
$\bgamma\neq \boldeta$, it turns out that Theorem \ref{thdimension} cannot be simply derived from Theorem \ref{mainth}: indeed, one important ingredient of the proof of Theorem \ref{mainth} consists in computing the lower local density of ${\bf M}$ with respect to $g$; this lower limit is achieved 
along a deterministic sequence of radii whose images by $g$ are hard to compare with a power function. \cq 
\end{comment}

The paper is organised as follows. In Section \ref{secnotations}, we recall definitions and basic results.  Section \ref{sectionpacking} is devoted to packing measures and to two 
comparison lemmas that are standard technical tools used to compute packing measures. 
Section \ref{sectionexpo} gather elementary facts on the power exponents $\bdelta, \bgamma$ and $\boldeta$ associated with the branching mechanism $\psi$. 
In Section \ref{sectionheight} and in Section \ref{levysnake}, we recall the 
definitions of -- and various results on -- the 
$\psi$-height process, the corresponding $\psi$-Lévy tree and  the associated $\psi$-Lévy snake. In Section \ref{secestimates}, we prove estimates on a specific subordinator (Sections \ref{secsubord} and Section \ref{sectionsubord}) and on functionals of the snake involving the hitting time of a given ball (Section \ref{sectionhitting} and Section \ref{sectionbadpts}). Section \ref{sectionproofs} is devoted to the proof of the two main theorems:  
we prove Theorem \ref{mainth}  first and Theorem \ref{thdimension} next.

\section{Notations, definitions and preliminary results.}\label{secnotations}
\subsection{Packing measures.}
\label{sectionpacking}
In this section, we briefly recall basic results on packing measures on the Euclidian space $\R^d$ that have been introduced by Taylor and Tricot in \cite{TaTr}.

  A \textit{gauge function} is an increasing continuous function $g: \! (0, r_0)\!  \rightarrow \! (0, \infty)$, where $r_0 \! \in \! (0, \infty)$, such that $\lim_{0+} g \! =\!  0$ and that 
satisfies a \textit{doubling condition}: namely, there exists $C \! \in \! (0, \infty)$ such that 
\begin{equation}
\label{doubling}
\forall r \in (0, r_0/2), \quad g(2r) \leq Cg(r) \; .
\end{equation}   
Let $B \! \subset \! \R^d$ and let $\varepsilon \!  \in \! (0, \infty)$. Recall that a (closed) \textit{$\varepsilon$-packing of $B$} is a finite 
collection of pairwise disjoint closed ball $(\overline{B}(x_m, r_m))_{1\leq m \leq n}$ whose centers $x_m$ belong to $B$ and whose radii $r_m$ are not greater than $\varepsilon$.  We then set 
\begin{equation}
\label{preprepac}
 \cP^{(\varepsilon)}_g (B) = \sup \Big\{\! \! \!\!  \sum_{\; \; \; 1\leq m \leq n}\!\!\!\!  g(r_m) \; ;  \; \left( \overline{B}(x_m, r_m)\right)_{1\leq m \leq n}  ,   \; \textrm{$\varepsilon$-packing of $B$}  \Big\}   
 \end{equation}
and 
\begin{equation}\label{packingpremeadef}
\cP^*_g (B)= \lim_{\varepsilon \rightarrow 0+}  \cP^{(\varepsilon)}_g (B) \; \in  \; [0, \infty] \; , 
\end{equation}
that is called the \textit{$g$-packing pre-measure of $B$}. 
The \textit{$g$-packing measure of $B$} is then given by 
\begin{equation}
\cP_g (B)= \inf \Big\{  \sum_{n \geq 0} \cP^*_g (B_n) \; ; \; B \! \subset \! \bigcup_{n \geq 0} B_n \;  \Big\} \; .
\end{equation}

\begin{rem}
\label{slightdif}
The definition (\ref{preprepac}) of $\cP^{(\varepsilon)}_g$ that we adopt here is slightly different from the one introduced by Taylor and Tricot \cite{TaTr} who take the infimum of $\sum_{m=1}^n g(2r_m) $ over $\varepsilon$-packings with open balls. However, since $g$ is increasing, continuous and satisfies a doubling condition (\ref{doubling}), the resulting measure is quite close to Taylor and Tricot's definition: the difference is irrelevant to our purpose and their main results on packing measures immediately apply to the $g$-packing measures defined by (\ref{packingpremeadef}). \cq 
\end{rem}

Next recall from Lemma 5.1 \cite{TaTr} that $\cP_g$ is a Borel-regular outer measure. Moreover, it is obvious from the definition that for any subset $B\subset \R^d$, 
\begin{equation}
\label{premeabound}
\cP_g (B) \leq \cP^*_g (B) \; . 
\end{equation}
Next, if $B$ is a $\cP_g$-measurable set such that $\cP_g (B) \! <\!  \infty$, then for any $\varepsilon \! >\! 0$, 
there exists a closed subset $F_\varepsilon \! \subset \! B$ such that 
\begin{equation}
\label{innerreg}
\cP_g (B) \leq \cP_g (F_\varepsilon) + \varepsilon \; .
\end{equation}
We recall here Theorem 5.4 \cite{TaTr} that is a standard comparison result for packing measures.
\begin{thm}[Theorem 5.4 \cite{TaTr}]
\label{TaTrcomparesu} Let $\mu$ be a finite Borel measure on $\R^d$. Let $B$ be a Borel subset of $\R^d$. Let $g$ be a gauge function satisfying a doubling condition (\ref{doubling}) with a constant $C\! >\! 0$. Then, the following holds true. 
\begin{itemize} 
\item[(i)]  If $\,  \liminf_{r \rightarrow 0} \frac{\mu (B(x, r))}{g(r )} >1\, $ for any $x \! \in\! B$, then $\mu (B) \geq \cP_g (B)$. 
\item[(ii)] If $\, \liminf_{r \rightarrow 0} \frac{\mu (B(x, r))}{g(r)} < 1\,  $ for any $x \! \in \! B$, then 
$ \mu(B) \leq C^2 \cP_g (B)  $. 
\end{itemize}
\end{thm}
\noi
We also recall the following specific result due to Edgar in \cite{Edg07}, Corollary 5.10. 
\begin{lem}[Corollary 5.10 \cite{Edg07}]
\label{equalitycomp}Let $\mu$ be a finite Borel measure on $\R^d$. Let $\kappa \!  \in \! (0, \infty)$ and let $B$ be a Borel subset of $\R^d$ such that 
$$ \forall x \in B  , \quad \liminf_{r \rightarrow 0+} \frac{\mu (B(x,r))}{g(r)} = \kappa  \; .$$
Then, $\mu (B)= \kappa\cP_g (B) $.
\end{lem}
\begin{rem}
\label{StrVitProp} 
The main purpose of Edgar's article \cite{Edg07} is to deal with fractal measures in metric spaces with respect to possibly irregular gauge functions. Corollary 5.10 \cite{Edg07} (stated here as Lemma \ref{equalitycomp}) holds true in this general setting if $\mu$ satisfies the so called \textit{Strong Vitali Property} (see \cite{Edg07} p.43 
for a definition and a discussion of this topic). A result due to Besicovitch \cite{Besi45} ensures that any finite measure on $\R^d$ enjoys the Strong Vitali Property. 
Therefore, Lemma \ref{equalitycomp} is an immediate consequence of Edgar's Corollary 5.10 \cite{Edg07}.  \cq 
\end{rem}
\noi
We finally recall the definition of the \textit{packing dimension}: let $\alpha\! \in \!(0, \infty)$; we simply write $\cP_\alpha$ instead of $\cP_g$ when 
$g(r)=r^\alpha$, $r\! \in (0, \infty)$. Let $B\! \subset \! \R^d$. Then, the \textit{packing dimension of $B$}, denoted by $\dim_p(B)$ is the unique real number in $[0, d]$ such that 
\begin{equation}
\label{dimpacdef}
\cP_\alpha(B)=\infty \; \textrm{if }\; \alpha<\dim_p (B) \quad \textrm{and} \quad \cP_\alpha(B)=0 \; \textrm{if }\; \alpha> \dim_p (B)\; .
\end{equation}

\subsection{Exponents.}
\label{sectionexpo}
In this section we briefly recall from \cite{Duq12} several results 
relating power exponents associated with $\psi$ to properties of the gauge function $g$ introduced 
in (\ref{gaugedef}). Recall that the branching mechanism $\psi$ has the L\'evy-Khintchine form (\ref{LevyKhin}). It is well-known that 
$\psi^\prime $ is the Laplace exponent of a subordinator, just like $\psi^{-1}$, the reciprocal of $\psi$. 
Thus, $\varphi= \psi^\prime \circ \psi^{-1}$ is also the Laplace exponent of a subordinator.
Note that $\psi^\prime (0)= \alpha$. As already mentioned, the reciprocal function of 
$\varphi $, denoted by $\varphi^{-1} $, is then defined from $[\alpha, \infty)$ to $[0, \infty)$. 
We also introduce the function $\widetilde{\psi} (\lambda) = \psi (\lambda ) / \lambda$ that easily shown to be also 
the Laplace exponent of subordinator. Next observe that 
$1/ \varphi$ is the derivative of $\psi^{-1}$ and recall that $\psi$ is convex and that $\psi^\prime$, $ \widetilde{\psi}$, 
$\psi^{-1}$ and $\varphi$ are concave. In particular, this implies $\widetilde{\psi}(2\lambda) \leq 2 \widetilde{\psi} (\lambda)$ and the following 
\begin{equation}
\label{convexpsi}
\psi (2\lambda )\! \leq \! 4 \psi (\lambda) , \;  \,     \widetilde{\psi} (\lambda) \! \leq \! 
\psi'(\lambda) \! \leq \! \frac{\psi (2\lambda) -\psi (\lambda)}{\lambda}\! \leq \! 4 \widetilde{\psi} (\lambda)  \; \,  {\rm and } \; \,     \frac{\lambda}{\psi^{-1} (\lambda) } \! \leq \! \varphi (\lambda) \! \leq \! \frac{4\lambda}{\psi^{-1} (\lambda) }. 
\end{equation}
Let $\phi: [0, \infty) \rightarrow [0, \infty)$ be a continuous increasing function. We agree on $\sup \emptyset = 0$ and $\inf \emptyset = \infty $ and we define the following exponents that compare $\phi$ with power functions at infinity: 
\begin{equation}\label{defgamma}
\gamma_{\phi} = \sup \big\{ c \in \R_+ :  \lim_{^{\lambda \rightarrow \infty}} \phi (\lambda) \lambda^{-c} = \infty  \big\} \, , \quad 
\eta_{\phi} = \inf\{  c \in \R_+ : \lim_{^{\lambda \rightarrow \infty}} \phi (\lambda) \lambda^{-c} = 0 \big\} \; . 
\end{equation}
Then, $\gamma_{\phi} $ (resp.~$\eta_\phi$) is the lower exponent (resp.~the upper) 
of $\phi$ at $\infty$. We also introduce the following exponent
\begin{equation} \label{defdelta}
\delta_{\phi}= \sup  \big\{  c \in \R_+   :  \exists \, C \!\in \!(0, \infty)\; \textrm{such that} \; C \phi (\mu) \mu^{-c}  \! \leq \! \phi (\lambda) \lambda^{-c}  , \; 1 \! \leq \! \mu \! \leq \! 
\lambda  \big\}  
\end{equation}
that plays a important role for the regularity of the gauge function. Thus by (\ref{defexpointro}) and (\ref{defdeltaintro}) 
$$\bgamma = \gamma_{\psi} \, , \quad  \boldeta= \eta_{\psi}\, ,  \quad \bdelta= \delta_\psi \; .$$ 
 It is easy to check that $1\! \leq \!  \bdelta \! \leq \!  \bgamma \! \leq \!  \boldeta \! \leq \!  2$. If $\psi$ is regularly varying at $\infty$, all these exponents coincide. In general, they are however distinct and we mention that there exist branching mechanisms $\psi$ of the form (\ref{LevyKhin}) such that $\bdelta\! = \! 1 \! <\!  \bgamma \! = \! \boldeta$: see Lemma 2.3 and Lemma 2.4 in \cite{Duq12} for more detail. As a direct consequence of (\ref{convexpsi}) we have 
$\delta_{\widetilde{\psi}} \! =\! \delta_{\psi'} \! =\!  \bdelta \! -\! 1$, $\gamma_{\widetilde{\psi}} \! =\! \gamma_{\psi'} \! =\!   \bgamma \! -\! 1$ and $\eta_{\widetilde{\psi}} \! =\! 
\eta_{\psi^\prime}\! = \! \boldeta \! -\! 1 $. 
Moreover, we get 
$\delta_\varphi= (\bdelta -1)/\bdelta$, $\gamma_{\varphi}=  (\bgamma -1)/\bgamma$ and $\eta_{\varphi}= (\boldeta -1)/\boldeta$. 

Recall from (\ref{gaugedef}) the definition of the function $g$. The arguments of the proof of Lemma 2.3 $(i)$ in \cite{Duq12} can be immediately adapted to prove that $g: \! (0, r_0) \! \rightarrow \! (0, \infty)$ is is an increasing continuous function such that $\lim_{0+} g= 0$ and such that it satisfies the following. 
\begin{itemize} 
\item[($i$)] The function $g$ satisfies the doubling condition (\ref{doubling}) iff $\bdelta \! >\! 1$. 
\item[($ii$)] If $\psi$ is regularly varying at $\infty$ with exponent $c \! >\! 1$, then $\bdelta \! = \! \bgamma \! =\!  \boldeta=c$ and $g$ is regularly varying at $\infty$ 
with exponent 
$c/(c-1)$. 
\end{itemize}
We shall further need the following bound that is a consequence of (\ref{convexpsi}). 
\begin{lem}\label{convexBP}
Let $g$ the gauge function defined by (\ref{gaugedef}). Let $c \! \in \! (0, \infty)$. 
Then there exists $r(c) \! \in \! (0, r_0) $ that only depends on $c$ such that 
$$\forall r\in(0,r (c)),\quad g(r)\, \invDpsi \! (c/r^2)\leq 4r^2 \; .$$
\end{lem}
\begin{proof}
Take $r(c)\! \in \! (0, r_0)$ such that $\log\log 1/r(c) \! \geq \! 1 \! \vee \! \sqrt{c}$. Thus, 
$$ r\in(0,r (c)),\quad   \invDpsi (cr^{-2}) \! \leq \! \invDpsi ( (r^{-1} \log \log r^{-1})^2) \; .$$ Recall that $\varphi^{-1}\! \! = 
\! \psi  \circ  \invDpsi$. 
By comparing $\widetilde{\psi}$ and $\Dpsi$ thanks to (\ref{convexpsi}), we get for all $r \! \in \! (0,r (c))$
\begin{eqnarray*}
g(r)\, \invDpsi(cr^{-2})&  = & \frac{ \invDpsi(cr^{-2}) \log \log \frac{1}{r} }{\psi ( \psi^{\prime -1} ((\frac{1}{r} \log \log \frac{1}{r} )^2 )}   \leq 
\frac{ \invDpsi ( (\frac{1}{r} \log \log \frac{1}{r} )^2) \log \log \frac{1}{r} }{\psi ( \psi^{\prime -1} ((\frac{1}{r} \log \log \frac{1}{r} )^2 )}  \\
&\leq & \frac{\log\log \frac{1}{r}}{\widetilde{\psi}\left(\invDpsi((\frac{1}{r} \log \log \frac{1}{r} )^2)\right)}\leq \frac{4r^2}{\log\log\frac{1}{r}}\leq 4 r^2,  
\end{eqnarray*}
which is the desired result. \cqfd
\end{proof}

\medskip

\subsection{Height process and Lévy trees.}
\label{sectionheight}
In this section we recall the definition of the height process that encodes L\'evy trees. 
The L\'evy trees are the scaling limit of Galton-Watson trees and they are the genealogy of super-Brownian 
motion.

\paragraph{The height process.} Recall that $\psi$ stands for a branching mechanism of the form (\ref{LevyKhin}). 
{\it We always assume that $\bgamma \! >\! 1$}. 
It is convenient to work on the canonical space $\bbD ([0, \infty), \R )$ of c\`adl\`ag paths equipped with Skorohod topology and the corresponding Borel sigma-field. We denote by  $X= (X_t)_{t \geq 0}$ the canonical process and by $\bbP$ the distribution of the spectrally positive L\'evy processes starting from $0$ whose 
Laplace exponent is $\psi$. Namely, 
$$ \forall t, \lambda \in [0, \infty), \quad \bbE \big[ \! \exp (-\lambda X_t) \big] = \exp (t \psi (\lambda) )\; .
$$
Note that the specific form (\ref{LevyKhin}) of $\psi$ implies that $X_t$ is integrable and that $\bbE [X_t] \! = \! - \alpha t$. This easily entails that $X$ does not drift to $\infty$. Conversely, if a spectrally positive L\'evy process does not drift to $\infty$, then its Laplace exponent is necessarily of the form (\ref{LevyKhin}).  
We shall assume that $\bgamma \! >\! 1$ which easily implies that either $\beta \! >\! 0$ (and $\bgamma \! = \! 2$) 
or $\int_{{(0, 1)}} r \pi (dr) \! =\!  \infty$. 
It entails that $\bbP$-a.s.$\, X$ has unbounded variation paths: see Bertoin \cite{Bebook}, Chapter VII, Corollary 5 $(iii)$.

Note that $\bgamma \! >\! 1$ entails $\int^\infty \! d\lambda / \psi (\lambda) \! < \! \infty$ and,  
as shown by Le Gall and Le Jan \cite{LGLJ1} (see also \cite{DuLG}, Chapter 1), there exists a {\it continuous process} $H= (H_t )_{ t \geq 0}$ 
such that for any $t \in [0, \infty)$, the following limit holds true in $\bbP$-probability: 
\begin{equation}
\label{Hlimit}
H_t=\lim_{\veps\to 0} \frac{1}{\veps}\int_0^t {\bf 1}_{\{I^s_t<X_s<I^s_t+\veps\}}\,\ddr s , 
\end{equation}
where $I^s_t$ stands for $\inf_{s\leq r\leq t} X_r$. 
The process $H= (H_t)_{t \geq 0}$ is called the $\psi$-{\it height process}; it turns out to encode the genealogy of 
super-Brownian motion with branching mechanism $\psi$ as explained below.  
We shall need the following result that is proved in \cite{DuLG}. 
\begin{lem}[\cite{DuLG} Theorem 1.4.4]\label{holderh} 
Assume that $\bgamma \! >\! 1$. Then for every 
$c\in(0,\frac{\bgamma -1}{\bgamma})$, $H$ is $\bbP$-a.s.~locally $c$-Hölder continuous. 
\end{lem}

\paragraph{Excursions of the height process.} We denote by $I$ is the infimum process of $X$: 
$$\forall t \in \R_+, \quad  I_t= \inf_{0\leq r\leq t} X_r \; .$$
When $\psi$ is of the form $\psi (\lambda) \! =\!  \beta \lambda^2$, $X$ is distributed as a Brownian motion and (\ref{Hlimit}) easily implies that $H$ is proportional to $X\! -\! I$, which is distributed as a reflected Brownian motion. In the general cases, $H$ is neither a Markov process nor a martingale. However it is possible to develop an excursion theory for $H$ as follows.

Since $X$ has unbounded variation sample paths, basic results of 
fluctuation theory (see Bertoin \cite{Bebook}, Sections VI.1 and VII.1) entail that $X\! -\! I$ is a strong Markov process in $[0, \infty)$ and that $0$ is regular for 
$(0, \infty)$ and recurrent with respect to this Markov process. Moreover, $\! -I$ 
is a local time at $0$ for $X\! -\! I$ (see Bertoin \cite{Bebook}, Theorem VII.1). 
We denote by $N$ the corresponding excursion 
measure of $X\! -\! I$ above $0$. More precisely, 
we denote by $(a_j, b_j)$, $j \! \in \! \cI$, the excursion intervals of $X\! -\! I$ above $0$ and we define the corresponding excursions by 
$X^j \! = \! X_{(a_j +\,  \cdot )\wedge b_j} \! -\! I_{a_j}$, $j \! \in \! \cI$.  
Then, $\sum_{j\in \cI} \delta_{(-I_{a_j}, X^j)}$ is a Poisson point measure on $[0, \infty) \! \times \! \bbD([0, \infty), \R)$ with intensity 
$\ddr x \, N (\ddr X)$.

  Next observe that under $\bbP $, the value of $H_t$ only depends on the excursion of $X\! -\! I$ straddling time $t$ and we easily check that 
$$\bigcup_{^{j\in \cI}} (a_j, b_j)= \{ t \geq 0: H_t >0 \} \; .$$ 
This allows to define 
the height process under $N$ as a certain measurable function $H(X)$ of $X$. 
We denote by $C(\R_+, \R)$ the space of the continuous functions from $[0, \infty)$ to $\R$ equipped with the topology of the uniform convergence on every compact subsets of $[0, \infty)$; by convenience, we shall slightly abuse notation by 
denoting by $H \! =\!  (H_t )_{ t\geq 0}$ the {\it canonical process on $C(\R_+, \R)$} and by denoting by $N(\ddr H)$ the "distribution" of the height process $H(X)$ associated with $X$ under the excursion measure $N (\ddr X)$. 
Then we derive from the previous results the following Poisson decomposition of the height process $H(X)$ associated with $X$ under $\bbP$: 
for any $j\in  \cI$, set $H^j =H_{(a_j+\, \cdot )\wedge b_j} $; then, under $\bbP$, the point measure 
\begin{equation}
\label{Poissheight}
\sum_{j\in \cI} \delta_{(-I_{a_j},H^j)}
\end{equation}
is distributed as a Poisson point measure on $[0, \infty) \! \times \! C(\R_+, \R)$ with intensity $\ddr x\,  N(\ddr H)$. For more details, we refer to \cite{DuLG}, Chapter 1.

\smallskip

We denote by $\sigma$ the \textit{duration} of $X$ under its excursion measure $N$ (with an obvious definition). It is easy to check that $H$ and $X$ under $N$ have the same duration and that the following holds true.  
$$ \textrm{$N$-a.e.} \quad \sigma < \infty\, , \quad H_0=H_{\sigma}=0 \quad \textrm{and} \quad H_t >0  \Longleftrightarrow t \in (0, \sigma) \; .$$ 
Basic results of fluctuation theory (see Bertoin \cite{Bebook}, Chapter VII) also entail: 
\begin{equation}
\label{lifetimeexc}
\forall \lambda\in (0, \infty) \quad N \big[ 1\! -\! e^{-\lambda \sigma}  \big]= \psi^{-1} (\lambda ). 
\end{equation}

\noindent
{\bf Local times of the height process.} 
We recall  from \cite{DuLG}, Chapter 1, Section 1.3, 
the following result: there exists a jointly measurable process $(L^a_s)_{a, s \in [0, \infty)}$ 
such that $\bbP$-a.s.~for any $a \! \in \! [0, \infty) $, $s \mapsto L^a_s$ 
is continuous, non-decreasing and such that 
 \begin{equation}
 \label{approxtpsloc}
 \forall \, t,  a \geq 0 , \quad  \lim_{\varepsilon \rightarrow 0} \bbE \left[  \sup_{ 0\leq s \leq t} \left| \frac{1}{\varepsilon} \int_0^s \!\!\!  \ddr r \, 
\un_{\{ a < H_r \leq a+ \varepsilon \}} -L_s^a \right| \right] =0\; . 
\end{equation}
The process $(L^a_s)_{s \in [0, \infty)}$ is called the {\it $a$-local time of $H$}. Recall that $I$ stands for the infimum process of $X$. One can show 
(see \cite{DuLG05}, Lemma 1.3.2) that for fixed $t$, $L^0_t=-I_t$. Moreover, one can 
observe that the support of the random Stieltjes measure $\ddr L^{a}_{{\cdot }}$ is contained in the closed set $\{t \geq 0:H_t=a\}$. 

  A general version of the {\it Ray-Knight theorem for $H$} asserts the following. For any $x\geq 0$, set $T_x=\inf \{ t\geq 0 \; :\;
X_t=-x\}$; then, {\it the process $(L^a_{T_x} \; ;\; a \geq 0)$ is  distributed as a  continuous-states branching process CSBP with branching mechanism $\psi$ and initial state $x$} (see Le Gall and Le Jan \cite{LGLJ1}, Theorem 4.2, and also \cite{DuLG}, Theorem 1.4.1).

\medskip

 It is possible to define the local times of $H$ under the excursion measure $N$ as follows. For any $b >0$, let us  set $v(b)=N( \sup_{^{t \in [0, \sigma ]}} H_t > b )$. Since $H$ is continuous, the Poisson decomposition (\ref{Poissheight}) implies that $v(b) < \infty$, for any $b >0$. 
It is moreover clear that $v$ is non-increasing and that $\lim_{\infty} v= 0 $. Then, for every $a\in (0, \infty)$, we define a continuous increasing process $(L^a_t)_{t\in [0, \infty)} $, such that for every
$b \in (0,\infty)$ and for any $t \! \in \!  [0, \infty)$, one has 
\begin{equation}
\label{localapprox}
\lim_{\varepsilon \rightarrow 0} \, 
N  \Big[  \un_{\{\sup H>b \}} \!\!  \sup_{ 0\leq s \leq t\wedge \sigma} \Big|  \frac{1}{\varepsilon} \int_0^s
\!\!\! \ddr r \, 
\un_{\{ a-\varepsilon< H_r \leq a\}} -L_s^a \Big| \, \Big] =0.
\end{equation}
We refer to \cite{DuLG}{ Section 1.3} for more details. Note that $N$-a.e.~$L^a_t=L_\sigma^a$ for all $t\geq \sigma$. The process $(L^a_t)_{t\in [0, \infty)} $ 
is the {\it $a$-local time of the excursion of the height process}.

\paragraph{L\'evy trees.} We briefly explain how the height process $H$ under its excursion measure $N$ can be viewed as the contour process of a tree called the L\'evy tree.  
Recall that $\sigma$ is the duration of $H$ under $N$. For any $s,t\in[0,\sigma]$, we set
\begin{equation}
\label{distH}
d(s,t)= H_t +H_s -2  \inf_{u\in [s\wedge t , s\vee t]} H_u \; .
\end{equation}
The quantity $d(s,t)$ represents the distance between the points corresponding to $s$ and $t$ in the L\'evy tree. Indeed $d$ is obviously symmetric in $s$ and $t$ and 
we easily check that $d$ satisfies the triangle inequality. 
Two real numbers $s,t \! \in \! [0, \sigma] $ correspond to the same point in 
the L\'evy tree iff $d(s, t) \! = \! 0$, which is denoted by $s\sim t$. Observe that $\sim$ is an equivalence relation. The L\'evy tree is then given by the quotient set 
$$\cT = [0, \sigma ] / \sim \; .$$ 
Then, $d$ induces a true metric on $\cT$ that we keep denoting $d$. 
Denote by $p : [0,\sigma]\rightarrow \cT$ the \textit{canonical projection}. 
Since $H$ is continuous $N$-a.e., so is $p$, which implies that $(\cT, d)$ 
is a random compact connected metric space. More specifically, $(\cT, d)$ is $N$-a.e~a compact 
\textit{real-tree}, namely a compact metric space such that any two points are connected by a unique self-avoiding path, that turns out to be geodesic: see \cite{DuLG05} for more details on L\'evy trees viewed as real-trees. 

The \textit{mass measure} of the L\'evy tree $\cT$, denoted by $\bm$, 
is the pushforward measure of the Lebesgue measure $\ell$ on $[0,\sigma]$ via the canonical projection $p$. Namely, $N$-a.e.~for all Borel measurable function $f: \cT \rightarrow \R_+$,  
$$ \langle \bm ,   f \rangle = \int_0^\sigma \ell (\ddr t) \, f( p(t)) \; .$$ 
One can show that $N$-a.e.~the mass measure is diffuse; obviously its topological support is $\cT$ and $\bm (\cT)\! =\!  \sigma$. We refer to \cite{DuLG05} for more details. 

Let $r\! \in (0, \infty)$ and let $t \! \in\! [0,\sigma]$. Let $B(p(t),r)$ denote the open ball in $(\cT, d)$ with center $p(t)$ and radius $r$. Then the 
mass measure of $B(p(t),r)$ in $(\cT, d)$ is then given by 
\begin{equation}\label{masstypicalball}
\ba(t,r):=\bm \left(B(p(t),r)\right)=\int_0^{\sigma} \!\! \!  \ell (\ddr s) \, \un_{\{d(s,t)\leq r\}} \; .
\end{equation}  
We shall need the following result on the lower density of $\bm$ at typical points that is proved in 
\cite{Duq12}, Theorem 1.2. To that end, we set 
\begin{equation}\label{defk}
\forall r\in(0, \alpha \wedge e^{-e}),\quad k(r):=\frac{\log\log\frac{1}{r}}{\varphi^{-1}(\frac{1}{r}\log\log
\frac{1}{r})}.
\end{equation}

\begin{lem}[\cite{Duq12} Theorem 1.2]\label{lemmetree}
Let $\psi$ be a branching mechanism of the form (\ref{LevyKhin}). Let $\sigma$ the the duration of the height process $H$ under its excursion measure $N$. 
Let $k$ be as in (\ref{defk}). 
Assume that $\bdelta \! > \! 1$. Then, there exists a constant $\Cl[c]{clemmetree} 
\! \in \! (0,\infty)$ 
that only depends on $\psi$ such that 
$$\textrm{$N$-a.e.~for $\ell$-almost all $t \in [0,\sigma]$}\qquad  \liminf\limits_{r\rightarrow0}\frac{\ba(t,r)}{k(r)}\geq \Cr{clemmetree} \; ,$$
where $\ba(t,r)$ is defined by (\ref{masstypicalball}) for all $r\! \in \! (0, \infty)$ and for all $t \! \in \! [0, \sigma]$.  
\end{lem}

\noindent
{\bf The exploration process.} As already mentioned, the height process is not a Markov process. To explore in a Markovian way the L\'evy tree, Le Gall and Le Jan in \cite{LGLJ1} 
introduce a measure valued process $\rho=(\rho_t)_{ t\geq 0}$ that is called \textit{the exploration process} whose definition is the following. Denote by 
$M_f(\R_+)$ the set of finite measures on $[0,\infty)$ equipped with the \textit{total variation distance}. 
Recall that $X$ under $\bbP$ is a spectrally positive L\'evy process
starting from $0$ whose Laplace exponent is $\psi$ that satisfies $\bgamma \! >\! 1$. 
We denote by $\cF^X_t$ the sigma-field generated by $X_{\cdot \wedge t}$ augmented with the $\bbP$-negligible events. 
Recall that for all $s, t \! \in \! [0, \infty)$ such that $s \! \leq \! t$, $I^s_t$ stands for $\inf_{u \in [s,t]} X_u$. 
Then, for all $t\! \in [0, \infty)$, the following definition makes sense under $\bbP$ or $N$:
\begin{equation}
\label{new-rho}
\rho_t(\ddr r)=\beta \un_{[0,H_t]}(r)\,\ddr r+\sum\limits_{\substack{0<s\leq t\\X_{s-}<I^s_t}}
 (I^s_t-X_{s-})\,\delta_{H_s}(\ddr r).
\end{equation}
Note that the $M_f(\R_+)$-valued process $\rho$ is $(\mathcal{F}^X_t)_{ t\geq 0}$-adapted. 
The height process $H$ can be deduced from $\rho$ as follows: for any $\mu \! \in \! M_f(\R_+)$, we denote by $\supp (\mu)$ its topological support and we define 
$$H(\mu) = \sup ( \supp (\mu)) \; ,$$ 
that is possibly infinite. We can show that 
$$\textrm{$\bbP$-a.s.~(or $N$-a.e.)} \quad \forall t\in  [0,\infty), \quad \supp (\rho_t)=[0,H_t] \; .$$ 
As proved in \cite{DuLG}, the exploration process $\rho$ admits a c\`adl\`ag modification under $\bbP$ and $N$. By Proposition 1.2.3 \cite{DuLG}, under $N$, $\rho$ is a càdlàg strong Markov process with respect to $(\mathcal{F}^X_{t+})_{ t\geq 0}$. 

\subsection{The L\'evy Snake.}\label{levysnake}
The $\psi$-L\'evy snake is a generalization of Le Gall's Brownian snake that greatly facilitates the study of super processes: it provides a Markovian parametrisation of the genealogy and the spatial positions of the underlying continuous population that gives rise to the super process. We recall from 
\cite{DuLG}, Chapter 4, the following definition of the $\psi$-L\'evy snake. 
To that end, recall that $\xi\! = \! (\xi_t)_{t\geq 0}$ is a  $\R^d$-valued continuous process defined on 
$(\Omega, \cF)$ and that for any $x\! \in \! \R^d$, $\P_x$ is a probability measure on $(\Omega, \cF)$ such that under $\P_x$, $\xi$ is distributed as a standard $d$-dimensional Brownian motion starting from $x$.

\paragraph{Snake with a deterministic Hölder-regular lifetime process.} 
We denote $\cW$ the set of \textit{continuous stopped paths}, namely the set of pairs $(\w,\zeta_\w)$ where $\zeta_\w \in[0,\infty)$ and $\w :[0,\zeta]\rightarrow\R^d$ is a \textit{continuous} 
function. Here $\zeta_\w$ is the \textit{lifetime} of $\w$.  
We shall slightly abuse notation by simply denoting 
$\w$ instead of $(\w,\zeta)$ in the sequel. The set 
$\cW$ is equipped with the  metric $\bd$ defined for $\w,\w'\in\cW$ by : 
$$\bd(\w,\w^{\prime})=\lvert\zeta_{\w} \! -\! \zeta_{\w^\prime}\rvert + \sup_{t\geq 0} \; 
\lVert \, \w(t \! \wedge \! \zeta_{\w}) \! -\! \w'(t \! \wedge \! \zeta_{\w^\prime}) \, \rVert.$$
Here $\lVert \cdot \rVert$ stands for the Euclidian norm on $\R^d$. 
It can be shown that $(\cW,\bd)$ is a Polish space.

To define the finite dimensional marginal distributions of the snake, we first need to introduce its transition kernels. Let $\w\in\cW$, let $a\in [0,\zeta_\w]$ and let $b\in[a,\infty)$. 
We plainly define a law $R_{a,b}(\w,\ddr\w')$
on $\cW$ by requiring the following.
\begin{description}
\item{$(i)$} $R_{a,b}(\w,\ddr \w')$-a.s.~$\w'(t)=\w(t)$, $\forall
t\in[0,a]$. 
\item{$(ii)$} $R_{a,b}(\w,\ddr \w')$-a.s.~$\zeta_{\w'}=b$.
\item{$(iii)$} The law of $(\w'(a\! +\! t))_{0\leq t\leq b-a}$ 
under $R_{a,b}(\w,\ddr \w')$ is the law
of $(\xi_t)_{0\leq t\leq b-a}$ under $\P_{\w(a)}$.
\end{description}
In particular, $R_{0,b}(\w,\ddr \w')$ is the law of $(\xi_t)_{0\leq t\leq b}$ under $\P_{\w(0)}$.

\medskip

We denote by $(W_s)_{ s\geq 0}$ the canonical process on the space $\cW^{\R_+}$ of the $\cW$-valued functions on $\R_+$ equipped with the product sigma-field. We next fix $x\! \in \! \R^d$. 
We slightly abuse notation by denoting $x$ the stopped path with null lifetime starting from $x$ (and therefore ending at $x$).  
Let $h \! \in \!  C(\R_+,\R_+)$ such that $h(0)=0$. 
We call $h$ the \textit{lifetime process}. 
For all $s,s^\prime \! \in \! \R_+$ such that $s^\prime \! \geq \! s$, 
we use the notation $b_h(s,s^\prime)=\inf_{s\leq r\leq s^\prime}h(r)$. 
From the definition of the laws $R_{a, b}$ and Kolmogorov extension theorem 
there is a unique probability measure $Q^h_{x}$ on $\cW^{\R_+}$ such
that for all $0=s_0<s_1<\cdots<s_n$,
\begin{align}\label{defQh}
Q^h_{x}& \big[W_{s_0}\in A_0,\ldots,W_{s_n}\in A_n \big]\\
&=\un_{A_0}(x)\int_{A_1\times \cdots\times A_n} \!\!\!\!\!\!\!\!\!\!\!\! \!\!\! 
R_{b_h(s_0,s_1),h(s_1)}(\w_0,\ddr\w_1)\ldots
R_{b_h(s_{n-1},s_n),h(s_n)}(\w_{n-1},\ddr \w_n)\nonumber.
\end{align}
Note that $(h ,x) \mapsto Q^h_{x}$ is measurable and for all $t \! \in \! \R_+$, 
$Q^h_{x}$-a.s.~$\zeta_{W_t}=h(t)$. 

We next discuss the regularity of the process $W$ under $Q_{x}^h$. To that end, we assume the following.  
$$\textrm{The lifetime process $h$ is locally Hölder continuous with exponent $r \in (0, 1]$.}$$
Fix $p\! \in \! (1, \infty)$ and $t_0 \! \in \! (0, \infty)$. 
The last inequality of the proof of Proposition 4.4.1\cite{DuLG} p.~120, entails that 
there exists a constant $C$ that only depends on $t_0$, on $p$ 
and on the Hölder constant of $h$ on $[0, t_0]$, such that  
\begin{equation}\label{Komolg2}
\forall s, t\in [0, t_0], \quad  
Q^h_{x} \! \big[ \, \bd(W_s,W_{t})^p \, \big] \leq C \,  |t\! -\! s|^{pr/2}.
\end{equation}
If $p\! >\! 2/r$, then the Kolmogorov continuity criterion applies and asserts that there exists a continuous modification of the process $W$. We slightly abuse notation by keeping notation $Q^h_{x}$ for law on $C(\R_+,\cW)$ of such a modification; likewise, we also 
keep denoting by $(W_t)_{t\geq 0}$ 
the canonical process on $C(\R_+,\cW)$. We then call $Q_{x}^h$ the law of the \textit{snake with lifetime process $h$ starting from $x$}. 
Working on $C(\R_+, \cW)$, we see that $Q^{h}_{x}$-a.s.~for all 
$t\! \in \! \R_+$, $\zeta_{W_t} = h(t)$. We then set 
\begin{equation}
\label{termpt}
\widehat{W_t}= W_t (h(t)) \; .
\end{equation}
The process $\widehat{W}$ is called the snake's \textit{endpoint process} that is $Q^h_{x}$-a.s.~continuous. From (\ref{defQh}), we easily get that under $Q^h_x$, the endpoint process is Gaussian whose covariance is characterized by 
\begin{equation}
\label{covar}
\forall s, t \in \R_+ , \quad 
 Q^h_x \left( \lVert \widehat{W_t} \! -\! \widehat{W_s} \rVert^2 \right) = h(t)+h(s)-2\!\!\!\inf_{s\wedge t \leq u \leq s\vee t}\!\!\!\!  h(u) \; .
 \end{equation}
Moreover since (\ref{Komolg2}) holds for any $p\! >\! 1$, the 
Kolmogorov criterion implies that for any $q \! \in \! (0,r/2)$, $\widehat{W}$ is $Q^h_x$-a.s.~locally $q$-Hölder continuous.

\paragraph{The definition of L\'evy snake.}  
The L\'evy snake is the snake whose lifetime process is the height process $H$ introduced 
in Section \ref{sectionheight}. Recall that we assume that $\bgamma \! > \! 1$ and recall from Lemma \ref{holderh} that $H$ is $\bbP$-a.s.~(or $N$-a.e.) Hölder regular and that the previous construction of the snake applies. We then define the excursion measure of the $\psi$-Brownian snake starting from $x \! \in \! \R^d$ by 
\begin{equation}
\label{snex}
\N_x =\int_{C(\R_+, \R)} \!\!\!\!\!\!\!  \!\!\!\!\!\!\!   N(H \! \in \! \ddr h) \, Q^h_{x}\; .
\end{equation}
Then $H$ is the lifetime process of $W$. Namely $\N_x$-a.e.~for all $t\! \in\! [0,\sigma]$, $\zeta_{W_t} \! = \! H_t$ and thus, $\widehat{W_t}\! =\!  W_t (H_t)$.  
 Moreover, under $\N_x$, the conditional law of $W$ given $H$ is $Q^H_x$: we refer to 
 \cite{DuLG}, Chapter 4, for more details. Lemma \ref{holderh} and the results discussed right after (\ref{covar}) entail the following lemma.  
\begin{lem}
\label{snakhold} Assume that $\bgamma \! >\! 1$. Then, for any $q \! \in \! (0, \frac{\bgamma -1}{2\bgamma } )$, $\widehat{W}$ is $\N_x$-a.e.~locally 
$q$-Hölder.  
\end{lem}
The range of the endpoint process $\widehat{W}$ is a connected 
compact subset of $\R^d$ and we use the following notation
\begin{equation}
\label{rangeWdef}
 \cR_W =  \big\{ \widehat{W}_t \, ;  \; t \in [0, \sigma ]  \big\} \; .
\end{equation}
Recall that for any $a\! \in \! (0, \infty)$, $ (L^a_s)_{s\geq 0}$ stands for the local time of $H$ at level $a$. We then denote by $\z_a(W)$ the random
measure on $\R^d$ defined by 
$$\langle \z_a (W),f \rangle =\int_0^{\sigma}
\ddr L^a_s\,f(\wh W_s), $$
for any Borel measurable $f: \! \R^d \! \rightarrow \! \R_+$. We also set $\z_0 (W)\! =\!  0$, the null measure. 
Recall that $M_f(\R^d)$ stands for the space of finite Borel measures on $\R^d$ equipped with the topology of weak convergence. We can proves that under $\N_x$, the $M_f(\R^d)$-valued process 
$(\z_a)_{a\geq 0}$ has a c\`adl\`ag modification that is denoted in the same way by convenience. 
The \textit{occupation measure of the snake} $\cM_W$ is then defined by 
\begin{equation}
\label{occuWdef}
\langle \cM_W , f\rangle   = \int_0^\infty \langle \z_a (W) , f \rangle  \ddr a= \int_0^\sigma f\big( \widehat{W_s}\big) \, \ddr s \; ,   
\end{equation}
for any Borel measurable $f: \! \R^d \! \rightarrow \! \R_+$.

We then recall Theorem 4.2.1 \cite{DuLG} that connects the $\psi$-Lévy snake to super Brownian motions. 
\begin{thm}[\cite{DuLG} Theorem 4.2.1]
\label{super}
We keep notation from above. Assume that $\bgamma \! >\! 1$.  
Let $\mu \! \in \! M_f(\R^d)$. Let 
$$\sum_{i\in \cJ} \delta_{(x_i,W^i)}$$
be a Poisson point measure on $\R^d \! \times \! \cW$ with intensity $\mu(\ddr x)\N_x(\ddr W)$.
For every $a\! \in (0, \infty)$ set 
$$Z_a=\sum_{i\in \cJ} \z_a(W^i) \; .$$
and also set $Z_0=\mu$. Then, the process $(Z_a)_{ a\geq 0}$ is a
$\psi$-super Brownian motion starting from $\mu$ (as defined in the introduction section). Moreover, if $\, {\bf R}$ and $\, {\bf M}$ are defined in terms of $Z$ by (\ref{defboldR}) and (\ref{defboldM}), then, \begin{equation}
\label{connection}
{\bf R} \cup \{ x_j \; ; \, j \in \cJ \} = \bigcup_{j \in \cJ} \cR_{W^j}  \quad {\rm and} \quad {\bf M} = \sum_{j \in \cJ } \cM_{W^j} \; . 
\end{equation}
\end{thm}
The last point (\ref{connection}) is not part of Theorem 4.2.1 \cite{DuLG} but it is an easy consequence of that result. To simplify notation, when there is no risk of confusion we shall simply write 
$$ \z_a := \z_a (W), \quad \cR:= \cR_W  \quad \textrm{and} \quad  \cM:= \cM_W \; .$$

\paragraph{Consequences of Markov property.} 
As the height process, the $\psi$-Lévy snake $(W_t)_{ t\geq 0}$ defined above is not a Markov process. However $\ov W:=(\rho_t, W_t)_{ t\geq 0}$ is a strong Markov process under $\N_x$: see Theorem 4.1.2 \cite{DuLG} for more details. 
Instead of fully discussing the Markov property of $\ov W$, we only state here 
the various results we need, that are consequences of the strong Markov property.  

\medskip

Denote by $(\cF^W_t)_{t\geq 0}$ the filtration generated by 
$(\ov W_t)_{t\geq 0}$. As a consequence of the strong Markov property for $\ov W$ 
(see \cite{DuLG}, Theorem 4.1.2) and a specific decomposition of the snake into excursions above the infimum of its lifetime proved in Lemma 4.2.4 \cite{DuLG}, we get the following result 
that is used in Section \ref{sectionbadpts}. 
\begin{lem}[\cite{DuLG} Theorem 4.1.2 and Lemma 4.2.4]
\label{Markovpretalemploi} 
Let $T$ be a $(\mathcal{F}^W_{t+})_{t\geq 0}$-stopping time. 
Let $Y$ be a nonnegative $\mathcal{F}^W_{T+}$-measurable random variable. 
Let $G: \R^d \rightarrow \R_+$ be 
Borel measurable. Then,
\begin{align*} 
\N_0 \Big[ \un_{\{ 0 < T < \infty \}} Y \exp\Big( \! -\! \! \int_T^{\sigma} \!\!\! G(\widehat{W}_{s})\ddr s  \Big) &  \Big] \\
& \hspace{-10mm}=  \N_0 \Big[\un_{\{0 < T< \infty \}} Y \exp \Big( \! -\!\! \int_{[0, H_T] } \!\!\! \!\!\!\!\!\!\! \!\! \rho_T(\ddr h) \, 
\N_{W_T(h)} \big(1\! -\! e^{-\int_0^\sigma G(\widehat{W}_{s})\ddr s}\,  \big) \Big)  \Big].
\end{align*}
\end{lem}

  We shall apply the strong Markov property of $\ov W$ at specific hitting times of the snake. 
More precisely, let us introduce several notations. 
Let $r \! \in \! (0, \infty)$. We define the first hitting time of $W$ of the closed ball $\overline{B}(0,r)$ in $\R^d$ by 
\begin{equation}\label{deftau}
\tau_{r}:=\inf \big\{t\in \R_+  : \widehat{W}_t \in \overline{B}(0,r) \big\},
\end{equation}
with the convention $\inf\emptyset =\infty$. We also introduce the following function 
\begin{equation}\label{predefuxr}
\forall x \in \overline{B}(0, r)^c, \quad  u_r(x):=\N_x\left(\tau_r <\infty\right) \; .
\end{equation}
Since $t\mapsto\widehat{W}_t$ is continuous, we also get 
\begin{equation}
\label{repi}
\forall x \in \overline{B}(0, r)^c, \quad  
 u_{r}(x)=\N_x (\cR \cap \overline{B}(0,r) \neq \emptyset ) \; .
 \end{equation}
From \cite{DuLG} p.~121 and p.~131, we know that $u_r(x) \! \in \! (0, \infty)$, for all $r \! \in \! (0, \infty)$ and all $x \! \in \! \overline{B} (0, r)^c$, and that $u_r$ is moreover radial. We then denote by  
$\widetilde{u}_r$ the function from $(r, \infty)$ to $(0, \infty)$ such that 
$$ \forall x \in \overline{B}(0, r)^c, \quad  \widetilde{u}_r (\lVert x \rVert ) = u_r (x) \; .$$
Let $r^\prime \! \in \! (0, \infty)$. For all stopped path $\w \! \in \! \cW$, we next set 
\begin{equation}\label{defTB}
T_{r^\prime}(\w)=\inf \big\{ s \in [0,\zeta_\w] : \w(s)\in \overline{B}(0,r^\prime) \big\}\; , 
\end{equation} 
with the convention $\inf\emptyset =\infty$. 
We then define a function $\varpi: \R_+^2 \rightarrow \R_+$ by 
\begin{equation}\label{defgammapsi}
\forall \lambda_1, \lambda_2 \in \R_+ , \quad \varpi (\lambda_1 ,\lambda_2)= \left\{
\begin{array}{ll}
\left(\psi (\lambda_1) -\psi ( \lambda_2) \right) / (\lambda_1-\lambda_2)
\quad& {\rm if } \quad \lambda_1\neq \lambda_2, \\
\psi'(\lambda_1) \quad& {\rm if } \quad \lambda_1=\lambda_2 \; .\hfill\\ 
\end{array}
\right.
\end{equation}
Recall that $\xi\! = \! (\xi_t)_{t\geq 0}$ is a  $\R^d$-valued continuous process defined on 
$(\Omega, \cF)$ and that for any $x\! \in \! \R^d$, $\P_x$ is a probability measure on $(\Omega, \cF)$ such that under $\P_x$, $\xi$ is distributed as a standard $d$-dimensional Brownian motion starting from $x$. 

  The following proposition is a specific application of Theorem 4.6.2 \cite{DuLG} that we use in the proof of Lemma \ref{BPnouveau} in Section \ref{sectionbadpts}. 
\begin{prop}[\cite{DuLG} Theorem 4.6.2]
\label{firsthitting}
Let $x \! \in\! \R^d$. Let $r, r^\prime \! \in \! (0, \infty)$ be such that $ r^\prime \! >\! r$ and 
$x \! \in\! \overline{B}(0,r^\prime )^c $. We keep the previous notation.  
Let $F, G: \cW \rightarrow \R_+$ be Borel-measurable. Then, 
\begin{align*} 
\N_x & \Big[ \un_{\{\tau_r <\infty\}} 
F\big( 
(W_{\tau_r} (s))_{0\leq s\leq T_{r^\prime}(W_{\tau_r})}
 \big)
\exp \Big(\! -\!\! 
\int_{[0, T_{r^\prime}(W_{\tau_r})]}  \!\!\!\!\!  \!\!\!\!\!  \!\!\!\!\! \!\!\!\!\! 
\rho_{\tau_r}(\ddr h) \; G \big( 
(W_{\tau_r} (s))_{0\leq s\leq h} 
\big) 
\Big)
\Big] \\
&=\widetilde{u}_r (r^\prime ) \, \E_x
\Big[ \un_{\{T_{r^\prime} (\xi) < \infty \} }
F \big( 
(\xi_s)_{ 0\leq s\leq T_{r^\prime} (\xi)}
\big)
\exp \Big(\! - \!\! 
\int_{[0, T_{r^\prime}(\xi)] }\!\!\!\!  \!\!\!\!  \!\!\!\! \!\!\!\!  \ddr h\;  \varpi \big(u_{r}(\xi_h) ,G((\xi_s)_{0\leq s\leq h})\big)
\Big)
\Big] \; .
\end{align*}
\end{prop}

\paragraph{Palm formula} We introduce the following notation 
\begin{equation}
\label{psistardef}
\forall \lambda \in \R_+, \quad \psi^* (\lambda)= \psi (\lambda)-\alpha \lambda \; , 
\end{equation} 
that is clearly the Laplace exponent of a spectrally L\'evy process. 
We then fix $x\! \in \! \R^d$. Again, recall that $\xi= (\xi_t)_{t\geq 0}$ under $\P_x$ is distributed as a standard $d$-dimensional 
Brownian motion starting from $x$. Let $U\!= \! (U_a)_{a\geq 0}$ be a subordinator defined on $(\Omega, \cF, \P_x)$ that is assumed to be independent of $\xi$ and whose Laplace exponent is 
$$ \forall \lambda \in \R_+, \quad  \widetilde{\psi}^*(\lambda): =\frac{\psi(\lambda)}{\lambda} -\alpha \; .$$ 
For any $a \! \in \! \R_+$,  we denote by 
$R_a(\ddr b)$  the random measure $\un_{[0,a]}(b)
\ddr U_b$. We first recall from \cite{DuLG}, formula (106) p.~105, 
that for any measurable function $F: M_f(\R_+) \! \times \! \cW \rightarrow \R_+$, the following holds true: 
\begin{equation} 
\label{invar-snake}
\N_x\Big(\int _0^\sigma \!\!\!  \ddr s\,F(\rho_s\, ,W_s) \, \Big)
=\int_0^\infty \!\!\! \ddr a\,e^{-\alpha a}\,\E_x \big[ F (R_a \, , (\xi_s)_{0\leq s\leq a}\, ) \big] \; .
\end{equation}

\medskip

We next provide a Palm decomposition for the occupation measure $\cM$ of the snake that is used to 
estimate its lower local density at "typical" points. To that end we need to introduce the following auxiliary random variables. Let $(V_t)_{ t\geq 0}$ be a subordinator defined on $(\Omega, \cF, \P_0)$ that is independent of $\xi$ and whose Laplace exponent is $\Dpsistar(\lambda): =\! \Dpsi(\lambda) \! -\! \alpha$. We then introduce the following point measure on $[0,\infty) \! \times \! C(\R_+,\cW)$: 
\begin{equation}\label{defN}
\Nstar(\ddr t \, \ddr W)=\sum\limits_{j\in\mathcal{J}^*}\delta_{(t_j,W^j)} \; , 
\end{equation}
such that under $\P_0$ and conditionally given $(\xi,V)$, $\cN^*$ is distributed as a Poisson point process with intensity $\ddr V_t \, \N_{\xi_t} (\ddr W)$. 

For all $j\! \in \! \mathcal{J}^*$, we denote by $\cM_j$ the occupation measure of the snake $W^j$. Then   
for all $a\! \in \! \R_+$, we define the following random measure on $\R^d$; 
\begin{equation}\label{defMa}
\cM_a^*=\sum\limits_{j\in\mathcal{J}^*}\un_{[0,a]}(t_j)\mathcal{M}_j.
\end{equation}
Note that $\cM_a^*$ is $\P_0$-a.s.~a random finite Borel measure on $\R^d$. 
Informally $\cM_a^*$ is the sum of the 
the occupation measure of the snakes grafted at a rate given by $V$ on the \textit{spatial spine $\xi$} between time $0$ and $a$. As a by-product of Formula (113) p.113~in \cite{DuLG}, we get the following Palm decomposition of $\cM$ under~$\N_0$.  
\begin{prop}[\cite{DuLG} (113)]
\label{propPalm}
Let $F : \R^d \! \times \! M_f(\R^d)\rightarrow\R_+$ be measurable. 
Then, 
\begin{equation}\label{Palm}
\N_0 \Big[ \int_{\R^d} \!\!\!\! \cM(\ddr y) \, F(y,\cM) \Big]
=\int_0^\infty \!\!\! \ddr a \,  e^{-\alpha a} \E_0 \left[ F(\xi_a,\Ma)\right].
\end{equation}
\end{prop}
We shall mostly use the Palm formula in this way: for any measurable functional 
$G: \bbD(\R_+, \R) \rightarrow \R_+$, we get 
\begin{equation}\label{corPalm}
\N_0 \Big[ \int_{\R^d} \!\!\!\! \cM(\ddr y) \,  G\big( (\cM(B(y,r)))_{ r\geq 0} \big) \Big]
=\int_0^\infty\! \!\! \! \ddr a \, e^{-\alpha a}\E_0 \big[  G\big( (\cM_a^*(B(0,r)))_{r\geq 0} \big) 
\big].
\end{equation}

\section{Estimates.}
\label{secestimates}
\subsection{Tail of a subordinator.}
\label{secsubord}
Recall from (\ref{psistardef}) that $\psi^* (\lambda)= \psi (\lambda) -\alpha \lambda$, that is the Laplace transform of a spectrally positive L\'evy process. Therefore, $\Dpsistar$ is the Laplace exponent  of a subordinator that is conservative for $\Dpsistar (0) \! = \! 0$. By subordination, $\sqrt{\Dpsistar\circ \psi^{-1} }$ is also 
the Laplace exponent of a conservative subordinator. The main idea of the proof of Theorem \ref{localth} consists in comparing 
the mass of a typical ball with a subordinator whose Laplace exponent is $\sqrt{\Dpsistar\circ \psi^{-1} }$. To that end, we first need the following result. 
\begin{lem}
\label{lemmesubord} Assume $\bdelta \! >\! 1$. Recall from (\ref{gaugedef}) the definition of the gauge function $g\! :\!  (0, r_0)\!  \rightarrow \! (0, \infty)$. 
Let $(S_r)_{r\in [0, \infty)}$ be a subordinator defined on the auxiliary probability space $(\Omega, \mathcal{F}, \P_0)$. We assume that the Laplace exponent of $S$  
is $\sqrt{\Dpsistar\circ \psi^{-1} }$. Let $\rho_n \in (0, r_0)$, $n \in \N$, be such that 
\begin{equation}
\label{miam}
 \rho_{n+1} \leq e^{-n} \rho_n   \quad {\rm and} \quad  \sup_{{n \geq 0}} \,  n^{-2} \log 1/ \rho_n < \infty\; .
\end{equation}
Then, $\sum_{n \geq 0} \P_0 \big(S_{\rho_n} \leq g(4 \rho_n) \big) = \infty $. Moreover, we get 
$$ \textrm{$\P_0$-a.s.} \quad \limsup_{n\rightarrow \infty} \frac{S_{\rho_{n+1}}}{g(4 \rho_n) } \; < \infty \; .$$
\end{lem}
\begin{proof}  To simplify notation, we set $\Phi \! = \! \sqrt{\Dpsistar\circ \psi^{-1} }$. Thus, $\E_0 [\exp (-\lambda S_r)]= \exp(-r\Phi (r))$. 
Denote by $\Phi^{-1}$ the reciprocal function of $\Phi$. 
For any $r\in (0, e^{-1})$, we set 
$$g_* (r)= \frac{\log \log \frac{1}{r} }{ \Phi^{-1} (\frac{1}{r} \log \log \frac{1}{r} )} \; .$$ 
An easy computation implies that $\Phi^{-1} (y)= \varphi^{-1} (y^2 + \alpha )$. Since $\alpha \in [0, \infty)$, we easily get $g_* (r) \leq g(r)$, $r \in (0, r_0)$. For any $n \in \N$, we then set $\lambda_n \! =\!  \Phi^{-1} ( (4\rho_n )^{-1} \log \log 1/4\rho_n))$. Then, observe that 
$$\lambda_n g(4\rho_n)\geq \lambda_n g_*(4\rho_n)=\log\log(1/4\rho_n) \; .$$ 
Next note that for all $a, x \! \in [0, \infty)$, $(1-e^{-a}) \un_{\{ 0 \leq x \leq a\}} \geq e^{-x} -e^{-a}$, which easily entails 
$$ \P_0 \big(S_{\rho_n} \leq g(4 \rho_n) \big) \geq
 \frac{\exp (-\rho_n \Phi (\lambda_n ))  -\exp (-\lambda_n g (4 \rho_n) )}{1-\exp (-\lambda_n g(4 \rho_n))}  \underset{^{n \rightarrow \infty}}{\sim} (\log 1/(4\rho_n))^{-\frac{1}{4}}  \; .   \; $$
By the second assumption in (\ref{miam}), $\sum_{n \geq 0} (\log 1/(4\rho_n))^{-\frac{1}{4}} \! = \infty$, which proves the first point of the lemma.

Let us prove the second point. By a standard Markov inequality, we get 
$$  \P_0 \big(S_{\rho_{n+1}} \geq g_*(4 \rho_n) \big) \leq \frac{1-\exp (-\rho_{n+1} \Phi (\lambda_n)) }{1-\exp (-\lambda_n g_* (4 \rho_n))} \leq \frac{\rho_{n+1} \Phi (\lambda_n) }{1-\exp (-\lambda_n g_* (4 \rho_n))} . $$
First observe that $1\!-\! \exp (-\lambda_n g_* (4 \rho_n)) \! = \! 1\! -\!  (\log 1/4\rho_n)^{-1} \! \longrightarrow 1$, as $n\rightarrow \infty$. 
By (\ref{miam}), there exists a constant $\Cl[c]{miamloc} \! \in \! (0, \infty)$ such that 
 $$ \rho_{n+1} \Phi (\lambda_n) = \frac{\rho_{n+1}}{4\rho_n} \log \log 1/ (4\rho_n) \leq \Cr{miamloc} e^{-n} \log n \; , $$
Thus, $\sum_{n \geq 0} 
\P_0 \big(S_{\rho_{n+1}}\!  \geq \! g(4 \rho_n) \big)\leq \sum_{n \geq 0}\P_0 \big(S_{\rho_{n+1}} \! \geq \! g_*(4 \rho_n) \big) < \infty$, which completes the 
proof by the Borel-Cantelli lemma. \cqfd
\end{proof}

\subsection{Estimates on hitting probabilities.}
\label{sectionhitting}
As already mentioned in (\ref{Sheucond}), the total range ${\bf R}$ of a $\psi$-super Brownian motion 
is bounded if the starting measure  $\mu$ has compact support and if 
\begin{equation}
\label{Wcont}
\int_1^\infty \frac{\ddr b}{\sqrt{\int_1^b \psi (a) \ddr a}} < \infty.
\end{equation}
Observe that if $\bdelta \! >\! 1$, then $\bgamma \! >\! 1$ 
and (\ref{Wcont}) holds true, which allows to define the following function 
\begin{equation}
\label{Idef}
 \forall v \in (0, \infty)   \, , \quad I(v)= \int_v^\infty \frac{\ddr b}{\sqrt{\int_v^b \psi (a) \, \ddr a}}=\int_0^\infty \frac{\ddr b}{\sqrt{\int_0^b \psi (v+a) \,  \ddr a}}  \; .
 \end{equation}
This function is clearly decreasing and continuous and it plays a role in the proof of an upper bound of the hitting probabilities of the $\psi$-L\'evy snake. We first need the following elementary lemma. 
\begin{lem}
\label{kellerdelta} 
Assume $\bdelta \! >\! 1$. Then, there exists $\Cl[c]{ckellerdelta} \! \in \! (0,\infty)$ that only depends on $\psi$ such that for all 
$v \! \in \! (1, \infty)$ and all $r \! \in \! (0, \infty)$ satisfying $r \! \leq \! I(v)$, we have 
$$ v \leq \psi^{\prime -1} \big( 4\Cr{ckellerdelta} r^{-2} \big) \; .$$
\end{lem}
\begin{proof}
By an elementary change of variable, we get 
$$ I(v) = \int_1^{\infty}\frac{\sqrt{v}\, \ddr b}{\sqrt{\int_1^b \psi\left(va\right)\ddr a}} \; .$$
Fix $c \in (1, \bdelta)$. 
By the definition (\ref{defdeltaintro}) of $\bdelta$, there exists $C \!  \in \! (0, \infty)$ such that $\psi (va) \!  \geq \! C \psi (v) a^c $, for any $a, v \! \in \! (1, \infty)$. 
Let $v \! \in \! (1, \infty)$ and $r \! \in \! (0, \infty)$ be such that $r \! \leq \! I(v)$. Then, 
$$ r \leq I(v) \leq C^{-1/2} (v/ \psi (v))^{1/2}  \int_1^{\infty}\frac{\ddr b}{\sqrt{\int_1^b a^c\ddr a}} =: \big(\Cr{ckellerdelta} v/ \psi (v) \big)^{1/2} , $$
which implies the desired result since $\psi^\prime (v) \leq 4 \psi (v)/v $ by (\ref{convexpsi}). \cqfd 
 \end{proof}

\medskip

Recall from (\ref{rangeWdef}) the definition of the range $\cR$ of the snake. Recall from (\ref{snex}) the notation $\N_x$ 
for the excursion measure of the snake starting from $x$. Let $r\! \in \! (0, \infty)$. Recall that $B(0, r)$ stands for the open 
ball in $\R^d$ with radius $r$ and center $0$. Then, we set for all $x \! \in \! B(0, r)$
\begin{equation}
\label{vrdef}
 v_r (x) = \N_x(\cR\cap B(0,r)^c\neq\emptyset) \; .
\end{equation} 
From \cite{DuLG} p.~121 and p.~131, we know that 
$$ \forall x\in B(0, r), \quad v_r(x) \! \in\! (0,\infty) \quad \textrm{and} \quad 
\lim_{\quad \lVert x \rVert \rightarrow r-} \!\!\!\!\!\! v_r (x) \! = \! \infty \; .$$
\vspace{-4mm}

\noi
Moreover, $v_r$ is $C^2$ on $B(0, r)$ and it satisfies $\frac{1}{2}\Delta v_r = \psi (v_r)$. As an easy consequence of Brownian motion isotropy, $v_r$ is a radial function: 
namely, $v_r (x)$ only depends on $\lVert x\rVert$ (and $r$). Therefore, one can derive estimates on $v_r$ by studying the associated ordinary differential equation corresponding to the radial function, as done by Keller in \cite{Ke57}, p.~507 inequality (25), who proves the following:  
\begin{equation}
\label{Keller}
\forall r \in (0, \infty), \quad \frac{2}{\sqrt{d}}r  \leq  I(v_r (0) ) \leq 2r \; .
\end{equation} 
We use this bound as follows. For any $r\in (0, \infty)$ and any $x \in B(0,r)^c  $, recall that  
\begin{equation}
\label{urdef}
 u_r (x) =\N_x (\cR \cap \overline{B}(0,r) \neq \emptyset ) \; .
\end{equation} 
From \cite{DuLG} p.~121 and p.~131, we know that $\ur(x)\in(0,\infty)$ for all $x \in B(0,r)^c  $, that 
\begin{equation}
\label{urbord}
\textrm{$u_r$ is radial,} \quad \lim_{\lVert x\rVert \rightarrow\infty} u_r (x)= 0 \quad {\rm and} 
 \quad \lim_{\lVert x\rVert \rightarrow r+} u_r (x)= \infty \; .
\end{equation}
Recall that we denote by $\widetilde{u}_r$ the radial function yielded by $u_r$, namely:  
\begin{equation}
\label{uradef}
\forall x \in \overline{B} (0, r)^c, \quad \widetilde{u}_r (\lVert x \rVert )= u_r (x) \; .
\end{equation}
Moreover, $u_r$ is $C^2$ in $B(0,r)^c$ and it satisfies 
\begin{equation}
\label{urEDP}
\frac{_1}{^2}\Delta u_r=\psi(u_r) \quad \textrm{in $B(0,r)^c$.}
\end{equation}
We shall use several times the following upper bound of $u_r$. 
\begin{lem}
\label{estppale}
Assume that $\bdelta \! >\! 1$ and that $d\! \geq \! 3$. Let $\varrho \! \in \! (0, \infty)$. Then there exist $\Cl[r]{restppale}, \Cl[c]{cestppale}\! \in \! (0, \infty)$, that only depend 
on $\psi, d$ and $\varrho$, such that 
$$ \forall r \in (0, \Cr{restppale}), \;  \forall x\in B\big(0, (1\! +\! \varrho )r \big)^c , \quad  \ur(x)\leq \big((1\!+ \! \varrho) r / \lVert x \rVert \big)^{d-2}  \invDpsi (\Cr{cestppale} \, r^{-2} ) \; .$$
\end{lem}

\begin{proof} Let $y \! \in \! \R^d$ be such that $\lVert y \rVert = (1+\varrho) r$. First note that 
$$\N_y\left(\cR\cap\overline{B}(0,r)\neq\emptyset\right)\leq\N_y \left(\cR\cap B(y,\varrho r)^c \neq\emptyset\right) \; .$$
By translation invariance of Brownian motion, the right member of the previous inequality does not depend on $y$ and we get 
$u_r (y) \! \leq \! v_{\varrho r} (0)$, where $v_{\varrho r}$ is defined by (\ref{vrdef}). For any $x \! \in \! B(0, (1\! +\! \varrho)r)^c$, we next set 
$w(x)= u_r (x)-v_{\varrho  r} (0) ((1\!+\! \varrho)r/ \lVert x\rVert )^{d-2}$ that is clearly subharmonic. The previous upper bound implies that 
$w(y) \! \leq \! 0$, if $\lVert  y\rVert \! = \! (1\! + \! \!\! \varrho) r$. By (\ref{urbord}), 
$\lim_{\lVert x \rVert \rightarrow \infty} w(x)\! = \! 0$ and by the maximum principle, we get that $w\leq 0$ on $B(0, (1\! +\! \varrho)r)^c$. Namely, 
\begin{equation}
\label{ursubharm}
\forall r \in (0, \infty) , \quad  \forall x \in B \big(0, (1+\varrho )r \big)^c  \, , \quad u_r (x) \leq v_{\varrho r} (0) \big((1\!+ \! \varrho) r / \lVert x \rVert \big)^{d-2} \; .
\end{equation}
Since $\bdelta \! >\! 1$, (\ref{Wcont}) is satisfied and the function $I$ given by (\ref{Idef}) is well-defined; we easily check that $I (v) \rightarrow 0$ iff $v\rightarrow \infty$. 
Then, (\ref{Keller}) implies that $\lim_{r\rightarrow 0}v_{\varrho r} (0)=\infty$, so we can find $\Cr{restppale} \! \in \! (0, \infty)$ such 
that for all $r \! \in \!  (0, \Cr{restppale})$, $v_{\varrho r}(0) \! \geq \! 1$ and by the left member of (\ref{Keller}) we also have $2r \varrho /\sqrt{d} \leq I(v_{\varrho r} (0))$. 
Thus, Lemma \ref{kellerdelta} applies and asserts that $v_{\varrho r} (0)\leq \invDpsi (\Cr{cestppale}r^{-2} )$, where $\Cr{cestppale}:= \Cr{ckellerdelta} d \varrho^{-2}$, which completes the proof thanks to (\ref{ursubharm}). \cqfd 
\end{proof}

\bigskip

We use the previous lemma to get an upper bound of the expectation of a specific additive functional of the Brownian motion that involves $u_r$. 
More precisely, for any $r \! \in \! (0, \infty)$, we define 
\begin{equation}
\label{qdef}
q_r= \invDpsi (\Cr{cestppale}r^{-2} ) \; .
\end{equation}
Recall that $\Cr{cestppale}$ is the constant appearing in Lemma \ref{estppale}. 
Note that for any $0 \! < \! r \! < \! (\Cr{cestppale}/ \Dpsi(1) )^{1/2}$, we have $q_{r} \! \geq \! 1$. 
We then define 
\begin{equation}
\label{Jdef}
\forall r\in \big(0, (\Cr{cestppale}/ \Dpsi(1) )^{\frac{1}{2}} \big),  \quad J(r)=r^2q_r^{\frac{2}{d-2}}\int_1^{q_r} 
\Dpsi(v)v^{-\frac{d}{d-2}}\ddr v \; .
\end{equation}
Recall that $\xi= (\xi_t)_{t\geq 0}$ stands for a standard Brownian motion starting from $0$ that is 
defined on the auxiliary probability space $(\Omega , \mathcal{F} , \P_0)$. 
We next prove the following lemma.
\begin{lem}
\label{expectaddi}
Assume that $\bdelta \! >\! 1$ and that $d \! \geq \! 3$. Let $a \! \in \! (0, \infty)$. Then, there exist $\Cl[c]{cexpectaddi}, \Cl[c]{c1expectaddi} , \Cl[r]{rexpectaddi}\! \in \! (0, \infty)$ 
that only depend on $\psi$, $d$ and $a$, such that  
$$\forall r \in (0, \Cr{rexpectaddi})\quad  \E_0 \! \left[\int_0^{2a} \!\! \ddr s \, \un_{\left\{\lVert\xi_s\rVert\geq 2r\right\}}\Dpsi(\ur(\xi_s))\right]  \leq 
\Cr{cexpectaddi}+ \Cr{c1expectaddi} J(r) \; .$$
\end{lem}
\begin{proof} To simplify notation, we denote by $L$ the expectation in the left member of the previous inequality. Recall from (\ref{uradef}) the notation $\widetilde{u}_r$ for the radial function yielded by $u_r$. By Fubini and easy changes of variable, we have the following. 
\begin{eqnarray*}
L& = & \int_{B(0, 2r)^c} \!\!\!\! \!\!\!\! \!\!\! \ddr x \;  \int_0^{2a} \!\! \ddr s \, (2\pi s)^{-d/2} e^{-\lVert x \rVert^2/ 2s} \psi^\prime \big( \widetilde{u}_r (\lVert x\rVert )\big) \\
& =& \Cl[c]{gloupi} \int_{2r}^\infty  \!\! \ddr y \, y^{d-1}  \psi^\prime \big( \widetilde{u}_r (y) \big) \! \int_0^{2a} \!\! \ddr s \, s^{-d/2} e^{-y^2/ 2s} \\
& =&  \Cl[c]{gloup}  \int_{2r}^{\infty} \!\!  \ddr y \, y f(y)  \Dpsi \big(\widetilde{u}_r (y)  \big)   \, , 
\end{eqnarray*}
where for any $y \in (0, \infty)$, we have set $f(y) \! = \! \int_{y^2/(4a)}^{\infty}\ddr u \,  u^{d/2-2}e^{-u}$ and where $\Cr{gloupi}$ and $\Cr{gloup}$ are constants that only depend on $d$. Since we assume that $d\! \geq \! 3$, $f(0)$ is well-defined and finite, and it is easy to check that $\int_{0}^\infty  y f(y) \ddr y < \infty$.

We next use Lemma \ref{estppale} with $\varrho \! = \! 1$ to get for all $r \! \in \! (0,\Cr{restppale})$ and all $y \! \in \! (2r, \infty)$ that 
$\widetilde{u}_r (y) \! \leq \! (2r/y)^{d-2} q_r $. We then set $\alpha_r \! = \! 2r  q_r^{1/(d-2)}$. Thus, 
$\widetilde{u}_r (y) \! \leq \! (\alpha_r/y)^{d-2}$, for all $r \! \in \! (0,\Cr{restppale})$ and all $y \! \in \! (2r, \infty)$. 
We next set 
$$\Cr{rexpectaddi}:= \Cr{restppale} \wedge (\Cr{cestppale}/ \Dpsi (1))^{1/2} \; . $$
Observe that for any $r\! \in \! (0, \Cr{rexpectaddi})$, 
$q_{r} \! \geq \! 1$, which implies that 
$\alpha_r \geq 2r$. Next, observe that for all $r \! \in \! (0, r_2)$ and all $y \! \in \! (\alpha_r , \infty)$, $(\alpha_r/y)^{d-2}\!\!  \leq \! 1$. Thus, 
$\Dpsi ( \widetilde{u}_r (y) ) \leq \Dpsi (1)$. It implies 
\begin{eqnarray}
\label{rognon}
 L  & \leq &   \Cr{gloup}  \Dpsi (1) 
\int_{\alpha_r}^\infty  \!\! y f (y)  \, \ddr y  \; + \;  \Cr{gloup} \int_{2r}^{\alpha_r} \!\!\! y  f(y)  \Dpsi \big((\alpha_r/y)^{d-2} \big)  \, \ddr y  \; \nonumber \\
& \leq &  \Cr{cexpectaddi} \; + \;   \Cr{gloup} f (0) \!\! \int_{2r}^{\alpha_r} \!\!\! y  \Dpsi \big( (\alpha_r/y)^{d-2} \big) \,  \ddr y \,  , 
\end{eqnarray}
where $ \Cr{cexpectaddi} \! : = \! \Cr{gloup} \Dpsi (1)  \int_{0}^\infty  y f(y) \ddr y $. By using the change of variable 
$v\! =\! (\alpha_r/y)^{d-2}$ we get 
$$ \Cr{gloup} f (0) \!\! \int_{2r}^{\alpha_r} \!\!\! y  \Dpsi \big( (\alpha_r/y)^{d-2} \big) \,  \ddr y = \frac{_1}{^{d-2}}\Cr{gloup} f (0) \alpha_r^2 \int_1^{(\alpha_r/2r)^{d-2}} \!\!\!\!\!\!\! \!\!\!\!\!\!\! \Dpsi(v) v^{-\frac{d}{d-2}}\, \ddr v = \Cr{c1expectaddi} J(r)  \; , $$
where we have set $\Cr{c1expectaddi}\! := \! \frac{4}{d-2}  \Cr{gloup} f (0)$. Then, the desired result follows from  
(\ref{rognon}). \cqfd 
\end{proof}

\bigskip

When $d$ is greater than $\frac{2\bdelta}{\bdelta-1}$, the function $J$ is \emph{bounded} for all small values of $r$ as proved in the following lemma. 
\begin{lem}
\label{Jbornee}
Assume that $\bdelta \! >\! 1$ and that $d \! >\! \frac{2\bdelta}{\bdelta-1}$. Then, there exists a constant $\Cl[c]{cJbornee} \! \in \! (0,\infty)$ that depends on $d$ and $\psi$ such that $ J(r) \! \leq \! \Cr{cJbornee}$ for all $r\! \in \! (0, \Cr{rexpectaddi})$. 
\end{lem}
\begin{proof}
Observe that $\frac{2}{d-2} \! <\! \bdelta\!-\!1$. Recall that $\delta_{\Dpsi}\! =\! \bdelta \! -\! 1$. Let us fix $u \! \in \! (\frac{2}{d-2},\bdelta\!-\!1)$. 
By the definition (\ref{defdeltaintro}) of the exponent $\delta_{\Dpsi}$, there exists $K\! \in \! (0,\infty)$ depending on $\psi$ and $u$ such that 
\mbox{$\forall\  1 \! \leq \! \lambda \! \leq \! \mu$}, $ \Dpsi(\lambda) \! \leq \! K\, \Dpsi(\mu)(\lambda/ \mu)^{u}.$ Recall from (\ref{qdef}) that $\Dpsi(q_r) \! = \! \Cr{cestppale}r^{-2}$, where $\Cr{cestppale}$ is the constant appearing in Lemma \ref{estppale}. Then we get the following. 
 \begin{align*}
J(r)=r^2q_r^{\frac{2}{d-2}}\int_1^{q_r} 
\Dpsi(v)v^{-\frac{d}{d-2}}\ddr v
&\leq  r^2q_r^{\frac{2}{d-2}}\Dpsi(q_r)\int_1^{q_r} 
\left(v/q_r\right)^{u}v^{-\frac{d}{d-2}}\ddr v\nonumber\\
&\leq \Cr{cestppale}q_r^{\frac{2}{d-2}-u}\int_1^{q_r} 
v^{u-\frac{2}{d-2}-1}\ddr v \nonumber \\ 
& \leq \Cr{cestppale}q_r^{\frac{2}{d-2}-u}\int_0^{q_r} 
v^{u-\frac{2}{d-2}-1}\ddr v = \frac{\Cr{cestppale}}{^{u-\frac{2}{d-2}}}\; , 
\end{align*}
which implies the desired result with $\Cr{cJbornee}:=\frac{\Cr{cestppale}}{u-\frac{2}{d-2}}$. \cqfd 
\end{proof}

\bigskip

 When $d \! \in \! (\frac{2\bgamma}{\bgamma -1},\frac{2\bdelta}{\bdelta-1}]$, we are only able to prove that 
 $\liminf_{r\rightarrow 0} J (r) \! < \! \infty$. More precisely, we prove the following lemma. 
\begin{lem}\label{lemmertheta}
Assume that $\bgamma \!> \! 1$ and that $d \! > \! \frac{2\bgamma}{\bgamma-1}$. Recall that 
$\Cr{cestppale}$ appears in Lemma \ref{estppale}. 
Then, there exists a decreasing function 
$$\theta \in (\Dpsi(1) , \infty) \longmapsto r_\theta \in \big( 0 \, , \, (\Cr{cestppale}/\Dpsi(1))^{\frac{1}{2}} \big)$$ 
such that $\lim_{\theta \rightarrow \infty} r_\theta = 0$ and such that there exists  $\Cl[c]{clemmertheta} \! \in \! (0, \infty)$, that only depends on $d$ and $\psi$, 
and that satisfies 
\begin{equation}
\label{Jbound}
\forall \theta \in (\Dpsi (1) , \infty), \quad J(r_\theta) \leq \Cr{clemmertheta} \, . 
\end{equation}
Moreover, for any $\theta^\prime\! , \theta  \! \in \! (\Dpsi (1), \infty)$ such that 
$\theta^\prime \! \geq  \! \theta $, we also have  
\begin{equation}
\label{rcontrol}
r_{\theta^\prime }/ r_\theta \leq (\theta/ \theta^\prime)^{1/2} 
\quad {\rm and} \quad r_\theta \geq \Cl[c]{c1lemmertheta} \, \theta^{-\Cl[c]{c2lemmertheta}} \; , 
\end{equation}
where $\Cr{c1lemmertheta}, \Cr{c2lemmertheta} \! \in \! (0, \infty)$ only depend on $d$ and $\psi$. 
\end{lem}
\begin{proof} Note that $\frac{2}{d-2} \! < \! \bgamma \! -\! 1 \! =\! \gamma_{\psi^{\prime}}$. Let us fix $c \! \in \! (\frac{2}{d-2} , \gamma_{\Dpsi} )$. Thus, 
$\lambda^{-c}\Dpsi(\lambda) \rightarrow \infty$ as $\lambda \rightarrow \infty$, which allows to define the following for any $\theta \! \in \! ( \Dpsi (1), \infty)$: 
\begin{equation}
\label{rthetadef}
r_\theta =\big(\Cr{cestppale} / \Dpsi (\lambda_\theta ) \big)^{\frac{1}{2}} 
\quad {\rm where} \quad 
\lambda_\theta = \inf \big\{ \lambda \in [1, \infty)  :  \lambda^{-c} \Dpsi (\lambda) = \theta  \big\} .
\end{equation}
Note that if $\theta \! >  \!  \Dpsi (1)$, then $r_\theta \! <\!  (\Cr{cestppale}/\Dpsi(1))^{\frac{1}{2}}$ and $J(r_\theta)$ is well-defined by (\ref{Jdef}). 
Clearly, 
$\theta \mapsto \lambda_\theta$ increases to $\infty$ as $\theta \rightarrow \infty$. Consequently $\theta \mapsto r_\theta$ decreases to $0$ as $\theta \rightarrow \infty$. 
Recall from (\ref{qdef}) and (\ref{Jdef}) the definitions of $q_r$ and $J (r)$ and note that $q_{r_\theta}= \lambda_\theta$. By definition, 
$\Dpsi (v) \leq \theta v^{c}$, for any $v \in [1, \lambda_\theta]$, which implies 
\begin{eqnarray*}
J(r_\theta) & = &  r_\theta^2 \, \lambda_\theta^{\frac{2}{d-2}} \!\! \int_1^{\lambda_\theta} 
\!\! \Dpsi (v) v^{-\frac{d}{d-2}} \ddr v  \leq  r_\theta^2 \, \lambda_\theta^{\frac{2}{d-2}} \, \theta \!\!  \int_1^{\lambda_\theta} \!\! v^{c-\frac{2}{d-2}-1}  \ddr v   \\
& \leq & \frac{_1}{^{c-\frac{2}{d-2}}}  r_\theta^2 \,  \lambda_\theta^{\frac{2}{d-2}} \,  \theta \,  
\lambda_\theta^{c-\frac{2}{d-2}}= \Cl[c]{floum}  \frac{\theta \lambda_\theta^c}{\Dpsi (\lambda_\theta )} = \Cr{floum},  
\end{eqnarray*}
where $\Cr{floum} \! =\!  \Cr{cestppale}/(c\! -\! \frac{2}{d-2})$ and since $\Dpsi (\lambda_\theta ) \! = \! \theta \lambda_\theta^c$, by definition. Next, observe that 
$$ \theta r^2_\theta = \frac{\Cr{cestppale} \theta }{ \Dpsi (\lambda_\theta)} =  \Cr{cestppale} \lambda_\theta^{-c}  \; .$$
Thus, $\theta \mapsto \theta r^2_\theta $ decreases, which proves the first inequality in (\ref{rcontrol}). 
To prove the second inequality, we fix $\varepsilon \!  \in \! (0, \infty)$ such that 
$c+ \varepsilon \! < \! \gamma_{\Dpsi}$. By definition of $\gamma_{\Dpsi}$, there exists $K \! \in \! (0, \infty)$ such that $\lambda^{-c} \Dpsi (\lambda) \geq K \lambda^{\varepsilon}$, for any $\lambda \in [1, \infty)$. 
It entails that $\theta = \lambda_\theta^{-c} \Dpsi (\lambda_\theta )\geq K \lambda_\theta^{\varepsilon}$. Thus, 
$$ r_\theta =\big(\Cr{cestppale} / \Dpsi (\lambda_\theta ) \big)^{\frac{1}{2}} = \big(\Cr{cestppale}  / (\theta \lambda_\theta^c)  \big)^{\frac{1}{2}} \geq  \big(\Cr{cestppale}  K^{\frac{c}{\varepsilon}} \big)^{\frac{1}{2}} \theta^{-\frac{1}{2} (1+\frac{c}{\varepsilon} )}, $$
which implies the desired result with $\Cr{c1lemmertheta}= \big(\Cr{cestppale}  K^{\frac{c}{\varepsilon}} \big)^{\frac{1}{2}} $ and $\Cr{c2lemmertheta}= \frac{1}{2} (1+\frac{c}{\varepsilon} )$.  \cqfd 
\end{proof}

\bigskip

By combing the previous lemmas we obtain the following result. 
\begin{lem}  
\label{lesbonsradis} Assume that $\bdelta \! >\! 1$ and that $d \! >\! \frac{2\bgamma}{\bgamma-1}$. Then $d\! \geq  \! 4$. 
Let $a \! \in \! (0, \infty)$. For any $n \in \N$, set $\rho_n = r_{e^{n^2}}$, where $(r_\theta)_{\theta \in [\Dpsi (1), \infty)}$ is defined as in Lemma \ref{lemmertheta}.  
Then, the sequence $(\rho_n)_{n \geq 0}$ satisfies (\ref{miam}) in Lemma \ref{lemmesubord}. Moreover, 
there exists a constant $\Cl[c]{clesbonsradis} \! \in \! (0, \infty)$, that only depends on $d$, $\psi$ and $a$, such that for all sufficiently large $n\in \N$, 
$$ \E_0\left[\int_0^{2a} \ddr s \un_{\left\{\lVert\xi_s\rVert\geq 2\rho_n \right\}}\Dpsi( u_{\rho_n} 
(\xi_s))\right]  \leq \Cr{clesbonsradis} \; .$$
\end{lem}
\begin{proof} Note that $\bgamma \!  \leq \! 2$, which easily entails that $d\! \geq  \! 4$. By (\ref{rcontrol}) in Lemma \ref{lemmertheta}, 
$$\rho_{n+1} / \rho_n  \leq \big( e^{n^2}/e^{(n+1)^2} \big)^{1/2}= e^{-n -\frac{1}{2}} \leq e^{-n} \quad \textrm{and} \quad n^{-2} \log 1/\rho_n \leq \Cr{c2lemmertheta} -n^{-2} \log \Cr{c1lemmertheta} \; , $$
which proves that $(\rho_n)_{n\geq 0}$ satisfies (\ref{miam}) in Lemma \ref{lemmesubord}. Moreover, for all $n\! \in \! \N$ such that $\rho_n \! \in (0, \Cr{rexpectaddi})$, Lemma \ref{expectaddi} and (\ref{Jbound}) in Lemma \ref{lemmertheta} imply 
$$  \E_0\left[\int_0^{2a} \ddr s \un_{\left\{\lVert\xi_s\rVert\geq 2\rho_n \right\}}\Dpsi( u_{\rho_n} 
(\xi_s))\right]  \leq \Cr{cexpectaddi} + \Cr{c1expectaddi} J(r_{e^{n^2}}) \leq  \Cr{cexpectaddi} + \Cr{c1expectaddi}\Cr{clemmertheta}=: \Cr{clesbonsradis} \; , $$
which completes the proof of the lemma. \cqfd 

\end{proof}

\subsection{The spine and the associated subordinator.}
\label{sectionsubord}
Recall from Section \ref{levysnake}, the Palm formula for the occupation measure of the snake. 
Recall that $\xi= (\xi_t)_{t\geq 0}$ is 
$d$-dimensional Brownian motion starting from $0$ that is defined on $(\Omega, \cF, \P_0)$. 
Recall that $(V_t)_{ t\geq 0}$ is a subordinator defined on $(\Omega, \cF, \P_0)$ that is independent of $\xi$ and whose Laplace exponent is $\Dpsistar(\lambda)\! =\! \Dpsi(\lambda) \! -\! \alpha$. Recall from (\ref{defN}) that under $\P_0$, conditionally given $(\xi, V)$, 
$\Nstar(\ddr t \ddr W)=\sum_{j\in\mathcal{J}^*}\delta_{(t_j,W^j)}$ is a a Poisson point process on $[0,\infty) \! \times \! C(\R_+,\cW)$ with intensity $\ddr V_t \, \N_{\xi_t} (\ddr W)$. Then recall from (\ref{defMa}) that for all $a \! \in \! \R_+$, we have set $\cM_a^* \! = \! \sum_{{j\in\mathcal{J}^*}}\un_{[0,a]}(t_j)\mathcal{M}_j$ where 
for all $j\! \in \! \mathcal{J}^*$, $\cM_j$ stands for the occupation measure of the snake $W^j$ as defined in (\ref{occuWdef}).  

For any $a\! \in \! \R_+$, we then introduce 
\begin{equation}\label{defTc}
T_a:= \langle \cM^*_a , \un \rangle = \sum\limits_{j\in\cJ*}\un_{\left[0,a\right]}(t_j)\sigma_j , 
\end{equation}
where $\sigma_j$ is the total duration of the excursion of the snake $W^j$. By construction of the snake excursion measure,  
\begin{equation}
\label{rappdura}
\N_x \big[1\! -\! e^{-\lambda \sigma }\big]= N \big[1\! -\! e^{-\lambda \sigma} \big] = \psi^{-1} (\lambda) \; .
\end{equation}
We shall assume throughout the paper that $d\! \geq \! 4$. We then 
introduce the following two last exit times: for all $r \! \in \! \R_+$, we set 
\begin{eqnarray}\label{deflastexit}
\vartheta(r) &=&\sup\{t\geq0 : \lVert\xi_t\rVert\leq r\} \\
\gamma(r)&=&\sup\left\{t\geq0 : \sqrt{(\xi_{t}^{_{(1)}})^2+(\xi_{t}^{_{(2)}})^2+(\xi_{t}^{_{(3)}})^2}\leq r\right\}
\end{eqnarray}
where $(\xi_t^{(i)})_{ t\geq 0}$ stands for the $i$-th coordinate of $\xi$.
We then recall the two basic facts on the processes $\gamma$ and $\vartheta$. 
\begin{itemize}
\item[(a)] The increments of $(\vartheta(r))_{ r\geq0}$ are independents. Moreover,$(\lVert\xi_{\vartheta(r)+t}\rVert)_{t\geq0}$ is independent of the two processes $(\vartheta(r^{\prime}))_{ 0\leq r^{\prime}\leq r}$ and $(\lVert\xi_{t\wedge \vartheta(r)}\rVert)_{ t\geq 0}$.
\item[(b)] The process $(\gamma(r))_{ r\geq0}$ has independent and stationary increments: it is a subordinator with Laplace exponent $\lambda\longmapsto\sqrt{2\lambda}$. 
\end{itemize}
\noindent 
The first point is proved in Getoor \cite{Get79}. The second is a celebrated result of Pitman \cite{Pit75}.

\medskip

Before stating our lemma, we introduce the following random variables: 
\begin{equation}\label{defNatheta}
\forall \, t \geq s \geq 0,   \quad N_r (s, t)=\#\left\lbrace j\in\cJ^* : \, s \! <\! t_j\! <\! t \quad   \textrm{and} \quad  \cR_j\cap\overline{B}(0,r)\neq\emptyset\right\rbrace,
\end{equation}
\noindent that counts the snakes that are grafted on the spatial spine $\xi$ 
between times $s$ and $t$, and that hit the ball $\overline{B}(0, r)$. 
\begin{lem}\label{lemmeTc} Assume that $d\! \geq \! 4$. 
We keep the previous notation. Then, the following holds true.  
\begin{itemize} 
\item[(i)] For all real numbers $r\! >\! r^\prime \! >\! \rho\! >\! \rho^\prime \! >\! 0$ 
and all $a \! \in \! (0,\infty)$, the random variables 
$$T_{\vartheta(2\rho)} \! -\! T_{\vartheta(2\rho^\prime )}\, , \quad  T_{\vartheta(2r)} \! -\! T_{\vartheta(2r^\prime )} \quad \textrm{and} \quad 
N_r(\vartheta (2r), a\! +\!  \vartheta (2r)) $$
are independent. 
\item[(ii)] The process $(T_{\gamma(r)})_{ r\geq 0}$ is a subordinator with Laplace exponent 
$\sqrt{2\, \Dpsistar \! \circ \! \psi^{-1}}$.
\end{itemize}
\end{lem}

\begin{rem}
We take the opportunity to mention that it is stated incorrectly in \cite{Duq09} (60),  
that $T_{\gamma(2\rho)} \! -\! T_{\gamma(2\rho')}$, $T_{\gamma(2r)} \! -\! T_{\gamma(2r')}$ and 
$N(\vartheta (2r), a )$ are independent. More precisely, the statement (49) in \cite{Duq09} is incorrect. 
We provide here a correct statement and a correct proof. 
\end{rem}
\begin{proof}
Let us prove $(i)$. Let $\lambda, \mu, \theta \! \in \! \R_+$. To simplify notation we set 
$$ Y=  \exp \big( \! \! - \! \lambda(T_{\vartheta(2r)}  -\! T_{\vartheta(2r^\prime )})  -\! \mu (T_{\vartheta(2\rho)} \! -\! T_{\vartheta(2\rho^\prime )}) \! -\! \theta N_r(\vartheta (2r), a\! +\!  \vartheta (2r)) \,  \big) . $$
Then recall (\ref{rappdura}) and recall from (\ref{uradef}) 
notation $\widetilde{u}_r(\lVert x\rVert )= \N_x (\cR \cap \overline{B} (0, r) \neq \emptyset)$. Then, basic results on Poisson processes imply 
$$ \E_0 \big[ Y   \big|  (\xi, V)\big] \!\!  =\!  \exp \Big(  \! -\!  \psi^{\! -1} \! (\lambda) (V_{\vartheta(2r)}  -\! V_{\vartheta(2r^\prime )})  -\! \psi^{\! -1} \! (\mu) (V_{\vartheta(2\rho)} \! -\! V_{\vartheta(2\rho^\prime )})  -\! (1\! - e^{-\theta}) \!  \!\! \int_{(0, a)} \!\!\!\! \!\!\!\!  \ddr V_t   
\widetilde{u}_r ( \lVert \xi_{\vartheta(2r)+t }\rVert )  \Big). $$
Since $V$ is a subordinator with Laplace exponent $\psi^{* \prime}$, we then get 
\begin{align*}
  \E_0 \big[ Y  \, \big| \,\xi\big] 
 =  \exp \Big(  \! - (\vartheta(2r)  &-\! \vartheta(2r^\prime )) \psi^{* \prime}(\psi^{-1} (\lambda)) \\
  & - (\vartheta(2\rho) \! -\! \vartheta(2\rho^\prime )) \psi^{* \prime}(\psi^{-1} (\mu))  -\! 
\int_{0}^a \!\! \ddr t \, \psi^{*\prime} \big( (1\! - e^{-\theta})  \widetilde{u}_r ( \lVert \xi_{\vartheta(2r)+t } \rVert ) \big) 
 \Big) . 
 \end{align*}
The above mentioned property (a) for last exit times then implies that 
$$  \vartheta(2r)  -\! \vartheta(2r^\prime ), \quad \vartheta(2\rho) \! -\! \vartheta(2\rho^\prime ) \quad \textrm{and} \quad \int_{0}^a \!\! \ddr t \, \psi^{*\prime} \big( (1\! - e^{-\theta})  \widetilde{u}_r ( \lVert \xi_{\vartheta(2r)+t } \rVert )\big)  $$
are independent which easily implies $(i)$. 

The second point is proved in a similar way: let $0\! =\! r_0 \! <\! r_1\! < \! \ldots \! <\! r_n$ and let $\lambda_1,\ldots\lambda_n \! \in \R_+$. We set 
$$Y^\prime= \exp \Big(\! -\!\!\! \sum_{1\leq k \leq n} \!\! \lambda_k \, \big(T_{\gamma(r_k)}\! -\! T_{\gamma(r_{k-1})}  \big) \Big) $$
Thus, 
\begin{eqnarray*}
 \E_0 \big[ Y^\prime \big]  & = & \E_0 \Big[ \exp \Big(\! -\!\!\!  \sum_{1\leq k\leq n} \!\! 
 \psi^{-1} (\lambda_k) \,  \big( V_{\gamma (r_k)} \! -\! V_{\gamma (r_{k-1})}\big) \Big)\Big] \\
 & =&  \E_0 \Big[ \exp \Big(\! - \!\!\! \sum_{1\leq k\leq n} \!\! \psi^{*\prime} (\psi^{-1} (\lambda_k)) 
 \, \big( \gamma (r_k) \! -\! \gamma (r_{k-1})\big)  \Big)\Big]   \\
 & =& \exp \Big(\! -\!\!\!  \sum_{1\leq k\leq n} \!\! (r_k \! -\! r_{k-1}) \sqrt{2 \psi^{*\prime} (\psi^{-1} (\lambda_k))} 
\, \Big) \; . 
\end{eqnarray*}
Indeed, the first equality comes from basic results on Poisson point measures, the second equality comes from the fact that $V$ is a subordinator with Laplace exponent $\psi^{*\prime}$ and the last equality is a consequence of the above mentioned Property (b) of the last exit times $\gamma (r)$. This completes the proof of $(ii)$. \cqfd 
\end{proof}

\bigskip

Recall from (\ref{defNatheta}) the notation $N_r(s, t)$ 
that counts the snakes that are grafted on the spatial spine $\xi$ 
between times $s$ and $t$, and that hit the ball $\overline{B}(0, r)$. 
We state the following lemma that actually means, in some sense, 
that in supercritical dimension there is no 
snake $W^j$ grafted far away that hit $\overline{B}(0, r)$.
\begin{lem}\label{lemmeJars}
Assume that $\bgamma \! >\! 1$ and that $d \! >\! \frac{2\bgamma}{\bgamma-1}$. Then, for all 
$ t \! >\! s \! >\! 0$, 
$$ \forall \, t >s>0, \quad  \lim_{r\rightarrow0}\P_0 \left(N_r (s, t)=0\right)=1 \; .$$ 
\end{lem}

\begin{proof} Recall from (\ref{urdef}) the definition of $\ur$ and from (\ref{uradef}) that of 
$\urtild$. By the definition (\ref{defNatheta}), conditionally given 
$(\xi,V)$, $N_r(s, t)$ is a Poisson random variable with parameter 
$\int_s^t \ddr V_w \ur(\xi_w)$. Thus 
\begin{equation}
\label{poipoi}
 \P_0 (N_r(s, t)= 0)= \E_0 \Big[ \exp\Big( \! - \!\! \int_{s}^t \! \! \ddr w \, 
\Dpsistar(\ur(\xi_w)) \Big)  \Big] \; .
\end{equation}
Next, note that $d\! >\! \frac{2\bgamma}{\bgamma -1}$ implies that $d\!-\! 2 \! >\frac{2}{\bgamma -1}$. 
Then, there exists $ b \! \in \! (0, 1)$ and $a \! \in \! (0, \bgamma \! -\! 1)$ such that 
\begin{equation}
\label{ennuiignoble}
(d-2)(1-b) > \frac{2}{a} \; .
\end{equation}
By the definition (\ref{deflastexit}) of the last exit time process $\vartheta$, 
\begin{equation}
\label{lowporc}
\E_0 \Big[ \exp\Big( \! - \!\! \int_{s}^t \! \! \ddr w \, 
\Dpsistar(\ur(\xi_w)) \Big)  \Big] \geq \exp \big( \! -t  \Dpsistar(\urtild(r^b)) \big) \, 
\P_0 \big( s \! >\! \vartheta(r^b) \big) \; .
\end{equation}
Clearly $\lim_{r\rightarrow 0} \P_0 (s \! >\! \vartheta(r^b) ) = 1$. Thus, we only have to choose 
$b$ such that $\lim_{r\rightarrow 0} \urtild(r^b)= 0$. To that end, we apply Lemma \ref{estppale} with $\varrho\! = \! 1$: for all $r \! \in \! (0, r_1\wedge 2^{-\frac{1}{1-b}})$, we get 
\begin{equation}
\label{bleurpik}
\urtild(r^b)\leq (2r/r^b)^{d-2}\invDpsi (\Cr{cestppale} r^{-2})= 2^{d-2} r^{(d-2)(1-b)}\invDpsi (\Cr{cestppale} r^{-2}) .
\end{equation}
Since $a \! \in \! (0,\bgamma \! -\! 1)$ and since 
$\gamma_{\psi^\prime} \!= \! \bgamma \! -\! 1$, 
the definition (\ref{defgamma}) of $\gamma_{\psi^\prime}$ entails that for all sufficiently large $\lambda$, $\psi^\prime (\lambda) \! \geq \! \lambda^a $ and thus $\invDpsi(\lambda)\leq\lambda^{1/a}$. Then (\ref{bleurpik}) implies that there exists a constant $c\! \in \! (0, \infty)$ such that 
$\widetilde{u}_r (r^b) \leq c \, r^{(d-2)(1-b)-2/a}$ for all sufficiently small $r$. 
By (\ref{ennuiignoble}), this entails $\lim_{r\rightarrow 0} \urtild(r^b)= 0$, which completes the proof by (\ref{poipoi}) and (\ref{lowporc}). \cqfd 
\end{proof}

\bigskip

We next prove a similar estimate for $N_r (\vartheta (2r), a \! + \!  \vartheta (2r))$. 
\begin{lem}\label{minoprobadelta}
Assume that  $\bdelta \! >\! 1$ and that $d \! >\! \frac{2\bdelta}{\bdelta-1}$. Then,  there exist two constants $\Cl[c]{cminoprobadelta} , \Cl[r]{rminoprobadelta}\! \in \! (0, \infty)$ that only depends on $d$, $\psi$ and $a$, such that 
$$\forall r\in (0,\Cr{rminoprobadelta}) \quad   \P_0 \big(N_r (\vartheta (2r), a \! + \!  \vartheta (2r)) \! =\!  0 \big)\geq \Cr{cminoprobadelta}\; .$$
\end{lem}
\begin{proof} Recall from (\ref{urdef}) the definition of $\ur$ and from (\ref{uradef}) that of 
$\urtild$. By the definition (\ref{defNatheta}), conditionally given 
$(\xi,V)$, $N_r (\vartheta (2r), a \! + \!  \vartheta (2r))$ is a Poisson random variable with parameter 
$\int_{\vartheta (2r) }^{a + \vartheta (2r)} \ddr V_t \,  \ur(\xi_t)$. Thus 
\begin{equation}
\label{poipo}
 \P_0 \big(N_r (\vartheta (2r), a \! + \!  \vartheta (2r)) \! =\!  0 \big)= 
 \E_0 \Big[ \exp\Big( \! - \!\! \int_{\vartheta (2r) }^{a + \vartheta (2r)} \! \! \! \! \! \! \! \! \! \! \! \!  \ddr t \, \, 
\Dpsistar(\ur(\xi_t)) \Big)  \Big] \; .
\end{equation}
Next note that on $\{ \vartheta (2r) \! \leq \! a \}$, $a + \vartheta (2r) \! \leq \! 2a$ 
and that $t \! \geq \! \vartheta (2r)$ implies that $\lVert \xi_t  \rVert \! \geq \!  2r$. Thus, by~(\ref{poipo}) 
\begin{eqnarray}
\label{poip}
\hspace{-3mm}\hspace{-3mm}\hspace{-3mm} \P_0 \big(N_r (\vartheta (2r), a \! + \!  \vartheta (2r)) \! =\!  0 \big) \hspace{-3mm}& \geq &
\hspace{-2mm} \E_0 \Big[\un_{\{  \vartheta (2r)  \leq  a\}} \exp\Big( \! - \!\! \int_{0 }^{2a} \! \! \! \! \!  \ddr t \,  
\un_{\{ \lVert \xi_t \rVert \geq 2r\}} \Dpsistar(\ur(\xi_t)) \Big)  \Big]  \nonumber \\
\hspace{-3mm}& \geq & \hspace{-2mm}\E_0 \Big[ \exp\Big( \! - \!\! \int_{0 }^{2a}\! \!  \ddr t \,  
\un_{\{ \lVert \xi_t \rVert \geq 2r\}} \Dpsistar(\ur(\xi_t)) \Big)  \Big] \! -\! \P_0 ( \vartheta (2r) \! >\! a )  \nonumber \\
\hspace{-3mm}& \geq & \hspace{-2mm} \exp\Big( \! - \! \E_0 \Big[  \int_{0 }^{2a}\! \!  \ddr t \,  
\un_{\{ \lVert \xi_t \rVert \geq 2r\}} \Dpsistar(\ur(\xi_t))   \Big] \Big) \! -\! \P_0 ( \vartheta (2r) \! >\! a ),   
\end{eqnarray}
where we use Jensen inequality in the last line. By Lemma \ref{expectaddi} and Lemma \ref{Jbornee}, for any 
$r\! \in \! (0, \Cr{rexpectaddi})$, 
$$ \exp\Big( \! - \! \E_0 \Big[  \int_{0 }^{2a}\! \!  \ddr t \,  
\un_{\{ \lVert \xi_t \rVert \geq 2r\}} \Dpsistar(\ur(\xi_t))   \Big] \Big) \geq \exp (-\Cr{cexpectaddi} - \Cr{c1expectaddi} \Cr{cJbornee}) \; .$$ 
Since $ \P_0 ( \vartheta (2r) \! >\! a ) \! \rightarrow \! 0$ as $r\! \rightarrow \! 0$, there exists $\Cr{rminoprobadelta} \! \in \! (0, \Cr{rexpectaddi})$ such that 
$\P_0 ( \vartheta (2\Cr{rminoprobadelta}) \! >\! a ) \! \leq \!  \frac{1}{2}e^{-\Cr{cexpectaddi} - \Cr{c1expectaddi} \Cr{cJbornee}}$. Thus by (\ref{poip}) and the previous inequality, for any $r\! \in \! (0, \Cr{rminoprobadelta})$,
$$ \P_0 \big( N_r (\vartheta (2r), a \! + \!  \vartheta (2r))=0\big) \geq \frac{_1}{^2} \exp (-\Cr{cexpectaddi} - \Cr{c1expectaddi} \Cr{cJbornee})=: \Cr{cminoprobadelta}\; , $$
which completes the proof of lemma. \cqfd 
\end{proof}

\bigskip

Under the less restrictive condition $d \! >\! \frac{2\bgamma}{\bgamma-1}$, we get a similar lower but only for the family of radii $(\rho_n)_{n\geq 1}$ introduced in Lemma \ref{lesbonsradis}.
\begin{lem}\label{minoproba}
Assume that $\bdelta \!  >\! 1$ and that $d \! >\! \frac{2\bgamma}{\bgamma-1}$. Let $(\rho_n)_{n\geq 1}$ be the sequence introduced in Lemma \ref{lesbonsradis}.
Then, there exists a constant $\Cl[c]{cminoproba} \! \in \! (0, \infty)$,  
that only depend on $d$, $\psi$ and $a$, such that 
$$\forall n\geq 1 \quad \P_0 \left(N_r (\vartheta (2\rho_n), a \! + \!  \vartheta (2\rho_n))\! = \! 0\right)\geq \Cr{cminoproba} \; .$$
\end{lem}
\begin{proof} The lower bound (\ref{poip}) applies for $r= \rho_n$. Then, Lemma \ref{lesbonsradis}, 
entails that for all sufficiently large $n$, 
$$ \exp\Big( \! - \! \E_0 \Big[  \int_{0 }^{2a}\! \!  \ddr t \,  
\un_{\{ \lVert \xi_t \rVert \geq 2\rho_n \}} \Dpsistar(u_{\rho_n}(\xi_t))   \Big] \Big) \geq \exp (-\Cr{clesbonsradis} ) \; , $$
and we completes the proof arguing as in the proof of Lemma \ref{minoprobadelta} with $\Cr{cminoproba}:= \frac{1}{2}e^{-\Cr{clesbonsradis}}$.  \cqfd 
\end{proof}

\bigskip

We then shall need the following result that is used in the proof of Lemma \ref{upperbound}.
\begin{lem}
\label{lemmeTgamma} Assume that $\bdelta \! >\! 1$ and that $d\! \geq \! 4$. 
Let $(\rho_n)_{n\geq 1}$ be the sequence of radii introduced in Lemma \ref{lesbonsradis}. Then,
$$ (i) : \sum_{n\geq 1}\P_0\left(T_{\gamma(2\rho_n)}\leq g(8\rho_n)\right)=\infty \quad \textrm{and} \quad (ii): \textrm{$\P_0$-a.s.}\; 
  \limsup\limits_{n\to \infty}\frac{T_{\gamma(2\rho_{n+1})}}{g(8\rho_n)}<\infty \; .$$
\end{lem}

\begin{proof}
Lemma \ref{lemmeTc} shows that $T_{\gamma(\cdot)}$ is a subordinator with Laplace exponent 
$\sqrt{2\, \Dpsistar \! \circ \! \psi^{-1}}$. Then, Lemma \ref{lesbonsradis} asserts 
that the sequence $(\rho_n)_{n\geq 0}$ satisfies the conditions (\ref{miam}) in Lemma \ref{lemmesubord}, which immediately entails $(i)$ and $(ii)$. \cqfd 
\end{proof}

\bigskip

We end this section with the following result that is close to the previous one and that is used 
in the proof of Theorem \ref{thdimension}.  
\begin{lem}
\label{lemmeTgammabis}
Assume that $\bgamma\! >\! 1$. Let $u \! \in \! (0, \frac{_{2\bgamma}}{^{\bgamma-1}})$. Then, 
there exists a decreasing sequence $(s_n)_{n\geq 1}$ that tends to $0$, that only depends on $\psi$ and $u$, and that satisfies the following. 
$$ (i) : \sum_{n\geq 1} \P_0 \! \left(\, T_{\gamma(2s_n)} \! \leq \! s_n^u \, \right)=\infty 
 \quad \textrm{and} \quad (ii): \textrm{$\P_0$-a.s.}\; \limsup\limits_{n\to \infty}\frac{T_{\gamma(2s_{n+1})}}{s_n^u}<\infty \; .$$
\end{lem}

\begin{proof}
Let $u^\prime \! \in \! (u,\frac{_{2\bgamma}}{^{\bgamma-1}})$. 
We set $\varphi^*\! =\! \psi^{*\prime}\circ\psi^{\!-1}$, where $\psi^*(\lambda) \! =\! \psi(\lambda)\! -\! \alpha$. 
Recall from (\ref{defgamma}) that $\gamma_\varphi \! =\! \gamma_{\varphi^*} \! =\! \frac{{\bgamma-1}}{{\bgamma}}$. We fix $a\! \in \! (0, \frac{\bgamma -1}{\gamma})$. 
By the definition (\ref{defgamma}) of $\gamma_{\varphi^*}$, there exists $\lambda_0 \! \in \! (0, \infty)$ such that $\varphi^* (\lambda ) \! \geq \! \lambda^a$ for any $\lambda \! \in \! [\lambda_0 , \infty)$. Next observe that $2/u^\prime \! >\! \frac{{\bgamma-1}}{{\bgamma}}=\gamma_{\varphi^*}$. 
As an easy consequence of definition (\ref{defgamma}) of $\gamma_{\varphi^*}$, we get 
$\liminf_{\lambda \rightarrow \infty} \varphi^* (\lambda) \lambda^{-2/u^\prime} \! = \! 0$. Consequently,  there exists an increasing sequence $(\lambda_n)_{n\geq 0}$ such that 
\begin{equation}\label{lambdaphi}
\forall n\in \N, \quad 2^n \leq \lambda_n \quad \textrm{and} \quad \lambda_n^a \leq \varphi^* (\lambda_n) \leq \lambda_n^{2/u^\prime} \; .
\end{equation}
We next fix $\veps \! \in \! (0,\infty)$ such that $(1+\veps)u/u^\prime \! <\! 1$, which is possible since $u^\prime \! > \! u$. Then, we set
\begin{equation}\label{rayonyon}
\forall n\in\N,\quad s_n= \varphi^*(\lambda_n)^{-\frac{1+ \veps}{2}} \; .
\end{equation}
Since, by Lemma \ref{lemmeTc}, $(T_{\gamma(r)})_{ r\geq 0}$ is a subordinator with Laplace exponent $\sqrt{2\varphi^*}$, a Markov inequality entails  
\begin{equation}\label{markovsn}
\P_0 \! \left(\, T_{\gamma(2s_{n})} \! >\! s_n^u \, \right) \leq \frac{1-\exp\left(-2\sqrt{2}s_{n} \sqrt{\varphi^*(\lambda_n)} \right)}{1-\exp (-\lambda_n s_n^u))}\\
 \leq 2\sqrt{2}\frac{s_{n} \sqrt{\varphi^* (\lambda_n)} }{1-\exp (-\lambda_n s_n^u)} .
\end{equation}
By (\ref{rayonyon}) and the last inequality of (\ref{lambdaphi}), we get 
\begin{equation}\label{onehand}
\lambda_n s_n^u=\lambda_n\varphi^* (\lambda_n)^{\!-(1+\veps)u/2} \geq \lambda_n^{1-(1+\veps)u/u'}\underset{n\to\infty}{-\!\!\! -\!\!\! \longrightarrow }\infty.
\end{equation}
Moreover, the first two inequalities in (\ref{lambdaphi}) and (\ref{rayonyon}) imply 
\begin{equation}\label{otherhand}
s_n\sqrt{\varphi^*(\lambda_n)}=\varphi^*(\lambda_n)^{\!-\veps/2}\leq \lambda_n^{\!-a\veps/2}\leq 2^{\! - n a\veps/2} \; .
\end{equation}
Then (\ref{markovsn}), (\ref{onehand}) and (\ref{otherhand}) imply that there exists $c\! \in \! (0, \infty)$ such that 
$$\P_0\! \left(\, T_{\gamma(2s_{n})} \! >\! s_n^u\, \right) \leq c 2^{\!-na\veps /2}.$$
It immediately implies $(i)$ and $(ii)$ follows from $\P_0 \left(T_{\gamma(2s_{n+1})} \! >\! s_n^u\right)\! \leq \! \P_0 \left(T_{\gamma(2s_{n})} \! >\! s_n^u\right)$ and from 
the Borel Cantelli lemma. \cqfd 
\end{proof}

\subsection{Estimates for bad points.}
\label{sectionbadpts}
Recall from (\ref{snex}) the definition of the excursion measure $\N_0$ of the $\psi$-L\'evy snake 
$W$. Recall that the lifetime process of $W$ is the height process $H$. Namely, $\N_0$-a.e.~for all 
$t \! \in \! \R_+$, $\zeta_{W_t}\! = \! H_t$. Therefore, the duration $\sigma$ 
of $W$ under $\N_0$ is the duration of the excursion of $H$. 
Recall that $\widehat{W}$ is the endpoint process of the snake. Recall from (\ref{occuWdef}) that 
$\cM$ stands  the occupation measure of the snake, namely the random 
measure on $\R^d$ that is the image via $\widehat{W}$ of the Lebesgue measure  on $[0, \sigma]$. 
Recall from (\ref{rangeWdef}) the definition of $\cR$, the range of the snake.

Let $\lambda, r \! \in \!  (0, \infty)$. Note that $\langle \cM \rangle \! =\!  \sigma$, therefore we get  
\begin{equation}
\label{inegdirecte}
\N_0 \big[ 1\! -\! e^{-\lambda\cM\left(\Bor\right)} \big]\leq \N_0 \big[1\! -\! e^{-\lambda\sigma} \big]=N 
\big[1\! -\! e^{-\lambda\sigma} \big]=\psi^{-1}(\lambda), 
\end{equation}
where the last to equalities comes from (\ref{lifetimeexc}) and (\ref{rappdura}). 
The next lemma states a lower bound of the same kind. 
\begin{lem}\label{BP1}
There exists $\Cl[c]{cBP1} \! \in \! (0,1)$, that only on $d$, such that 
\begin{equation}\label{minoBP}
\forall\mu,r\in(0,\infty),  \qquad \frac{r^2\mu}{\invpsi(\mu)}\geq1 \quad =\!\! =\!\! \Longrightarrow \quad  \N_0 \big[1\! -\! e^{-\mu\cM(\Bor)} \big]\geq \Cr{cBP1}\invpsi(\mu).
\end{equation}
\end{lem}

\begin{proof} To simplify, set 
$q(\mu,r)\! :=\! \N_0\left[1\! -\! e^{-\mu\cM(\Bor)}\right]$. Note that $\cM(\Bor)\! =\!  \int_0^\sigma \! \ddr t \,  
\un_{ \{ \lVert\widehat{W_t}\rVert < r \}  }$. An easy argument combined with Fubini first entails the following: 
\begin{equation}
\label{gloupin}
q(\mu,r)=\mu \! \int_0^{\infty} \!\! \! \! \ddr t \,  \N_0 \Big[ \un_{ \{ t\leq\sigma \, ; \, \lVert\widehat{W_t}\rVert < r \}  }e^{-\mu\int_t^{\sigma} \! \ddr s \, \un_{ \{\lVert\widehat{W}_s\rVert < r \} } }\Big] \; .
\end{equation}
We next apply Lemma \ref{Markovpretalemploi} at the (deterministic) time $t$ (recall 
that this lemma is a specific form of the Markov property for $\ov W=(\rho,W)$). Then we get for any $t\! \in \! \R_+$: 
\begin{align}
\label{gloupi}
\N_0 \Big[ \un_{ \{ t\leq\sigma \, ; \, \lVert\widehat{W_t}\rVert < r \}  } & e^{-\mu\int_t^{\sigma} \! \ddr s \, \un_{ \{\lVert\widehat{W_s}\rVert < r \} } }\Big] \nonumber \\ 
= & \,  \N_0 \Big[\un_{ \{ t\leq\sigma \, ; \, 
\lVert\widehat{W_t}\rVert < r \} }
\exp \Big( \! -\!\! \int_{[0, H_t]} \!\!\!\!\!\!\!  \rho_t(\ddr h) \, \N_{W_t(h)} \big[ 1 - e^{-\mu\cM(\Bor)} \big] \Big) \Big]. 
\end{align}
From (\ref{inegdirecte}), we get that for all $t,h\geq 0$, $\N_{W_t(h)}\!\left[1-e^{-\mu\cM(\Bor)}\right]\leq \invpsi(\mu)$. Then, 
\begin{align}
\label{glopglop}
 \N_0 \Big[\un_{ \{ t\leq\sigma \, ; \, 
\lVert\widehat{W_t}\rVert < r \} }& 
\exp \Big( \! -\!\! \int_{[0, H_t]}   \rho_t(\ddr h) \, \N_{W_t(h)} \big[ 1 - e^{-\mu\cM(\Bor)} \big] \Big) \Big]  \nonumber \\
\geq &\,  \N_0\Big[\un_{ \{ t\leq\sigma \, ; \, \lVert\widehat{W_t}\rVert < r \} }e^{-\invpsi(\mu) \int_{[0, H_t]} \rho_t(\ddr h)} \Big]  
\end{align}
Then by (\ref{gloupin}), (\ref{gloupi}), (\ref{glopglop}) and Fubini we get 
\begin{equation}
\label{pasglop}
q(\mu,r) \geq \mu \N_0 \Big[ \int_0^\sigma \!\! \ddr t \, \un_{ \{  \lVert\widehat{W_t}\rVert < r \} }e^{-\invpsi(\mu) \int_{[0, H_t]}  \rho_t(\ddr h)}\Big] \; .
\end{equation}
We next apply  (\ref{invar-snake}) 
to the right member of the previous inequality: to that end, recall that $U\! =\! (U_a)_{a\geq 0}$ stands for a subordinator defined on $(\Omega, \cF, \P_0)$ that is independent of $\xi$ and whose Laplace exponent is $\widetilde{\psi}^* (\lambda)\! =\!  \psi (\lambda) / \lambda \! -\! \alpha  $. 
Then, (\ref{invar-snake}) to the right member of (\ref{pasglop}) entails the following:
\begin{eqnarray*}
\mu \N_0 \Big[ \int_0^\sigma \!\! \ddr t \, \un_{ \{  \lVert\widehat{W_t}\rVert < r \} }e^{-\invpsi(\mu) \int_{[0, H_t]}  \rho_t(\ddr h)}\Big] & = & \mu \int_0^{\infty} \!\!\!\! \ddr a \, e^{-\alpha a}\E_0 
\left[\un_{ \{ \lVert\xi_a\rVert < r \} }e^{-\invpsi(\mu)U_a}\right]  \\
& =& \mu \int_0^{\infty}  \!\!\!\! \ddr a \, e^{-a \mu / \psi^{-1} (\mu)} \P_0 (\,   \lVert\xi_a\rVert \! <\!   r\,  ) \; , 
 \end{eqnarray*}
since $\xi$ and $U$ are independent. Thus, (\ref{pasglop}) and a simple change of variable using the scaling property of Brownian motion, entail
\begin{equation}
\label{niglop}
q(\mu, r)  \geq \invpsi(\mu)\int_0^{\infty} \!\!\! \P_0 \left(\lVert\xi_{c}\rVert\leq r\sqrt{\mu/\invpsi(\mu)}\right)e^{-c}\ddr c \; .
\end{equation}
If $r^2\mu/ \invpsi(\mu)\!  \geq\! 1$, then (\ref{niglop}) implies (\ref{minoBP}) with $\Cr{cBP1}\! :=\! \!\int_0^{\infty}\P_0\left(\lVert\xi_c\rVert\leq 1\right)e^{-c}\ddr c$. Clearly, $\Cr{cBP1}$ only depends on $d$ and $\Cr{cBP1}\! \in \! (0, 1)$.  \cqfd 
\end{proof}

\bigskip

Before stating the next lemma, we recall a result about the first exit time from a ball for a $d$-dimensional Brownian motion. First set 
\begin{equation}\label{defchi}
\chi_{d,r}:=\inf \big\{ t\in \R_+ : \lVert\xi_t\rVert\!  = \! r \big\}.
\end{equation}
In dimension $1$, one get 
\begin{equation}\label{tempssortie1}
\forall r, \lambda \in \R_+, \quad 
\E_0 \big[ \exp(-\lambda \chi_{1,r}) \big]=\left(\cosh(r\sqrt{2\lambda})\right)^{-1}\leq 2\exp \big(\! -\! r\sqrt{2\lambda} \, \big).
\end{equation}
Indeed note that $t\mapsto\exp( \sqrt{2\lambda}\xi_t-\lambda t)$ is a martingale (see e.g.~Revuz and Yor \cite{RY}, Chapter II, (3.7)). 
In dimension $d$, observe that
$$\textrm{$\P_0$-a.s.} \quad\chi_{d,r}\geq\min\limits_{1\leq j \leq d} \inf \big\{ t\in \R+ : \sqrt{d}\ |\xi_t^{(j)}| \! =\! r \big\}, $$
where $\xi_t^{(j)}$ stands for the $j$-th coordinate of $\xi_t$. The previous inequality combined with  (\ref{tempssortie1}) then entails
\begin{equation}\label{estchi}
\forall \lambda , r \in \R_+, \quad \E_0 \big[ \exp (-\lambda\chi_{d,r}(\xi) ) \big] \leq 2d\, \exp \big(\! -\! r\sqrt{2\lambda/d} \, \big).
\end{equation}
\begin{lem}\label{BPnouveau}
Assume that $\bdelta \! >\! 1$ and that $d\! \geq \! 3$. Recall that $\Cr{restppale}$ and $\Cr{cestppale}$ are the constants appearing in Lemma \ref{estppale}. 
Then, there exist $ \Cl[c]{cBPnouveau}, \Cl[c]{c1BPnouveau} \! \in \! (0,\infty)$, that only depend on $d$ and $\psi$, such that
\begin{align*}
\forall \mu& \in(0,\infty) , \; \forall r\in (0, \Cr{restppale}) \textrm{\ such\ that \ }\frac{r^2\mu}{\invpsi(\mu)}\geq 16,\quad \forall  x\in\R^d\! \setminus \! \overline{B}(0,2r),\\
&\N_x\left[\un_{ \{\cR\cap\overline{B}(0,r)\neq\emptyset\} } \, e^{-\mu\cM(B(0,2r))}  \right] \leq\Cr{cBPnouveau}\left( r/\lVert x\rVert\right)^{d-2}\invDpsi \! \! \left(\Cr{cestppale}r^{-2}\right) \exp 
\big( \! - \! \Cr{c1BPnouveau} \, r\sqrt{\varphi(\mu)}\,  \big). 
\end{align*}
\end{lem}
\begin{proof}
We fix $\mu \! \in \! (0, \infty)$ and $x\in \R^d \backslash \overline{B}(0, 2r)$. 
To simplify notation we set 
$$ p(x, \mu, r)=  \N_x \left[ \un_{ \{\cR\cap\overline{B}(0,r)\neq\emptyset\} } \, e^{-\mu\cM(B(0,2r))}  \right]
\; .$$ 
Recall from 
(\ref{deftau}) the definition of $\tau_{r}$ that is the hitting time in $\overline{B}(0, r)$ of the snake $W$.  First recall that 
$\{  \cR \cap \overline{B} (0,r) \neq \emptyset \}= \{ \tau_{r} \! < \infty\}$ and observe that 
\begin{equation*}
\textrm{$\N_x$-a.e.~on the event $\{  \cR \cap \overline{B} (0,r) \neq \emptyset \}$,} \quad 
\int_{\tau_r}^\sigma  \un_{ \{  \lVert \widehat{W}_s  \rVert < 2r   \}} \ddr s \; \leq \;  \cM (B(0, 2r)).
\end{equation*}  
Thus, we get 
\begin{equation}
\label{inegM1}
 p(x, \mu, r) \leq 
\N_x \Big[ \un_{\{ \tau_r < \infty \}} e^{-\mu\int_{\tau_r}^\sigma  \un_{ \{  \lVert \widehat{W}_s  \rVert < 2r   \}} \ddr s } \Big].
\end{equation}
We next apply Proposition \ref{Markovpretalemploi} to the right member of (\ref{inegM1}) and to 
the stopping time $\tau_r$; thus we get
we get 
\begin{equation}
\label{inegM2}
p(x,\mu, r) \leq \N_x \Big[ \un_{\{ \tau_r < \infty \}} \exp \Big( \! -\! \! 
\int_{[0, H_{\tau_r} ]}\!\!\!\!\!\!\!\! \rho_{\tau_r} (\ddr h) \, \N_{W_{\tau_r} (h)} \big[ 1-e^{-\mu \cM(B(0, 2r))}\big]
\Big) \Big]. 
\end{equation}
Let $\w \! \in \! \mathcal{W}$ be a continuous stopped path starting from $x\! \in \! \overline{B} (0, 2r)$; 
we denote by $\zeta_\w$ its lifetime. 
We define $T_1(\w)$, $T_2(\w)$ and $T_3(\w)$ by 
 \begin{eqnarray}\label{tempsT}
T_1(\w) &=&\inf \big\{ t\in[0,\zeta_\w] : \lVert \w(t)\rVert\leq 3r/2 \big\} \nonumber  \\
T_2(\w)&=&\inf \big\{ t\in[0,\zeta_\w-T_1(\w)] : \lVert \w(t+T_1(\w))\!-\! \w(T_1(\w))\rVert >r/4 \big\}\nonumber \\
T_3(\w)&=&\inf \big\{ t\in[0,\zeta_\w] : \lVert \w(t)\rVert\leq 5r/4 \big\}
\end{eqnarray}
with the convention that $\inf \emptyset = \infty$. 
Observe that if $T_3 (\w) \! < \! \infty$, then $T_1 (\w)+T_2 (\w) \! \leq \! T_3 (\w)$. 
Moreover, since $x\!  \in \!  \overline{B}(0, 2r)^c$, $\N_x$-a.e. on the event 
$ \{\tau_r \! < \! \infty \}$, we have $T_1 (W_{\tau_r})+T_2 (W_{\tau_r}) \leq T_3 (W_{\tau_r}) \! <\!  \tau_r \! < \!  \infty $ and for any 
$ t \! \in \! [T_1 (W_{\tau_r}) , T_1 (W_\tau)+T_2 (W_{\tau_r})  ]$, the following inequality holds true: 
$$\N_{W_{\tau_r}(t)} \! \!  \left[ 1-e^{  -\mu \cM (B(0,2r)) }\right]  \geq \N_{W_{\tau_r}(t)} \! \!  \left[ 1-e^{  -\mu \cM (B(W_{\tau_r}(t) ,r/4)) }\right]= 
\N_0 \left[ 1-e^{  -\mu \cM (B(0 ,r/4)) }\right]=: \Lambda_{\mu, r}\; ,  $$ 
the last equality being a consequence of the invariance of the snake under translation; here $\Lambda_{\mu, r}$ only depends on $\mu$ and $r$. 
To simplify notation we set $T_1\! =\! T_1 (W_{\tau_r})$ and $T_2\! =\! T_2 (W_{\tau_r})$. 
An elementary inequality combined with (\ref{inegM2}) entails 
\begin{eqnarray}
\label{froom}
p(x, \mu, r)& \leq & \N_x \Big[ \un_{\{ \tau_r < \infty \}} \exp \Big( \! -\! \! 
\int_{[T_1, T_1 +T_2 ]}\!\!\!\!\!\!\!\! \!\!\!\!\!\!\!\!  \rho_{\tau_r} (\ddr h) \, \N_{W_{\tau_r} (h)} \big[ 1-e^{-\mu \cM(B(0, 2r))} \, \big]
\Big) \Big] \nonumber \\
& \leq & \N_x\Big[  \un_{ \{ \tau_r <\infty \} }
\exp \Big(\! -\! \Lambda_{\mu,r}\! \int_{[T_1 , T_1 +T_2 ]}\!\!\!\!\!\!\!\! \!\!\!\!\!\!\!\!  \rho_{\tau_r} (\ddr h) \Big)  \Big]\; .
\end{eqnarray} 
Recall from (\ref{tempsT}) the definition of $T_3$. Recall from (\ref{defgammapsi}) the definition of the function $\varpi$. 
We next apply Proposition \ref{firsthitting} with $r^\prime=\frac{5}{4}r$ to the right member of 
(\ref{froom}). Then we get 
\begin{align}
\label{froomsi}
 \N_x\Big[  \un_{ \{ \tau_r<\infty \} }& 
\exp \Big(\! -\! \Lambda_{\mu,r}\! \int_{[T_1 , T_1 +T_2 ]}\!\!\!\!\!\!\!\! \!\!\!\!\!\!\!\!  \rho_{\tau_r} (\ddr h) \Big)  \Big] \nonumber \\
=& \; \urtild(5r/4) \, \E_x \Big[ \un_{\{ T_3 (\xi) < \infty\}} \exp \Big(\! -\!\! \int_{0}^{T_3(\xi)} \!\!\!\!\!\! \ddr t\, \varpi 
\big( u_{r}(\xi_t),\Lambda_{\mu,r}\un_{_{[T_1(\xi),T_1(\xi)+T_2(\xi)]}}(t) \big) \Big)  \Big] \nonumber  \\
 \leq &\;  \urtild(5r/4) \, \E_x \Big[ \un_{\{ T_1 (\xi) < \infty\}} \exp \Big(\! -\!\! \int_{T_1(\xi)}^{T_1(\xi)+ T_2(\xi)} \!\!\!\!\!\! \ddr t\, \varpi 
\big( u_{r}(\xi_t),\Lambda_{\mu,r} \big) \Big)  \Big] \; .
\end{align}
Here recall that $\xi$ under $\P_x$ is distributed as a standard $d$-dimensional Brownian motion starting from $x$. 
The convexity of $\psi$ provides the following lower bounds for $\varpi \left(u_{r}(\xi_t),\Lambda_{\mu,r}\right)$: 
\begin{itemize}
\item if $u_{r}(\xi_t)\! \geq \! \Lambda_{\mu,r}$, then $\varphi 
\left(u_{r}(\xi_t),\Lambda_{\mu,r}\right)\! \geq \!  \Dpsi(\Lambda_{\mu,r})$;
\item if $u_{r}(\xi_t) \! < \! \Lambda_{\mu,r}$, then $\varpi \left(u_{r}(\xi_t),\Lambda_{\mu,r}\right) 
\! \geq \! \psi(\Lambda_{\mu,r})/ \Lambda_{\mu,r} \! \geq \! \frac{1}{4}\Dpsi(\Lambda_{\mu,r})$, by (\ref{convexpsi}) for the last inequality. 
\end{itemize}
These inequalities combined with (\ref{froom}) and (\ref{froomsi}), entail 
\begin{equation}\label{BPnew}
p(x,\mu, r) \leq \urtild(5r/4) \,  \E_x\left[ \un_{\{ T_1 (\xi) < \infty\}} \exp\left(-\frac{_1}{^4}\psi^\prime(\Lambda_{\mu,r})T_2(\xi) \right)\right].
\end{equation}

We now assume that 
$$\frac{r^2\mu}{16\invpsi(\mu)}\geq1 \; .$$
Recall that $\Lambda_{\mu, r}\! =\!  \N_0 [1\! -\! e^{-\mu \cM(B(0, r/4)}]$. By Lemma \ref{BP1}, $\Lambda_{\mu,r}\geq \Cr{cBP1}\invpsi(\mu)$. 
We next use the concavity of $\Dpsi$ and the fact that $\Cr{cBP1} \! \in \! (0, 1)$, to get 
\begin{equation}\label{BP6}
\Dpsi(\Lambda_{\mu,r})\geq\Dpsi(\Cr{cBP1}\invpsi(\mu))\geq \Cr{cBP1}\Dpsi(\invpsi(\mu))=\Cr{cBP1}\varphi(\mu).
\end{equation}
Recall that $\Cr{restppale}$ and $\Cr{cestppale}$ are the constants appearing in Lemma \ref{estppale}. 
We assume that $r \! \in \! (0, \Cr{restppale})$. Then, Lemma \ref{estppale} with $\varrho= 1/4$ implies that $\urtild(5r/4) \leq \psi^{\prime -1} (\Cr{cestppale} r^{-2})$. 
Thus, by (\ref{BPnew}) and (\ref{BP6}) we get 
\begin{equation}\label{BPinterm}
p(x, \mu, r) \leq \psi^{\prime -1} (\Cr{cestppale} r^{-2})\,  \E_x\left[\un_{ \{ T_1(\xi)< \infty \} }\exp(-\frac{_1}{^4}\Cr{cBP1}\varphi(\mu)T_2(\xi))\right].
\end{equation}
Recall from (\ref{tempsT}) the definition of $T_1 (\xi)$ and $T_2 (\xi)$. By the Markov property at time $T_1(\xi)$, we get 
$$\E_x \! \left[\un_{ \{ T_1(\xi)< \infty \} }\exp(-\frac{_1}{^4}\Cr{cBP1}\varphi(\mu)T_2(\xi))\right]= \P_x (T_1 (\xi) \! <\! \infty )\,  \E_0 \! \big[ \exp(-\frac{_1}{^4}\Cr{cBP1}\varphi(\mu)\chi_{d,r/4}(\xi) )\big] . $$
Since $d\! \geq \! 3$, we get $\P_x (T_1 (\xi) <\infty ) =(3r/2\lVert x \rVert)^{d-2}$. Moreover by (\ref{estchi}), we get 
$$\E_0 \big[ \exp(-\frac{_1}{^4}\Cr{cBP1}\varphi(\mu)\chi_{d,r/4}(\xi) )\big] \leq 2d \exp \big(-\frac{_1}{^8} r \sqrt{2\Cr{cBP1}\varphi(\mu)/d} \, \big) \; .$$
Then we set $\Cr{cBPnouveau} \! = \! 2d \, (3/2)^{d-2}$ and $\Cr{c1BPnouveau} \! = \! \sqrt{\Cr{cBP1}/(32d)}$ and we get  
$$ \E_x\left[\un_{ \{ T_1(\xi)< \infty \} }\exp(-\frac{_1}{^4}\Cr{cBP1}\varphi(\mu)T_2(\xi))\right]\leq  \Cr{cBPnouveau} (r/ \lVert x\rVert)^{d-2} \exp \big( \! -\!  \Cr{c1BPnouveau} r \sqrt{
\varphi(\mu)}\,  \big) , $$
which implies the desired result by (\ref{BPinterm}). \cqfd 
\end{proof}

\bigskip

Recall from (\ref{gaugedef}) the definition of the gauge function $g$: 
$$ g(r)= \frac{\log \log \frac{1}{r}}{\varphi^{-1} \big(( \frac{1}{r} \log \log \frac{1}{r})^2 \big)} \; .$$
\begin{lem}\label{BP2}
Assume they $\bdelta \! >\! 1$ and that $d\! \geq \! 3$. Recall that $\Cr{cBPnouveau}$ is the constant appearing in Lemma \ref{BPnouveau} and that $\Cr{cestppale}$ is the constant appearing in Lemma \ref{estppale}. 
There exists $\Cl[c]{cBP2}, \Cl[c]{c1BP2}, \Cl[r]{rBP2}, \kappa_0\! \in \! (0,\infty)$, that only depend $d$ and $\psi$, such that 
for all $r\! \in \! (0,\Cr{rBP2})$, for all $\kappa \! \in \! (0,\kappa_0)$ and for all $x\! \in \! \overline{B}(0,2r)^c$,
\begin{align*}
\N_x  \big( \, \cR\cap\overline{B}(0,r) \neq\emptyset  \;  ; & \;  \cM \! \left(B(0,2r)\right) \leq \kappa g(r) \,   \big) \\
 \leq &\;  \Cr{cBPnouveau} \, \invDpsi \!\! \left( \Cr{cestppale} r^{-2}\right) \left(r/\lVert x\rVert\right)^{d-2} \left( \log1/r \right)^{-\Cr{cBP2}\, \kappa^{-\Cr{c1BP2} }} \; .
\end{align*}
\end{lem}

\begin{proof} Let $\kappa, \mu \! \in \! (0,\infty)$. Let $r \! \in \! (0, \Cr{restppale})$ where $\Cr{restppale}$ is the constant appearing in Lemma \ref{estppale}. 
Assume that such that $\frac{r^2\mu}{\invpsi(\mu)}\geq 16$.  A Markov inequality, combined with Lemma \ref{BPnouveau} entails :
\begin{align}\label{BP7}
\N_x \! \left(\cR\cap\overline{B}(0,r)\neq\emptyset ; \cM\left(B(0,2r)\right)\leq\kappa g(r)\right) 
&\leq e^{\kappa\mu g(r)}\N_x \big[ \un_{  \{\cR\cap\overline{B}(0,r)\neq\emptyset\}}\, e^{-\mu\cM(B(0,2r))}  \Big] \nonumber\\
&\leq \Cr{cBPnouveau} \left(r/\lVert x\rVert\right)^{d-2}\invDpsi \!\!  \left(\Cr{cestppale}r^{-2}\right) \exp \! \big( F(\mu,r,\kappa) \big) , 
\end{align}
where we have set 
$$F(\mu,r,\kappa):=\kappa\mu g(r)\! -\! \Cr{c1BPnouveau} r\sqrt{\varphi(\mu)} \; .$$ 
For any $q \! \in \! [1, \infty)$ and any $r\! \in (0, r_0)$, we set 
$$ \mu_{r,q}= \invphi \! \big(q ( \frac{_1}{^r} \log \log \frac{_1}{^r})^2 \big).$$
We first get an estimate for $\mu_{q, r}$. To that end fix $a$ such that $1/a \in (0, \delta_\varphi)$. By the definition (\ref{defdelta}) of $\delta_\varphi$, there exists $C\! \in \! (0,\infty)$ such that 
for any $p, \lambda \! \in \! [1, \infty)$, $Cp^{1/a} \varphi(\lambda) \leq \varphi (p\lambda)$. Thus, for any $z \! \geq \! \varphi (1)$ and any $q \! \geq \! C$, 
we get $\varphi^{-1} (qz) \! \leq \! (q/C)^{a} \varphi^{-1} (z)$. This easily entails that there exists 
$\Cl[r]{roupi} \! \in \! (0, r_0\wedge \Cr{restppale})$ and $\Cl[c]{coupi} \! \in \! (0, \infty)$ such that 
\begin{equation}
\label{mumuse}
\forall r \in (0, \Cr{roupi}) ,  \forall q \in [1, \infty), \quad \mu_{r, 1} \leq \mu_{r, q} \leq \Cr{coupi} q^a \mu_{r, 1}\; .  
\end{equation}
Recall that $\widetilde{\psi} (\lambda)= \psi (\lambda) / \lambda $. 
We next observe that for all $r\! \in \! (0, e^{-1})$ and all $q \! \in \! [1, \infty)$, (\ref{convexpsi}) implies 
$$ r^2\,  \frac{\mu_{r, q}}{\psi^{-1} (\mu_{r,q})}  =r^2\,\widetilde{\psi} (\psi^{-1} (\mu_{r,q} )) \geq \frac{_1}{^4} r^2\, \psi^\prime (\psi^{-1} (\mu_{r,q} ))\geq \frac{_1}{^4} r^2\, \psi^\prime (\psi^{-1} (\mu_{r, 1} ))= \frac{_1}{^4} (\log\log 1/r)^2 . $$
Then we set $\Cr{rBP2}= \exp (-e^{8}) \wedge \Cr{roupi}$ and we get  
\begin{equation}
\label{vlan}
\forall r \in (0, \Cr{rBP2}) ,  \forall q \in [1, \infty), \quad \frac{r^2 \mu_{r, q}}{\psi^{-1} (\mu_{r,q})} \geq 16 \; .
\end{equation}
Let $r \! \in \! (0, \Cr{rBP2})$, $q\! \in \! [1, \infty)$ and $\kappa \! \in \! (0, \infty)$. 
Observe that $g(r) = (\log \log 1/r)/ \mu_{r, 1}$ and that $r\sqrt{\varphi (\mu_{r,q})}= \sqrt{q} \log\log 1/r$. Thus, we get 
$$ F(\mu_{r,q}, r, \kappa)=  \Big(\kappa\frac{\mu_{r,q}}{\mu_{r,1}}-\Cr{c1BPnouveau} \sqrt{q}\Big) \log\log \frac{_1}{^r}  \leq  \sqrt{q} \big( \Cr{coupi} \kappa q^{a-\frac{1}{2}} - \Cr{c1BPnouveau}\big) \log\log \frac{_1}{^r}  \; .$$
Since $\delta_\varphi \! \leq \! 1 $, we get $a \! >\! 1$ and thus $a\! -\! \frac{1}{2} \! > \! \frac{1}{2}$. We then set $\kappa_0 \! =\!  \Cr{c1BPnouveau}/(2\Cr{coupi})$ and for all $\kappa \! \in \! (0, \kappa_0)$
we also set $q_\kappa  \! =\!  (\kappa_0/\kappa)^{\frac{1}{a-\frac{1}{2}}}$. Then $q_{\kappa} \! \geq \! 1$ and 
$$ \sqrt{q_\kappa} \big( \Cr{coupi} \kappa q_\kappa^{a-\frac{1}{2}} - \Cr{c1BPnouveau}\big) = -\frac{_1}{^2} \Cr{c1BPnouveau}(\kappa_0/\kappa)^{\frac{1}{2a-1}} \; .$$
We then set $\Cr{cBP2}= \frac{_1}{^2} \Cr{c1BPnouveau} \kappa_0^{\frac{1}{2a-1}} $ and $\Cr{c1BP2}= \frac{1}{2a-1}$. Then 
$$F(\mu_{r, q_\kappa }, r, \kappa) \leq -\Cr{cBP2}\kappa^{-\Cr{c1BP2}} \log\log \frac{_1}{^r} $$
and we complete the proof thanks to (\ref{vlan}) that allows to apply (\ref{BP7}) for any $r\! \in \! (0, \Cr{rBP2})$. \cqfd  
\end{proof}

\section{Proof of the results.}
\label{sectionproofs}
Recall from (\ref{occuWdef}) the definition of the total occupation measure of the snake $\cM$. We prove in Section \ref{locsec} the following results on the lower density of $\cM$. 
\begin{thm}\label{localth}
Let $\psi$ be a branching mechanism the form (\ref{LevyKhin}). Let $(W_t)_{ t\geq 0}$ be the associated snake. Let $g$ be defined by (\ref{gaugedef}). Assume 
that $\bdelta \! >\! 1$ and that $d \! > \! \frac{2\bgamma}{\bgamma-1}$. Then, there exists a constant $\kappa_{d,\psi}\! \in \! (0,\infty)$, that only depends on $d$ and $\psi$,  such that 
\begin{equation}\label{localliminf}
\textrm{$\N_0$-a.e.~for $\cM$-almost all $x$,} \quad \liminf\limits_{r\rightarrow0_+}\frac{\cM(B(x,r))}{g(r)}=\kappa_{d,\psi}.
\end{equation}
\end{thm}
This result is then used to prove the following theorem in Section \ref{proofsnaketh}. 
\begin{thm}
\label{snaketh}Let $\psi$ be a branching mechanism the form (\ref{LevyKhin}). Let $(W_t)_{ t\geq 0}$ be the associated snake. Let $g$ be defined by (\ref{gaugedef}). Assume 
that $\bdelta \! >\! 1$ and that $d \! > \! \frac{2\bgamma}{\bgamma-1}$. Then, there exists a constant $\kappa_{d,\psi}\! \in \! (0,\infty)$, that only depends on $d$ and $\psi$,  such that 
$$\textrm{$\N_0$-a.e.~for any Borel set $B$,} \quad  \cM (B)= \kappa_{d,\psi}  \,  \cP_g ( B \cap \cR \, ) \; .$$
\end{thm}

\subsection{Proof of theorem \ref{localth}.}
\label{locsec}
Recall from Section \ref{levysnake}, the Palm formula for the occupation measure of the snake. To that end,  
recall that $\xi\! =\!  (\xi_t)_{t\geq 0}$ is a continuous process defined on 
the auxiliary measurable space $(\Omega, \cF)$ and recall that  $\P_0$ is a probability measure on $(\Omega, \cF)$ under which $\xi$ is distributed as a standard $d$-dimensional Brownian motion starting from the origin $0$. 
Recall that $(V_t)_{ t\geq 0}$ be a subordinator defined on $(\Omega, \cF, \P_0)$ that is independent of $\xi$ and whose Laplace exponent is $\Dpsistar(\lambda)\! =\! \Dpsi(\lambda) \! -\! \alpha$. Recall from (\ref{defN}) that under $\P_0$, conditionally given $(\xi, V)$, 
$\Nstar(\ddr t \ddr W)=\sum_{j\in\mathcal{J}^*}\delta_{(t_j,W^j)}$ is a Poisson point process on $[0,\infty) \! \times \! C(\R_+,\cW)$ with intensity $\ddr V_t \, \N_{\xi_t} (\ddr W)$. Recall from (\ref{defMa}) that for all $a \! \in \! \R_+$, we have set $\cM_a^* \! = \! \sum_{{j\in\mathcal{J}^*}}\un_{[0,a]}(t_j)\mathcal{M}_j$ where 
for all $j\! \in \! \mathcal{J}^*$, $\cM_j$ stands for the occupation measure of the snake $W^j$ as defined in (\ref{occuWdef}). Also recall from (\ref{defNatheta}) the definition of the following random variables: 
$$\forall \, t \geq s \geq 0  \quad N_r (s, t)=\#\left\lbrace j\in\cJ^* : s\! <\! t_j\! <\! t \;\;   \textrm{and} \; \;   \cR_j\cap\overline{B}(0,r)\neq\emptyset\right\rbrace,$$
that counts the snakes that are grafted on the spatial spine $\xi$ 
between times $s$ and $t$, and that hit the ball $\overline{B}(0, r)$. We first prove the following lemma that is a consequence of the Blumenthal $0$-$1$ law. 
\begin{lem}\label{limdeterm}
Assume that $\bdelta \! >\! 1$ and that $d \! >\! \frac{2\bgamma}{\bgamma-1}$. Recall from (\ref{gaugedef}) the definition of the gauge function $g$. There exists a constant $\kappa_{d,\psi}\in[0,\infty]$ that only depend on $d$ and $\psi$, such that 
\begin{equation}\label{liminfMa}
\forall a\in (0, \infty), \quad \textrm{$\P_0$-a.s.}\quad \liminf\limits_{r\rightarrow 0+}\frac{\Ma(B(0,r))}{g(r)}=\kappa_{d,\psi}.
\end{equation}
\end{lem}
\begin{proof}
Let $a \! \in\! (0,\infty)$. Let $s \! \in\! (0,a)$. Observe that if $N_r(s, a)\! =\! 0$, then $\mathcal{M}_a^*(B(0,r))\! =\! \mathcal{M}_s^*(B(0,r))$. By Lemma \ref{lemmeJars}, there exists 
a deterministic sequence $(r_n)_{n\geq 0}$ decreasing to $0$ such that 
$$\sum_{n\geq 0} \P_0 \left(N_{r_n} (s, a) \neq 0\right)<\infty\; .$$ 
By the Borel Cantelli Lemma, $\P_0$-a.s.~for all sufficiently large $n$, $N_{r_n} (s, a)\! = \! 0$. Since $r\mapsto N_r(s, a)$ is non-decreasing, we get that $\P_0$-a.s.~for all sufficiently small $r$, $N_r(s,a) \! = \! 0$. 
Consequently, 
\begin{equation}\label{cptenzero}
\forall s\in(0,a), \quad \textrm{$\P_0$-a.s.}\quad \liminf\limits_{r\rightarrow 0+}\frac{\Ma(B(0,r))}{g(r)}=\liminf\limits_{r\rightarrow0+}\frac{\mathcal{M}_s^*(B(0,r))}{g(r)} \; .
\end{equation}
Let $\mathcal{G}_s$ be the sigma-field generated by $\un_{[0,s]}(t)\, \Nstar (\ddr t\,  \ddr W) $ and completed by the $\P_0$-negligible sets. 
Using properties of Poisson random measures and the Blumenthal zero-one law for $\xi$, we easily check that $\mathcal{G}_{0+}:=\bigcap_{s>0}\mathcal{G}_s$ 
is $\P_0$-trivial: namely, for all $A\! \in \! \mathcal{G}_{0+}$, either $\P_0 (A)\! =\!  0$ or 
$\P_0(A)\!= \! 1$. Then observe that (\ref{cptenzero}) implies that the random variable $\liminf_{r\rightarrow 0}\Ma(B(0,r))/ g(r)$ is $\mathcal{G}_{0+}$-measurable. It is  therefore $\P_0$-a.s.~ equal to a deterministic constant $\kappa_{d,\psi} \! \in [0, \infty]$ that does not depends on $a$. \cqfd 
\end{proof}
\begin{rem}
We point out that Lemma \ref{limdeterm} holds true for all gauge function $g$. \cq 
\end{rem}
By (\ref{corPalm}), the previous lemma entails that 
\begin{equation} 
\label{Partiel}
\textrm{$\N_0$-a.e.~for $\cM$-almost all $x$,}  \quad  \liminf_{r \rightarrow 0+} \frac{\cM 
(B(x, r))}{g(r)} = \kappa_{d,\psi}  \in [0, \infty] \; .
\end{equation}
Now, we need to prove that $0\!  <\! \kappa_{d,\psi}  \! <\!  \infty$, which is done in two steps.

\begin{lem}
\label{upperbound} Assume that $\bdelta\! >\! 1$ and that $d\! >\! \frac{2\bgamma}{\bgamma-1}$. Then, $\kappa_{d,\psi}<\infty$. 
\end{lem}
\begin{proof}
Fix $a\! \in \! (0,\infty)$ and recall from (\ref{defTc}) the definition of $T_a$ that is the sum of the durations of the snakes that are grafted on the spine $(\xi_s)_{ 0\leq s\leq a}$. 
Recall from (\ref{defNatheta}) the definition of $N_r(s,t)$. Observe that if $\vartheta(2r) \! \geq \! a$, then $\Ma(\Bor) \! \leq \! T_a \! \leq \!  T_{\vartheta(2r)}$. Next not that if 
$ \vartheta(2r)\! <\! a$ and if $N_r\left(\vartheta(2r),a+\vartheta(2r)\right)=0$, then, $\Ma(\Bor)\! \leq \! T_{\vartheta(2r)}$. Therefore, 
\begin{equation*}
\hspace{-1mm}\textrm{$\P_0$-a.s.~on $\{N_r\left(\vartheta(2r),a\! +\! \vartheta(2r)\right)\! =\! 0\}$,}  \quad \Ma(\Bor)\leq T_{\vartheta(2r)}.
\end{equation*}
Then, for all $r\! >\! r^\prime\!  \geq\! 0$, and all $A\! >\! 0$,
\begin{equation}\label{inegMa}
\hspace{-1mm}\textrm{$\P_0$-a.s.~on $\{N_r\left(\vartheta(2r),a\! +\! \vartheta(2r)\right)\! =\! 0\} \! \cap \! \{T_{\vartheta(2r)}\! -\! T_{\vartheta(2r')}\leq A\}$,} \quad \Ma(\Bor)\leq A+T_{\vartheta(2r')}.
\end{equation}
Recall from (\ref{deflastexit}) the definition of $\gamma (r)$ and $\vartheta (r)$ and observe that $\vartheta (r) \leq \gamma (r)$, which implies  
$$\textrm{$\P_0$-a.s.} \quad \forall r\in (0, \infty), \quad T_{\vartheta(r)}\leq T_{\gamma(r)} \; .$$
Thus, by (\ref{inegMa}), for all $r\! >\! r^\prime\!  \geq\! 0$, and all $A\! >\! 0$,
\begin{equation}\label{inegMa2}
\hspace{-1mm}\textrm{$\P_0$-a.s.~on $\{N_r\left(\vartheta(2r),a\! +\! \vartheta(2r)\right)\! =\! 0\} \! \cap \! \{T_{\vartheta(2r)}\! -\! T_{\vartheta(2r')}\leq A\}$,} \quad \Ma(\Bor)\leq A+T_{\gamma (2r')}, \end{equation}
which allows us to bound $\Ma(\Bor)$ by the subordinator $T_{\gamma(\cdot)}$ studied in section \ref{sectionsubord}. 

  Recall Lemma \ref{lesbonsradis}, where the sequence $(\rho_n)_{n\geq 1}$ is introduced; 
recall Lemma \ref{minoproba} and Lemma \ref{lemmeTgamma} that 
provide a control respectively on $\P_0 (N_{\rho_n} (\vartheta(2\rho_n),a+\vartheta(2\rho_n))=0)$ and on $T_{\gamma (2\rho_n)}$. For any $n\! \geq \!1$, we next introduce the following random variable: 
\begin{equation}
Y_n=\un_{\{N_{\rho_n}\left(\vartheta(2\rho_n),a+\vartheta(2\rho_n)\right)=0\}\cap\{T_{\vartheta(2\rho_n)}-T_{\vartheta(2\rho_{n+1})}\leq g(8\rho_n)\}} \; .
\end{equation}
By (\ref{inegMa2}), we get  
$$ \forall n \geq 1, \quad \textrm{$\P_0$-a.s.~on $\{Y_n=1\}$,}\qquad \Ma\left(B(0,\rho_n)\right)\leq g(8\rho_n)+T_{\gamma(2\rho_{n+1})}.$$
We then define the following event: 
$$E=\Big\{\sum_{n\geq 1}Y_n=\infty \Big\}.$$
We claim that 
\begin{equation}\label{pEpositive}
\P_0(E)>0.
\end{equation}
Then the proof of the lemma is completed as follows: the previous arguments first entail 
\begin{equation}\label{liminf}
\textrm{$\P_0$-a.s.~on $E$}, \quad \liminf\limits_{n\rightarrow\infty}\frac{\Ma(B(0,\rho_n))}{g(8\rho_n)}\leq
1+\limsup\limits_{n\rightarrow\infty}\frac{T_{\gamma(2\rho_{n+1})}}{g(8\rho_n)}.
\end{equation}
Recall that the assumption $\bdelta \! >\!1$ implies that $g$ satisfies a $C$-doubling condition (\ref{doubling}), which entails
$g(8r)\! \leq \! C^3 g(r)$. Then by (\ref{liminf}),  
\begin{equation}\label{liminff}
\textrm{$\P_0$-a.s.~on $E$,} \qquad \kappa_{d,\psi}=\liminf\limits_{r\rightarrow 0+}\frac{\Ma(B(0,r))}{g(r)} \leq C^3\left(1+\limsup\limits_{n\rightarrow\infty}\frac{T_{\gamma(2\rho_{n+1})}}{g(8\rho_n)}\right).
\end{equation}
By Lemma \ref{lemmeTgamma} $(ii)$ the right member of (\ref{liminff}) is $\P_0$-a.s.~finite, which completes the proof of the lemma. It only remain to prove our claim (\ref{pEpositive}). 

\medskip

\noi
\textit{Proof of (\ref{pEpositive})}. We use a second moment method. 
By the independence property of Lemma \ref{lemmeTc} $(i)$, we first get 
$$\E_0[Y_n]=\P_0\left( N_{\rho_n}\left(\vartheta(2\rho_n),a\! +\! \vartheta(2\rho_n)\right) \! =\! 0\right)\, \P_0\left(T_{\vartheta(2\rho_n)}\! -\! T_{\vartheta(2\rho_{n+1})} \! \leq\!  g(8\rho_n)\right).$$
Then, the lower bound of Lemma \ref{minoproba} and the fact that $ T_{\vartheta(r)} \! \leq \! T_{\gamma (r)}$ entail 
\begin{eqnarray}\label{minoYn}
\E_0[Y_n]& \geq & \Cr{cminoproba}\P_0\left(T_{\vartheta(2\rho_n)} \!-\! T_{\vartheta(2\rho_{n+1})} \!  \leq  \! g(8\rho_n)\right) \\
&\geq  & \Cr{cminoproba}\P_0\left(T_{\vartheta(2\rho_n)}\leq g(8\rho_n)\right) \nonumber \\
& \geq & \Cr{cminoproba}\P_0\left(T_{\gamma(2\rho_n)}\leq g(8\rho_n)\right). \nonumber 
\end{eqnarray}
So by Lemma \ref{lemmeTgamma} (i), we get 
\begin{equation}\label{KS1}
\sum\limits_{n\geq 1}\E_0[Y_n]=\infty \; .
\end{equation}
Besides, for $n>m\geq 1$, by the independence property of Lemma  \ref{lemmeTc} $(i)$
\begin{eqnarray}\label{KS2}
 & &  \hspace{-12mm}\E_0[Y_nY_m] \nonumber \\
 \hspace{-3mm}&\leq& \hspace{-3mm}\P_0\left(T_{\vartheta(2\rho_n)} \! -\! T_{\vartheta(2\rho_{n+1})} \! \leq \! g(8\rho_n);
T_{\vartheta(2\rho_m)} \! -\! T_{\vartheta(2\rho_{m+1})} \! \leq \! g(8\rho_m); N_{\! \rho_m} \!\! \left(\vartheta(2\rho_m),a \! +\! \vartheta(2\rho_m)\right)\! =\! 0\right) \nonumber \\
&\leq& \hspace{-3mm} \P_0\left(T_{\vartheta(2\rho_n)}\! -\! T_{\vartheta(2\rho_{n+1})} \! \leq \! g(8\rho_n)\right) \, \E_0[Y_m]\nonumber \\
&\leq& \hspace{-3mm} \frac{1}{\Cr{cminoproba}}\E_0[Y_n] \, \E_0[Y_m] \, ,  
\end{eqnarray}
where the last inequality follows from (\ref{minoYn}). 
Therefore, if we denote $L_n=\sum_{1\leq k\leq n} Y_k$, then (\ref{KS1}) and (\ref{KS2}) entail
$$\limsup\limits_{n\to\infty}\frac{\E[L_n^2]}{\E[L_n]^2}<\infty.$$
and the desired result (\ref{pEpositive}) follows from the Kochen Stone Lemma. \cqfd 
\end{proof}

\bigskip

The following lemma completes the proof of Theorem \ref{localth}. Its proof relies on a density result on the $\psi$-L\'evy tree that is proved in \cite{Duq12} and that is recalled 
here as Lemma \ref{lemmetree}. 
\begin{lem}
\label{lowerbound} Assume that $\bdelta \! >\! 1$ and that $d \! >\! \frac{2\bgamma}{\bgamma-1}$. Then $\kappa_{d,\psi}\! >\! 0$. 
\end{lem}
\begin{proof}
We work with $W$ under $\N_0$. The proof consists in lifting to $\widehat{W}$ the estimate of Lemma \ref{lemmetree} by using the fact that conditionally given $H$, $\widehat{W}$ is a Gaussian process (see (\ref{covar}) in Section \ref{levysnake}). More precisely, recall that the height process $H$ is the lifetime process of the snake $W$ and recall 
from (\ref{snex}) the definition of $\N_0$ that shows that conditionally given $H$, the law of $W$ is $Q_0^H$ (see Section \ref{levysnake}). 
Recall from (\ref{distH}) the following notation $d$ for the pseudo-distance in the L\'evy tree:  
$$\forall s, t\in [0, \sigma], \quad d(s,t)= H_t +H_s -2 \!\!\!  \inf_{u\in [s\wedge t , s\vee t]} \!\! H_u \; .$$
Let $r, R \! \in (0, \infty)$ and let $t \! \in \! [0, \sigma]$. We set  
$$ {\bf a}(t,r)= \int_0^\sigma \un_{ \{ d_H (s,t) \leq r \}} \, \ddr s \quad {\rm and} \quad  {\bf b}(t,r, R)= \int_0^\sigma 
\un_{ \{ d_H (s,t) \leq r \} \cap \{  \lVert \widehat{W}_s -\widehat{W}_t \rVert  \geq R \}} \, \ddr s \; .$$
The quantity ${\bf a}(t,r)$ has been already introduced in Lemma \ref{lemmetree}. First note that 
\begin{equation}
\label{ineqdens}
\forall t \in [0, \sigma], \quad  {\bf a}(t,r) \leq {\bf b} (t,r, R) + \cM (B(\widehat{W}_t , R)) \; .
\end{equation}
By (\ref{covar}), we then get 
\begin{equation}
\label{condexpec}
 \textrm{$N(\ddr H)$-a.e.} \; \forall t \in [0, \sigma ] \; , \; \,  Q^H_0 \left[ {\bf  b} (t,r, R) \right] \leq {\bf a}(t,r)  \int_{ \R^d \backslash B(0, R/\sqrt{r})} \! \! \! \! \! \! \! \! \!\! \! \! \! \! \! \! \! \! (2\pi)^{-d/2} e^{-\lVert x \rVert^2/2 } \ddr x \; .
\end{equation}
For any integer $n \! \geq \! 2$, we next set $R_n \! = \! 2^{-n}$ and $r_n\! =\! \frac{1}{4}R_n^2 (\log \log 1/R_n)^{-1}$. By elementary computation,  
$$ \forall n \geq 2, \quad \int_{ \R^d \backslash B(0, R_n/ \sqrt{r_n})} \! \! \! \! \! \! \! \! \!\! \! \! \! \! \! \! \! \! (2\pi)^{-d/2} e^{-\lVert x \rVert^2/2 } \ddr x \; \,  \leq \; \, \Cl[c]{crR} 
n^{-3/2} \; , $$ 
where $\Cr{crR} \! \in \! (0, \infty)$ only depends on $d$ (note that the power $3/2$ is not optimal). By (\ref{condexpec}), we get 
$$ \textrm{$N(\ddr H)$-a.e.} \; \forall t \in [0, \sigma ] \; , \quad  Q^H_0 \left[ \sum_{n\geq 2} \frac{{\bf b}(t,r_n, R_n)}{{\bf a}(t,r_n)} \right] 
 < \infty \; .$$
Thus, $N(\ddr H)$-a.e.~for all $t \in [0, \sigma ]$, $Q^H_0 (\limsup_{n \rightarrow \infty} {\bf b}(t,r_n, R_n)/ {\bf a} (t,r_n) \!  >\!  0) \! =\!  0$. 
Then, by Fubini, 
$$  \textrm{$N(\ddr H)$-a.e.} \; \quad Q^H_0 \left[ \int_0^{\sigma} \un_{ \left\{ \limsup\limits_{n \rightarrow \infty}  \frac{{\bf b}(t,r_n, R_n)}{{\bf a} (t,r_n)} \; \, > \; \, 0  \right\}} \, \ddr t  \right]= 0 \; , $$
which implies that 
\begin{equation}
\label{almostneglig}
 \textrm{$\N_0$-a.e.~for $\ell$-almost all $ t \! \in \! [0, \sigma ]$,}   \;   \quad 
\lim_{n \rightarrow \infty}  \frac{{\bf b}(t,r_n, R_n)}{{\bf a} (t,r_n)}  = 0  
\end{equation}
($\ell$ stands here for the Lebesgue measure on the real line). Recall from (\ref{defk}), the notation $k(r)$: 
$$   \forall r\in(0, \alpha \wedge e^{-e}),\quad k(r):=\frac{\log\log \frac{1}{r}}{\varphi^{-1}\! \big( \frac{1}{r} \log\log \frac{1}{r} \big)}\; .$$
Then, (\ref{almostneglig}) combined with
(\ref{ineqdens}), entails 
$$  \textrm{$\N_0$-a.e.~for $\ell$ almost all $ t \! \in \! [0, \sigma ]$,}   \; \, 
\liminf_{n \rightarrow \infty}  \frac{{\bf a}(t,r_n)}{k(r_n)}  \leq \liminf_{n \rightarrow \infty}  \frac{\cM (B(\widehat{W}_t , R_n))}{k(r_n)}  \; .$$
The definition of $\delta_{\varphi} \! \in\! (0,1)$ easily implies that there exists a constant $\Cl[c]{gizon} \! \in \! (0,\infty)$ such that for all $n$ sufficiently large 
$\Cr{gizon}k(r_n) \! \geq \! g(R_n) $. Then, Lemma \ref{lemmetree} entails that 
$$ \textrm{$\N_0$-a.e.~for almost all $ t \! \in \! [0, \sigma ]$,} \; \,  
\liminf_{n \rightarrow \infty}  \frac{\cM (B(\widehat{W}_t ,  2^{-n}))}{g(2^{-n})} \geq  \frac{\Cr{clemmetree}}{\Cr{gizon}}\; . $$
An easy argument involving the doubling property (\ref{doubling}) for $g$ completes the proof thanks to (\ref{Partiel}).  \cqfd 
\end{proof}

\subsection{Proof of Theorem \ref{snaketh}.}\label{proofsnaketh}

We first introduce a specific decomposition of $\R^d$ into dyadic cubes. We adopt the following notation: we denote by $\lfloor \, \cdot \rfloor$ the integer part application and we write $\log_2$ for the logarithm in base $2$; we fix $d > \frac{2\bgamma}{\bgamma-1}$ and we set 
$$ p := \lfloor \log_2 (4\sqrt{d} ) \rfloor \; , $$ 
so that $2^p > 2\sqrt{d}$. To simplify notation, we set $\cD_n = 2^{-n-p}\bbZ^d $, for any $n \geq 0$. 
For any $y = (y_1, \ldots , y_d) $ in $\cD_n$, we also set 
$$ D_n (y)= \prod_{j= 1}^d \, [ \, y_j -\frac{_1}{^2} 2^{-n} \,  ;  \, y_j+\frac{_1}{^2} 2^{-n} \, ) \;  \quad {\rm and} \quad 
D^\bullet_n (y)= \prod_{j= 1}^d\,  [ \, y_j - \frac{_1}{^2} 2^{-n-p} \, ; \, y_j +\frac{_1}{^2} 2^{-n-p} \, )  .$$
It is easy to check the following properties. 
\begin{itemize}
\item{Prop(1)~} If $y,y'$ are distinct points in $\cD_n$, then $D^\bullet_n (y) \cap D^\bullet_n (y') = \emptyset$. 

\medskip

\item{Prop(2)} Let $y \in \cD_n$. Then, we have 
$$ D^\bullet_n (y) \,  \subset \, \overline{B} ( y \, , \, \frac{_1}{^2} 2^{-n-p}\sqrt{d} \, ) \, \subset  \, 
\overline{B} ( y \, , \, 2^{-n-p}\sqrt{d} \, ) \, \subset \, D_n (y) \; .$$
\end{itemize}
 For any $r \! <\! (2d)^{-1}$, we set $ n(r)= \lfloor \log_2 ( r^{-1} (1+2^{-p})\sqrt{d}  )\rfloor $, so that the following inequalities hold: 
\begin{equation}
\label{coincr}
  \frac{_1}{^2} (1+2^{-p})\sqrt{d} \, 2^{-n(r)} < r \leq (1+2^{-p})\sqrt{d} \, 2^{-n(r)} \; .
\end{equation}
Next, for any $x = (x_1, \ldots , x_d)\in \R^d$ and for $j \in \{ 1, \ldots , d \}$, we set 
$$ y_j = 2^{-n(r)-p} \lfloor  x_j 2^{n(r)+p} +  \frac{_1}{^2} \rfloor  \; .$$
Therefore, $y= (y_1, \ldots , y_d) \in \cD_{n(r)}$ and we easily check the following: 
\begin{itemize}
\item{Prop(3)~} The point $x$ belongs to  $D_{n(r)}^\bullet (y) $ and $D_{n(r)} (y) \,  \subset\,  B(x,r) $. 
\end{itemize}
We work under $\N_0$. Recall Lemma \ref{BP2}: we fix 
\begin{equation}\label{choixkappa1}
\kappa_1 \in (0,\kappa_0)\quad \textrm{such \ that}\quad \Cr{cBP2} \kappa_1^{-\Cr{c1BP2}}>2 \; .
\end{equation}
Since we assume $\bdelta\! >\! 1$, $g$ satisfies the doubling condition (\ref{doubling}), which implies that there exists $\kappa_2 \! \in \! (0, \infty)$, that only depends 
on $d$, $\psi$ and $\kappa_1$, such that for all sufficiently large $n$,
\begin{equation}
\label{tutune}
\kappa_2 g(2^{-n}) \leq \kappa_1 g \big( \frac{_1}{^2} 2^{-n-p} \sqrt{d} \, \big) \; . 
\end{equation}
We then fix $A\! >\! 100$ and  for any $n $ such that $2^{-n} \leq 1/(2A)$, we set 
$$ U_n (A) = \sum_{\substack{ y \in \cD_n \\ 1/A \leq \lVert y \rVert \leq A}} g( \sqrt{d} (1+2^{-p}) 2^{-n} ) \un_{ \{ \,  \cM (D_n(y) ) \leq \kappa_2  g(2^{-n}) \,  \} \cap \{  \, \cR \, \cap\,  D_n^\bullet (y) \neq \emptyset \, \} } \;. $$
We first prove the following lemma. 
\begin{lem}
\label{controlvariable} Assume that $\bdelta \! >\! 1$ (so that (\ref{tutune}) holds) and that $d\! \geq \! 3$. Then, for all $A\! >\! 100$, 
\begin{equation}
\label{Ulimit}
\textrm{$\N_0$-a.e.} \quad   \lim_{N \rightarrow \infty} \sum_{n \geq N} U_n (A) = 0 \; .
 \end{equation}
\end{lem}
\begin{proof} We fix $n$ such that $2^{-n} \leq 1/(2A)$ and we fix $y \in \cD_n $ such that  $1/A\leq  \lVert y \rVert \leq A$. By Prop(2) and (\ref{tutune}), we get 
\begin{eqnarray*}
 \N_0 \left(   \cM (D_n(y) ) \leq \kappa_2 g(2^{-n}) \; ; \;  \cR \cap D_n^\bullet (y) \neq \emptyset  \right)  & &
  \\
 & &   \hspace{-60mm}  \leq \N_0 \left(   \cM (  B ( y \, , \, 2^{-n-p}\sqrt{d} \, ) ) \leq \kappa_1 
 g( \frac{_1}{^2} 2^{-n-p} \sqrt{d}) \; ; \;  \cR \cap \overline{B} ( y \, , \, \frac{_1}{^2} 2^{-n-p}\sqrt{d} \, )   \neq \emptyset  \right) \\
& & \hspace{-60mm}=  \N_{-y} \left(   \cM (  B ( 0 \, , \, 2^{-n-p}\sqrt{d} \, ) ) \leq \kappa_1 
 g( \frac{_1}{^2} 2^{-n-p} \sqrt{d}) \; ; \;  \cR \cap \overline{B} ( 0 \, , \, \frac{_1}{^2} 2^{-n-p}\sqrt{d} \, )   \neq \emptyset  \right) , 
\end{eqnarray*}
the last equality being an immediate consequence of the invraince of the snake by translation. 
We next apply Lemma \ref{BP2} with $x\! =\!  -y$ and $r\! =\!  \frac{_1}{^2} 2^{-n-p}\sqrt{d}$ and $\kappa\! = \! \kappa_1$ (that satisfies (\ref{choixkappa1}); thus, there exists 
$\Cl[c]{atsoa} , \Cl[c]{astoa} \! \in \! (0, \infty)$, that only depends on $d$ and $\psi$, such that 
$$   \N_0 \left( \,   \cM (D_n(y) ) \leq \kappa_2 g(2^{-n}) \; ; \;  \cR \cap D_n^\bullet (y) \neq \emptyset \,   \right)  
\leq \,\Cr{atsoa} (2^{-n-p})^{d-2} \lVert y \rVert^{2-d} n^{-2}  \, \invDpsi\! \! \left(\Cr{astoa} 2^{2n+2p}\right) \; .$$
By Lemma \ref{convexBP} and the doubling property (\ref{doubling}) of $g$, there exists $\Cl[c]{mickey} \! \in \! (0, \infty) $, that only depends on $d$ and $\psi$, 
such that for all sufficiently large $n$, 
$$ g \big( \sqrt{d} (1+2^{-p})  2^{-n} \big) \, \invDpsi\! \! \left(\Cr{astoa} 2^{2n+2p}\right) \leq \Cr{mickey} 2^{-2n-2p}  \; , $$
which entails the following:
\begin{equation}
g( \sqrt{d} (1+2^{-p})2^{-n} )  \N_0 \left(   \cM (D_n(y) ) \! \leq \! \kappa_2 g(2^{-n})  ;   \cR \cap D_n^\bullet (y) \! \neq \! \emptyset   \right)   \leq  \Cl[c]{donald} (2^{-n-p})^{d} \lVert y \rVert^{2-d} n^{-2} \; , 
\end{equation} 
where $\Cr{donald}\! =\!  \Cr{atsoa}\Cr{mickey}$. Elementary arguments entail the following inequalities: 
\begin{eqnarray*}
 \N_0 \left(  U_n (A) \right)& \leq & \Cr{donald} \, n^{-2} \!\!\!\!\!\!\! \sum_{\substack{ y \in \cD_n \\ 1/A \leq \lVert y \rVert \leq A} }\!\! \!\!  (2^{-n-p})^{d} \lVert y \rVert^{2-d} \\
 & \leq & \Cl[c]{carlos}\,  n^{-2} \!\! \int \un_{ \{1/A\leq  \lVert x \rVert \leq A \}}  \lVert x \rVert^{2-d} dx \\
 & \leq & \Cl[c]{regine}\,  n^{-2} \! \int_{1/A}^A \!\!\! \rho \, d\rho \\
& \leq &  \Cl[c]{gudrun} \, A^2 n^{-2} , 
\end{eqnarray*}
where $\Cr{carlos}, \Cr{regine}, \Cr{gudrun}\!  \in \! (0, \infty) $ only depend on $d$ and $\psi$. Therefore, 
$\N_0 ( \sum_{n \geq N} U_n (A) ) \! <\!  \infty $, which easily implies  the lemma.  \cqfd 
\end{proof}

\bigskip

We next prove the following lemma. 
\begin{lem}
\label{badpoint} Assume that $\bdelta>1$ and that  $d > \frac{2\bgamma}{\bgamma-1}$. Then, 
\begin{equation}
\label{badzero}
\textrm{$\N_0$-a.e.} \quad  \cP_g \left(  \left\{ x \in \cR \; :\; \liminf_{r \rightarrow 0+} g(r)^{-1} \cM (B(x,r)) \, \neq \kappa_{d,\psi}  \, \right\} \right) = 0 \; ,
\end{equation}
where $\kappa_{d, \psi}$ is the constant appearing in Theorem \ref{localth}. 
\end{lem}

\begin{proof} We fix $A \! >\! 100$. By Theorem \ref{localth} and Lemma \ref{controlvariable} there exists a Borel subset $\cW_A$ of $\cW$ such that $\N_0 (\cW \backslash \cW_A)= 0$ and such that on 
$\cW_A$, (\ref{localliminf}) and (\ref{Ulimit}) hold true. We shall work deterministically on  $W \in \cW_A$.

   Let $B $ be any  Borel subset of $\{ x \! \in\!  \R^d  :  1/A \!  \leq \! \lVert x \rVert \!   \leq \! A  \} $. 
Let $\varepsilon\in(0,\infty)$ and let $\overline{B} (x_1, r_1)$, $\ldots$ $\overline{B} (x_k, r_k)$ be any closed $\varepsilon$-packing of $B \cap \cR$. Namely, the later balls are disjoint, their centres belong to $B\cap \cR$ and their radii are smaller than $\varepsilon$. Let $\Cl[c]{zarai} \! \in \! (0,\infty) $ be a constant to be specified later. 
First observe that 
\begin{eqnarray}
\label{splitineq}
\sum_{i=1}^k  g(r_i)  &= & \sum_{i=1}^k  g(r_i)  \un_{ \{  \cM ( B( x_i, r_i) >\Cr{zarai} \, g(r_i)  \}} +  \sum_{i=1}^k  g(r_i)  \un_{ \{  \cM ( B( x_i, r_i) \leq \Cr{zarai} \, g(r_i)  \}} \nonumber \\
& \leq & \Cr{zarai}^{-1} \cM \big( B^{(\varepsilon)} \big) +  \sum_{i=1}^k  g(r_i)  \un_{ \{  \cM ( B( x_i, r_i) \leq \Cr{zarai} \,  g(r_i)  \}} \; , 
\end{eqnarray}
where we have set $B^{(\varepsilon)}= \{ x \in \R^d  :  {\rm dist} (x, B) \leq \varepsilon \} $. 
Next, fix $1 \leq i \leq k$; recall notation $n(r_i)$ from (\ref{coincr}) and denote by $y_i$ the point of $\cD_{n(r_i)}$ corresponding to $x_i$ such that Prop(3) holds true. Therefore, by (\ref{coincr}), we have 
\begin{eqnarray*}
 \cM ( B( x_i, r_i)\, )  \leq \Cr{zarai}\, g(r_i) \quad {\rm and} \quad  \;   x_i \in B \cap \cR \quad  \Longrightarrow  & & \\
& & \hspace{-70mm}  \cM (D_{n(r_i)} (y_i) )\leq \Cr{zarai}\, g( (1+2^{-p})\sqrt{d} 2^{-n(r_i) }) \quad {\rm and}  \quad \cR \cap D^\bullet_{n(r_i)} (y_i) \neq \emptyset .
\end{eqnarray*}
We use the doubling condition (\ref{doubling}) to choose $\Cr{zarai}$ such that 
$\Cr{zarai} g( (1+2^{-p}) \sqrt{d} 2^{-n(r) }) \! \leq \! \kappa_2 g( 2^{-n(r) })$ 
for all sufficiently small $r\in (0, 1)$. Thus, we get 
$$ \sum_{i=1}^k  g(r_i)  \un_{ \{  \cM ( B( x_i, r_i) \leq \Cr{zarai} g(r_i)  \}} \;  \leq \sum_{n\, : \,  2^{-n}  \leq \Cl[c]{cordy}\, \varepsilon } U_n (A) \; , $$
where $\Cr{cordy} \! =\!  2((1+2^{-p}) \sqrt{d})^{-1}$. Since $W$ 
belongs to $\cW_A$ where (\ref{Ulimit}) holds, this inequality combined with (\ref{splitineq})  implies the following. 
\begin{equation}
\label{abscont}
\cP_g \left( B \cap \cR\right) \leq \cP^*_g \left( B \cap \cR\right) \leq \Cr{zarai}^{-1}\,  \cM \Big(  \bigcap_{\varepsilon >0}B^{(\varepsilon)} \Big) \; .
\end{equation}
We next applies (\ref{abscont}) with  $B=B_A$ given by 
$$B_A = \left\{ x \!  \in \! \cR \; : \;  1/A \! \leq \! \lVert x \rVert \! \leq \! A \; \,  {\rm and} \; \, \liminf_{r \rightarrow 0+} g(r)^{-1} \cM (B(x,r)) \neq \kappa_{d,\psi}    \right\} \; .$$
Therefore $\cP_g (B_A)\! < \!\infty$. 
Suppose now that $\cP_g(B_A) \! >\! 0$. Then, as a consequence of (\ref{innerreg}), there exists a closed subset 
$F$, with $F \! \subset \! B_A$, such that 
$\cP_g (F) \! >\!  0$. Since $F$ is closed then $F=\bigcap_{\varepsilon >0} F^{(\varepsilon)}$; since 
$F$ is a subset of $B_A$ and since $W \! \in \! \cW_A$ (where (\ref{localliminf}) holds true), we get $\cM (F) \! \leq \! \cM (B_A) \! =\! 0$ and 
by  (\ref{abscont}) applied to $B\! =\! F$, 
we obtain $\cP_g (F) \! =\! 0$, which rises a contradiction. Thus, we have proved that $\N_0$-a.e.$ \cP_g \left( B_A \right) \! = \! 0 $, 
which easily entails the lemma by letting $A$ go to $\infty$, since $\cP_g (\{ 0 \} ) \! =\!  0$.  \cqfd 
\end{proof}

\bigskip

We now complete the proof of Theorem \ref{snaketh}: by Theorem \ref{localth} and Lemma \ref{badpoint} there exists a Borel subset $\cW^*$ of $\cW$ such that 
$\N_0 (\cW \backslash \cW^*)\! =\! 0$ and such that (\ref{localliminf}) and (\ref{badzero}) hold true on $\cW^*$. We fix $W \in \cW^*$ and we set 
$$ \mathtt{Good} = \left\{  x \in \cR \; : \;   \liminf_{r \rightarrow 0+} g(r)^{-1} \cM (B(x,r)) = \kappa_{d,\psi}   \right\} \quad {\rm and } \quad \mathtt{Bad}= \cR \backslash \mathtt{Good}. $$
Let $B$ be any Borel subset of $\R^d$. By (\ref{localliminf}) and (\ref{badzero}), we have 
$$ \cM ( B\cap \mathtt{Bad}) =\cP_g (B\cap \cR \cap \mathtt{Bad}) = 0 \; .$$
Then, we apply Lemma \ref{equalitycomp} to $\mathtt{Good}\cap B$ and we get 
$$ \cM ( B \cap \mathtt{Good})= \kappa_{d,\psi}  \, \cP_g ( B \cap \cR \cap\mathtt{Good} ) \; .  $$
Therefore, on $\cW^*$, for Borel subset $B$ of $\R^d$, $\cM (B)= \kappa_{d,\psi} \, \cP_g (B \cap \cR)$, which completes the proof of Theorem \ref{snaketh}.

\subsection{Proof of theorem \ref{mainth}}
\label{secmainth}
We derive Theorem \ref{mainth} from Theorem \ref{snaketh}.  To that end, we first need an upper bound of the upper box-counting dimension of $\cR$ under $\N_x$. Let us briefly recall the definition of the box-counting dimensions of a bounded subset $K\subset\R^d$: let $\veps\in(0,\infty)$ and let $n_\veps(K)$ stands for the minimal number of open balls of radius $\veps$ that are necessary to cover $A$. Then, 
\begin{equation}
\label{defdimboiteintro}
\underline{\dim}(K)=\liminf_{\veps\to 0}\frac{\log n_\veps(K)}{\log1/\veps}\quad \textrm{and} \quad \overline{\dim}(K)=\limsup_{\veps\to 0}\frac{\log n_\veps (K)}{\log1/\veps}.
\end{equation} 
Fix $x \! \in \! \R^d$ and recall Lemma \ref{snakhold} that asserts that 
for any $q \! \in \! (0, \frac{\bgamma-1}{2\bgamma})$, $\N_x$-a.e.~$(\widehat{W}_s)_{ s\in [0, \sigma]}$ is $q$-H\"older continuous. As already mentioned in Comment \ref{casOK}, it easily implies the following: 
\begin{equation} 
\label{upperboxrange}
\textrm{$\N_x$-a.e.} \quad \overline{{\rm dim}} (\cR) \leq \frac{2\bgamma}{\bgamma-1} \; , 
\end{equation}
where $\cR$ is the range of the Lévy snake as defined (\ref{rangeWdef}). We next prove the following lemma. 
\begin{lem}
\label{capacity} Assume that $\bgamma \! >\! 1$ and that $d \! > \! \frac{2\bgamma}{\bgamma-1}$. Let $x \! \in \! \R^d$. For any compact subset $K$ such that 
$\overline{{\rm dim}}(K) \! \leq \! \frac{2\bgamma}{\bgamma-1}$, we have $\N_x $-a.e.~$\cM (K)\! =\!  0$. 
\end{lem}

\begin{proof} Let us first assume that $x \notin K$ and set $k\! :=\!  \inf_{y \in K} \lVert x-y \rVert \! >\!  0$. For any 
$\varepsilon  \in (0, k/2)$, denote by $n_\varepsilon$ the minimal number of balls with radius $\veps$ that are necessary to cover $K$, and denote by $B(x^\varepsilon_1 , \varepsilon )$, ..., $B( x^\varepsilon_{n_\varepsilon} , \varepsilon)$ such balls. 
Then, (\ref{invar-snake}) combined with standard estimates of $d$-dimensional Green function entail the following inequalities.
 \begin{eqnarray*}
 \N_x \left( \cM (K)\right) & \leq &  \sum_{i=1}^{n_\varepsilon} \N_x \left( \cM(    B (x^\varepsilon_i , \varepsilon) ) \right) \\
 & \leq &  \sum_{i=1}^{n_\varepsilon}  \int_0^\infty \!\!\! e^{-\alpha a} \, \P_x  \! \left( \,  \xi_a \! \in\!  B (x^\varepsilon_i , \varepsilon ) 
     \right)  \, \ddr a \\
 & \leq & \Cl[c]{stone} \,  \sum_{i=1}^{n_\varepsilon} 
 \int_{B (x^\varepsilon_i , \varepsilon )}\!\!\!  \!\!\!\! \! \! \lVert x- y \rVert^{2-d}\,  \ddr y \\
 & \leq & \Cl[c]{charden} \, k^{2-d} \,  \varepsilon^{d}   n_\varepsilon ,  
\end{eqnarray*}
 where $\Cr{stone}, \Cr{charden} \! \in \! (0, \infty)$ only depend on $d$. Since $d \!  > \! \frac{2\bgamma}{\bgamma-1} \!  \geq  \! \overline{{\rm dim}} (K)$, the previous inequality implies that $\N_x (\cM (K))= 0$ as $\veps \rightarrow 0$. 
 
 Let us now consider the general  case: for any $r \!>\! 0$, the previous case applies to the compact set $K'= K \backslash B(x,r)$ and we get
$$ \textrm{$\N_x$-a.e.} \quad 
\cM (K) = \cM ( K \cap B(x, r)) + \cM (K \backslash B(x,r)) \leq \cM (B(x, r)) \; , $$ 
 which implies the desired result as $r\rightarrow 0$ since $\cM$ is diffuse.  \cqfd 
 \end{proof}
 
 \bigskip
 
The end of the proof of Theorem \ref{mainth} follows an argument due to 
Le Gall in \cite{LG99} pp.~312-313. Theorem \ref{snaketh}  and Lemma \ref{capacity} imply that for any compact set $K$ such that $\overline{{\rm dim}} (K) \leq \frac{2\bgamma}{\bgamma-1} $, and for any $x \in \R^d$,
 \begin{equation}
 \label{packingpolar}
\textrm{$\N_x$-a.e.} \quad  \cP_{g} (K \cap \cR ) = 0 \; .
 \end{equation}
 Recall the connection (\ref{connection}) in Theorem \ref{super} between ${\bf R}$, ${\bf M}$ and the excursions $W^j$, $j \in \cJ$, of the Brownian snake. An easy argument on Poisson point processes combined with (\ref{upperboxrange}) and  (\ref{packingpolar}) implies that almost surely $\cP_{g} \left( \cR_{W^{j}} \cap  \cR_{W^{i}}\right) =0$ for any $i\neq j $ in $\cJ$. Then, 
(\ref{connection}) entails  
$$ \cP_{g} \left( \, \cdot \cap {\bf R} \, \right) = \sum_{j \in \cJ } \cP_{g} \left( \, \cdot \cap \cR_{W^j}   \right) \; .$$
Theorem \ref{super} and (\ref{connection}) thus imply
$$ \kappa_{d,\psi} \, \cP_{g} \left( \, \cdot \cap {\bf R}  \right)  =  \sum_{j \in \cJ } \kappa_{d,\psi}  \,  \cP_{g} \left(\,  \cdot \cap \cR_{ W^{j} } \right)  =\sum_{j \in \cJ }  \cM_{W^{j} } = {\bf M}\; , $$
which is the desired result.

\subsection{Dimension of the range of the $\psi$-SBM.}

  We now prove Theorem \ref{thdimension}. 
To that end, recall that $\xi \! = \! (\xi_t)_{t\geq 0}$ is a continuous process defined on 
the auxiliary measurable space $(\Omega, \cF)$ and recall that  $\P_0$ is a probability measure on $(\Omega, \cF)$ under which $\xi$ is distributed as a standard $d$-dimensional Brownian motion starting from the origin $0$. 
Recall that $(V_t)_{ t\geq 0}$ be a subordinator defined on $(\Omega, \cF, \P_0)$ that is independent of $\xi$ and whose Laplace exponent is $\Dpsistar(\lambda)\! =\! \Dpsi(\lambda) \! -\! \alpha$. Recall from (\ref{defN}) that under $\P_0$, conditionally given $(\xi, V)$, 
$\Nstar(\ddr t \ddr W)=\sum_{j\in\mathcal{J}^*}\delta_{(t_j,W^j)}$ is a Poisson point process on $[0,\infty) \! \times \! C(\R_+,\cW)$ with intensity $\ddr V_t \, \N_{\xi_t} (\ddr W)$. Then recall from (\ref{defMa}) that for all $a \! \in \! \R_+$, we have set $\cM_a^* \! = \! \sum_{{j\in\mathcal{J}^*}}\un_{[0,a]}(t_j)\mathcal{M}_j$ where 
for all $j\! \in \! \mathcal{J}^*$, $\cM_j$ stands for the occupation measure of the snake $W^j$ as defined in (\ref{occuWdef}). Also recall from (\ref{defNatheta}) the definition of the following random variables: 
$$\forall \, t \geq s \geq 0  \quad N_r (s, t)=\#\left\lbrace j\in\cJ^* : s<t_j<t \;  \textrm{and} \,  \cR_j\cap\overline{B}(0,r)\neq\emptyset\right\rbrace,$$
that counts the snakes that are grafted on the spatial spine $\xi$ 
between times $s$ and $t$, and that hit the ball $\overline{B}(0, r)$. 
\begin{lem}\label{Maru}
Assume that $\bdelta>1$ and that $d>\frac{{2\bdelta}}{{\bdelta-1}}$. Then, for all $a \! \in\! (0,\infty)$, and for all $u \! \in \!  (0, \frac{{2\bgamma}}{{\bgamma-1}})$,
\begin{equation}\label{limMa}
\textrm{$\P_0$-a.s.}\quad \liminf\limits_{r\rightarrow0+}\,  r^{-u}\Ma(B(0,r))\, <\infty.
\end{equation}
\end{lem}

\begin{proof} 
The present proof is very similar to that of Lemma \ref{minoprobadelta}. We detail only the main steps. 
Recall from Lemma \ref{lemmeTc} the definition of the processes $(T_{\vartheta(r)})_ {r\geq 0}$ and $(T_{\gamma(r)})_{ r\geq 0}$. Recall from (\ref{defNatheta}) 
the definition of $N_r(s,t)$. Recall from Lemma \ref{lemmeTgammabis} the definition of 
the sequence $(s_n)_{n\geq 0}$. Then, 
for all $n\! \geq \! 0$, we set  
\begin{equation}
Y^\prime_n=\un_{\{N_{s_n}(\vartheta(2 s_n),a+\vartheta(2 s_n))=0\}\cap\{T_{\vartheta(2s_n)}-T_{\vartheta(2s_{n+1})}\leq s_n^u\}}.
\end{equation}
Reasoning as in the proof of (\ref{inegMa2}), we get 
\begin{equation}\label{inegMa3}
\textrm{$\P_0$-a.s.~on $\{ Y^\prime_n = 1\}$}, \quad \Ma(B(0,s_n))\leq 1+T_{\gamma(2s_{n+1})}.
\end{equation}
Thus, if we show that 
\begin{equation}\label{ppositivedim}
\P_0\Big( \sum_{n\geq 1}Y^\prime_n=\infty\Big)>0,
\end{equation}
we get (\ref{limMa}), by use of Lemma \ref{lemmeTgammabis} $(ii)$.

  The inequality (\ref{ppositivedim}) is obtained using Kochen Stone Lemma, as in the proof of (\ref{pEpositive}). Indeed, under the assumption $d>\frac{2\bdelta}{\bdelta-1}$, Lemma \ref{minoprobadelta} entails that  for all $n\in\N$, $\P_0(N_{s_n}(\vartheta(2 s_n),a + \vartheta(2 s_n)) \! =\! 0) \! \geq\! \Cr{cminoprobadelta} \! >\! 0$; we then argue exactly as in the proof of Lemma \ref{upperbound} to obtain (\ref{ppositivedim}).  
We leave the details to the reader. \cqfd 
\end{proof}

\bigskip

\noi
\textbf{Proof of Theorem \ref{thdimension}.} Assume that $\bdelta \! >\! 1$ and that $d\! > \! \frac{2\bdelta}{\bdelta -1}$. Recall from (\ref{rangeWdef}) the definition of $\cR$, the range of the Lévy snake. By Theorem \ref{super} and by spatial invariance of the snake, it is sufficient to prove 
$$ \textrm{$\N_0$-a.e.} \quad \dim_p(\cR)=  \overline{{\rm dim}}(\cR) = \frac{2\bgamma}{\bgamma-1} \; .$$
Recall that for every bounded subset $K \! \subset\! \R^d$, $\dim_p(A)\! \leq \! \overline{\dim}(A)$ 
(see e.g.~Falconer \cite{Falbook}). By (\ref{upperboxrange}), it then only remains to prove 
\begin{equation}\label{minodimP}
\textrm{$\N_0$-a.e.} \quad \dim_p(\cR)\geq \frac{2\bgamma}{\bgamma-1} \; .
\end{equation}
Lemma \ref{Maru} comnied with (\ref{corPalm}) implies that for all $u\! \in \! (0,  \frac{2\bgamma}{\bgamma-1})$, 
\begin{equation}\label{MWpp}
\textrm{$\N_0$-a.e.~for $\cM$-almost all $x$},\quad \liminf\limits_{r\rightarrow 0 +} \, r^{-u}\cM(B(x,r))\, <\infty, 
\end{equation} 
which implies that $\N_0$-a.e.~$\dim_p(\cR)\! \geq \! u$ by the comparison results stated here as Theorem \ref{TaTrcomparesu}. This entails (\ref{minodimP}) and the proof of Theorem \ref{thdimension} is completed. \cqfd

%
%
%

\end{document}